\documentclass[10pt]{amsart}
\usepackage{graphicx,amssymb,amsfonts,amsmath,amsthm,newlfont}
\usepackage{epsfig,url}
\usepackage{color}
\usepackage{enumitem}

\usepackage[all,2cell]{xy} \UseAllTwocells \SilentMatrices

\newtheorem{thm}{Theorem}[section]
\newtheorem{cor}[thm]{Corollary}
\newtheorem{lem}[thm]{Lemma}
\newtheorem{prop}[thm]{Proposition}
\theoremstyle{definition}
\newtheorem{defn}[thm]{Definition}
 
\newtheorem{question}{Question} 

\newtheorem{ex}[thm]{Examples}
\newtheorem{example}[thm]{Example}
\newtheorem{warning}[thm]{Caveat}

\theoremstyle{remark}
\newtheorem{rem}[thm]{Remark}
\numberwithin{equation}{section}

\newcommand{\Z}{\mathbb Z}
\newcommand{\C}{\mathbb C}

\newcommand{\R}{\mathbb R}

\newcommand{\Pro}{\mathbb P}

\newcommand{\gr}{\mathrm{gr}}

\font \rus= wncyr10
\newcommand{\sha}{\, \hbox{\rus x} \,}

\newcommand{\MT}{\mathcal{MT}}

\newcommand{\alg}{\mathrm{alg}}

\newcommand{\Isom}{\mathrm{Isom}}

\newcommand{\zetam}{\zeta^{ \mathfrak{m}}}

\newcommand{\Q}{\mathbb Q}
\newcommand{\Li}{\mathrm{Li}}

\newcommand{\U}{\mathcal{U}}

\newcommand{\To}{\longrightarrow}

\newcommand{\A}{\mathbb{A}}
\newcommand{\G}{\mathbb{G}}

\newcommand{\tone}{\overset{\rightarrow}{1}\!}

\newcommand{\opo}{{}_0 \Pi_{0}}

\newcommand{\Or}{\mathcal{O}}

\newcommand{\mm}{\mathfrak{m} }

\newcommand{\id}{\mathrm{id} }

\newcommand{\Lie}{\mathrm{Lie}\,}

\newcommand{\per}{\mathrm{per}}

\newcommand{\uu}{\mathfrak{u}}

\newcommand{\Spec}{\mathrm{Spec} \,}

\newcommand{\Pe}{\mathcal{P}}

\newcommand{\comp}{\mathrm{comp}}

\newcommand{\dr}{\omega}

\newcommand{\xpio}{{}_x\Pi^{\dr}_0}
\newcommand{\opio}{{}_0\Pi^{\dr}_0}
\newcommand{\xco}{{}_xc_0}
\newcommand{\xgo}{{}_x\gamma_0}

\newcommand{\co}{\mathrm{co}}
\newcommand{\ev}{\mathrm{ev}}

\newcommand{\GdrS}{G^{\dr}_{S}}
\newcommand{\UdrS}{U^{\dr}_{S}}

\newcommand{\PF}{\underline{\Pi}}

\begin{document}

\author{Francis Brown}
\begin{title}[Integral points  on curves and motivic periods]{Integral points  on curves, the unit equation,  and motivic periods}\end{title}
\maketitle

 \section{Introduction}

This paper recasts some of  the literature which has grown out of Kim's extension of Chabauty's method for bounding points on curves in the language of periods. We retrieve many  of the main results of the papers \cite{Kim1},  \cite{KimAlb}, \cite{Hadian}, \cite{DanCWewers} as a simple consequence of two  constructions: the universal comparison isomorphism applied to the unipotent fundamental group, and an elementary lemma in linear algebra.
We also show how the method can be made motivic, explicit, and extended to  divisors.  
 
The main idea can  be summarised as follows.  Suppose, for simplicity, that $X$ is a smooth affine scheme over $\Q$. Integration provides a natural comparison isomorphism between  its   algebraic de Rham and Betti cohomology \begin{eqnarray} H^n_{dR}(X;\Q) \otimes_{\Q} \C  & \overset{\sim}{\To} & H^n(X(\C))\otimes_{\Q} \C  \label{introcomp}  \\
  { [}\omega] & \mapsto & ([\gamma] \mapsto \int_{\gamma} \omega)  \nonumber 
\end{eqnarray} 
where $\omega$ is a regular differential  $n$-form on $X$ over $\Q$, and $\gamma$ a smooth topological chain on $X(\C)$  of real dimension $n$.  This isomorphism  is transcendental, but can be made  algebraic by replacing the ring $\C$  with a suitable ring of `motivic' periods $\Pe_{C}^{\mm}$, where $C$ is some version of   a Tannakian category of motives. For  its definition, and some candidates for $C$, see \S\ref{sectCperiods}. One deduces a universal comparison isomorphism
\begin{eqnarray} H^n_{dR}(X;\Q) \otimes_{\Q} \Pe_C^{\mm}  & \overset{\sim}{\To} & H^n(X(\C))\otimes_{\Q} \Pe_C^{\mm}  \label{introPcomp}   \\
   {[}\omega] & \mapsto & ([\gamma] \mapsto I^{\mm}_{\gamma}(\omega))  \nonumber 
\end{eqnarray} 
   where $I^{\mm}_{\gamma}(\omega)$ is a `motivic' version of the integral of $\omega$ along $\gamma$.  The ring $\Pe^{\mm}_C$ is a $\Q$-algebra equipped with a period homomorphism 
   $$\mathrm{per} : \Pe^{\mm}_C \To \C$$
   satisfying $\mathrm{per}  (I^{\mm}_{\gamma}(\omega)) = \int_{\gamma}\omega$, so $(\ref{introcomp})$ can be retrieved 
   from $(\ref{introPcomp})$ by applying the map  $\per$.
     The point is that the isomorphism $(\ref{introPcomp})$
   is  defined over $\Q$. Furthermore,  standard conjectures about mixed motives suggest that the subspace of $\Pe^{\mm}_C$ generated by the motivic periods of  $X$ should be  tightly constrained.  This implies  relations between the $I^{\mm}_{\gamma}(\omega)$ for varying $\gamma$ and $\omega$, and 
      can be used to infer arithmetic information about integral points on $X$ or varieties related to $X$. 
   
   Since the general  theory of motives is largely conjectural, we have to make do with  mixed Tate motives over  number fields, for which 
    the required bounds on motivic periods  are known. This is a consequence of Borel's deep theorems on the algebraic $K$-theory of number fields.
  As a consequence,  we mainly consider the case when $X$ is the  punctured projective line. We now state some results that can be obtained in this setting before returning to a more general situation in \S\ref{sectintrogen}.

\subsection{Main example: the unit equation}  Let $X = \Pro^1\backslash \{0,1, \infty\}$ and let $S$ be a finite set of rational primes. 
Let $X_S: = X(\Z[S^{-1}])$ denote the set of $S$-integral points on $X$.
It is the set of solutions to the unit equation 
$$u+ v = 1\ ,$$
where $u, v$ are rational numbers whose numerator and denominator have prime factors contained in $S$. 
This set  is  finite, as shown by   Siegel \cite{Siegel} and reproved by Kim \cite{Kim1} using the unipotent fundamental group.  The problem of  understanding $S$-units is related to a number of questions in diophantine geometry.

Let $\MT_S$ denote the category of mixed Tate motives ramified only at $S$. It is a full subcategory of the Tannakian category of mixed Tate motives over $\Q$ \cite{DeGo}.

The classical polylogarithms can be expressed  iterated integrals on $X(\C)$:
$$\Li_n(x) = \sum_{k\geq 1 } {x^k  \over k^n} = \int_{0}^x {dz \over  1-z} {dz \over z} \ldots {dz\over z} \ .$$
For $x \in X_S$, there exists a motivic version $\Li_n^{\mm}(x)$, which is an  element of the ring of effective motivic periods $\Pe^{\mm,+}_{\MT_S}$ of the category $\MT_S$,   satisfying
$\mathrm{per} ( \Li_n^{\mm}(x))= \Li_n(x)$. The ring $\Pe^{\mm,+}_{\MT_S}$ is a graded $\Q$-vector space, finite dimensional in every weight.  It follows that if there are  many integral points $x_1,\ldots, x_N \in X_S$, there must exist  $\Q$-linear relations between the $\Li_n^{\mm}(x_i)$, and hence their periods $\Li_n(x_i)$. This imposes a  constraint on the $x_i$.  One can replace the classical polylogarithms with  multiple polylogarithms or indeed any iterated integrals on $X$. The entirety of this paper reduces to this simple idea. Technicalities arise only if one wants to make this constraint explicit, or if one wishes to generalise to the non mixed-Tate case \S\ref{sectintrogen}.

\subsubsection{Effectivity} The problem is that the coefficients in such a relation will depend on the points $x_i$, so this method  is not   effective.  In order to obtain a relation which is universal for all points, one must replace $\Pe^{\mm,+}_{\MT_S}$ with  a graded ring $\Pe^{\uu}_{\MT_S}$ of unipotent de Rham periods --- a key technical point being that the unipotent version of $2  \pi i $ in the latter ring is trivial. This is equivalent to reducing `modulo $\pi$' and increases the number of available relations.
The ring $\Pe^{\uu}_{\MT_S}$ is simply  the affine ring $\Or(U_{\MT_S})$ of the unipotent radical of the Tannaka group of $\MT_S$ with respect to the de Rham fiber functor.  The Frobenius at $p$ endows it with a $p$-adic period map 
$$\mathrm{per}_p : \Pe^{\uu}_{\MT_S} \To \Q_p $$
for all $p \notin S$, and contains  objects $\Li_n^{\uu}(x_i)$ of degree $n$, whose periods are  `single-valued' $p$-adic versions of the classical polylogarithms. To gain some intuition for these objects, observe that  $\Li_1^{\uu}(x) = - \log^{\uu}(1-x)$, and that  the unipotent logarithm satisfies the usual functional equation: for any $x\in X(\Q)$, 
$$\log^{\uu}(x) = \sum_p v_p(x)  \log^{\uu}(p)\ ,$$
where the sum is over all primes $p>0$, and $v_p$ denotes the $p$-adic valuation. The elements $\log^{\uu}(p)$ are algebraically 
independent over $\Q$ \cite{NotesMot}. The higher unipotent polylogarithms  generalize   $p$-adic valuations of points $x \in X_S$.

 By applying an elementary lemma in linear algebra \S\ref{sect: co}  to the ring $\Pe^{\uu}_{\MT_S}$  and computing its dimensions in each degree, we  deduce the following theorem.
 
\begin{thm} \label{introthm} Let $s= |S|$ and let $k> 2s-1.$ There  are integers $w=w(s,k)$,  $N=N(s,k)$ such that, for any choice of 
$N$ linearly independent polynomials 
$$P_1,\ldots, P_N \quad \in \quad  \Q [ \log^{\uu}(x), \Li_1^{\uu}(x), \ldots, \Li_k^{\uu}(x) ] $$
which are homogeneous of degree $w$,  the equation 
\begin{equation}\label{detPeqn} \det (P_i(x_j)_{1\leq i,j\leq N}) =0  
\end{equation}
is satisfied for any $N$ integral points 
$x_1,\ldots, x_N \in X_S.$
\end{thm}

The integers $w, N$ are easily computable. One can also assume  (\S\ref{SectDepthonegrading}) that the $P_i$ do not involve any logarithms $\log^{\uu}(x)$.   The results of \cite{Kim1}, \cite{DanCWewers} are a consequence of this theorem.
 For $p$ a prime not in $S$, we can  take  the $p$-adic period of $(\ref{detPeqn})$ giving 
$$\det ( \per_p( P_i(x_j)_{1\leq i,j\leq N}))=0\ .$$
It defines an alternating  $p$-adic analytic function vanishing on $\wedge^N X_S$ and yields information about the disposition of  integral points.   For example, if $s=1$, and $k=2$, we can take $w=2, N=2$, set $x_1$ a tangent vector at infinity, and $x_2=x$ to retrieve Coleman's  equation 
$$ 
\det \begin{pmatrix} \Li_2^p(x)  &  \log^p(x) \Li_1^p(x)  \\ {1\over 2} & 1  \end{pmatrix}= 
\Li_2^p(x) - {1 \over 2} \log^p(x) \Li_1^p(x) =0 \qquad \hbox{ for all } x \in X_S\ .$$ 
The main result of \cite{DanCWewers} follows  from an application of  theorem \ref{introthm} with $N=3,w=4$. The approach in that paper used the conjectural non-vanishing of a $p$-adic zeta value.

The theorem can be applied as follows: given $x_1,\ldots, x_{N-1}$ `known' points in $X_S$, a row expansion of the determinant $(\ref{detPeqn})$ gives an explicit equation 
\begin{equation}\label{introPweqn} \sum_{i=1}^N  p_{w_i} P_i(x_N) =0 
\end{equation} 
which can be viewed as a constraint on any further point $x_N\in X_S$. One can  show, by a minimality argument,  that equation $(\ref{introPweqn})$ is not identically zero.  Thus, one has the following dichotomy:  either there are fewer than $N$ points on $X_S$, 
or there exists an explicit equation  for the remaining points on $X_S$.  In either case one reproves the  fact that $X_S$ is finite  \cite{Siegel, Kim1}.
The method easily extends to the cases considered in \cite{Hadian}.

Using the work  of Hardy and Ramunjan on the partition function, one can estimate the value of $N$. 
Unfortunately, it is rather large - possibly  slightly larger than the expected number of points on $X_S$.  In \S\ref{sectVirtual} we show how to replace  points $x_i\in X_S$ with divisors on $X(\Q)$ to   get around this problem.

Besser and de Jeu have shown how to compute the $p$-adic classical polylogarithms explicitly in \cite{Lip}, which raises the possibility of practical applications of the previous two  theorems for studying integral points.  
If one  assumes a version of Beilinson's conjecture, the results stated above carry through to curves of higher genus.  This is discussed in \S \ref{sectfinal}.
More precisely, for sufficiently many integral points $x_1,\ldots, x_N$
on (an integral model of) a curve $X$ of genus $\geq 2$,   the determinant of $p$-adic iterated integrals  along  Frobenius invariant paths from $p$ to $x_i$ should always vanish.  

\vspace{0.2in}

\emph{Acknowledgements.} This paper was partly written during a stay at the IHES in 2015. 
Many thanks to Richard Hain and Ishai Dan-Cohen for discussions on this topic.  The author is partly supported by ERC grant 724638 - `Galois theory of periods and applications'.

 \section{Motivation: non-abelian and `motivic' Abel-Jacobi map} \label{sectintrogen} We shall briefly explain, in this extended introduction,  why the ideas of this paper can be thought of as  a     natural generalisation of the classical Abel-Jacobi map.

\subsection{Abel-Jacobi} Let $X$ be a connected Riemann surface, and $p\in X$ a point.  The Abel-Jacobi map relative to the point $p \in X$ is
defined by 
\begin{eqnarray}  \label{introAJ}
X  &   \To  &  \mathrm{Hom}\big( H^0(X, \Omega^1_X), \C\big) \quad  / \quad  H_1(X;\Z)    \\
x & \mapsto &   \Big( \omega \mapsto \int_{\gamma_x} \omega \pmod \Lambda \nonumber  \Big)
\end{eqnarray}
where $\gamma_x$ is a continuous path from $p$ to $x$, and $\Lambda $ is  the period lattice:
\begin{eqnarray} \Lambda = \mathrm{Im} \Big( H_1(X;\Z) &\To & \mathrm{Hom}(H^0(X, \Omega^1_X),\C)\Big)\ 
\nonumber  \\
{[}\gamma] & \mapsto & (\omega \mapsto \int_{\gamma} \omega)  \nonumber 
\end{eqnarray}
 Since $H^0(X, \Omega^1_X) = F^1 H^1(X;\C)$,  it is customary to write the right-hand side of $(\ref{introAJ})$
as a space of double cosets
\begin{equation} \label{Jacob} F^0   \quad \backslash  \quad H^1(X;\C)^{\vee} \quad  / \quad  H_1(X;\Z) \ . 
\end{equation} 
The group in the middle  is classical or $C^{\infty}$ de Rham cohomology (we shall  denote algebraic de Rham cohomology by a subscript $dR$). 
However, the following form is better adapted to the non-abelian setting:
\begin{equation} \label{JacobBetter}  \mathrm{Hom}(F^1H^1(X;\C),\C)\quad  / \quad  H_1(X;\Z) 
\end{equation} 
This construction can be  generalised  in at least three different ways: first by defining a motivic version of this construction;  next by replacing cohomology with fundamental groups; and finally by changing cohomology theories (fiber functors).

\subsection{`Motivic' version}
 Now let $X$ be a smooth, geometrically connected  algebraic curve over $\Q$ for simplicity, and  suppose that $p, x \in X(\Q)$.  Define the motivic  version of the line integral to be the matrix coefficient (\S\ref{sectHcat}, or  \cite{NotesMot}, \S2) 
 $$I^{\mm}_{\gamma_x}(\omega) = [ H^1(X, \{p,x\}), [\gamma_x], [\omega]]^{\mm}$$
 in the ring of   periods $\Pe^{\mm}_C$ of a suitable category $C$ (for example, the category $\mathcal{H}$ of realisations of \S \ref{sectHcat}). 
 This means that $I_{\gamma_x}(\omega)$ is viewed as an element of the affine ring of tensor isomorphisms from de Rham to Betti fiber functors of the category $C$.
 In the  case $C=\mathcal{H}$,   the object $H^1(X,\{p,x\})$ in $\mathcal{H}$ denotes a triple
 $$(H^1(X(\C), \{p,x\}), H^1_{dR}(X,\{p,x\}), \mathrm{comp})$$
 where the first component is relative singular  (Betti) cohomology, the middle component is relative algebraic de Rham cohomology (both $\Q$-vector spaces) and 
 $\mathrm{comp}$ the canonical comparison isomorphism between their complexifications.  Since its  boundary   is contained in $\{p,x\}$, 
 the path $\gamma_x$  defines a relative homology class:
 $$[\gamma_x ] \qquad \in  \qquad H_1(X(\C), \{p,x\}) \cong H^1(X(\C), \{p,x\})^{\vee}\ .$$
 On the other hand, the holomorphic differential one-form $\omega$ restricts to zero at $x$ and $p$ and has a well-defined class in 
 \emph{relative} algebraic de Rham cohomology $H^1_{dR}(X, \{p,x\})$.
  One way to see this is to consider   the long exact relative cohomology sequence $$H^0_{dR}(\{x,p\}) \rightarrow H^1_{dR}(X,\{x,p\}) \rightarrow H^1_{dR}(X)\rightarrow 0 \ .$$
  By applying the Hodge filtration functor $F^1$, 
   we deduce an isomorphism
    $$F^1 H^1_{dR}(X,\{x,p\}) \overset{\sim}{\To} F^1 H^1_{dR}(X)\ .$$
   It is important to point out that  since $F^1 H^1_{dR}(X) \cong H^0(X, \Omega^1(X))$, we can indeed view the cohomology class of a  holomorphic differential form, via the previous isomorphism,  as a relative cohomology class in a canonical way, independently of $x$.

 Replacing the de Rham-Betti comparison isomorphism $(\ref{introcomp})$ with its universal version $(\ref{introPcomp})$, we can define a `motivic' version of the Abel-Jacobi map.
 
 \begin{defn}  Consider the map 
 \begin{eqnarray}  \label{introAJm}
X(\Q)  &   \To  &  \mathrm{Hom}\big(H^0(X, \Omega^1_X),  \Pe^{\mm}_C\big) \quad  / \quad  H_1(X;\Z)    \\
x & \mapsto &   \Big( \omega \mapsto  I^{\mm}_{\gamma_x}(\omega)    \pmod {\Lambda^{\mm}} \nonumber  \Big)
\end{eqnarray}
where $\Lambda^{\mm}$ is  the lattice of motivic periods of $X$, defined to be
\begin{eqnarray} \Lambda^{\mm} = \mathrm{Im} \Big( H_1(X;\Z) &\To & \mathrm{Hom}(H^0(X, \Omega^1_X),\Pe^{\mm}_C)\Big)\ 
\nonumber  \\
{[}\gamma] & \mapsto & (\omega \mapsto  I^{\mm}_{\gamma}(\omega))  \nonumber 
\end{eqnarray}
where $I^{\mm}_{\gamma}(\omega)$ is the motivic period $[H^1(X), [\gamma], [\omega]]^{\mm}$.
\end{defn}

The classical Abel-Jacobi map  $(\ref{introAJ})$ is recovered from $(\ref{introAJm})$ by composing with the period homomorphism  $\mathrm{per}: \Pe^{\mm}_C \rightarrow \C$, which induces
$$\mathrm{Hom}\big(H^0(X, \Omega^1_X),  \Pe^{\mm}_C\big)  \To \mathrm{Hom}\big(H^0(X, \Omega^1_X),  \C)\ .$$

Furthermore, the map $(\ref{introAJm})$ extends in the usual manner to divisors of degree zero:
\begin{eqnarray} \label{introDivmap} 
\mathrm{Div}^0(X(\Q)) &  \To  &  \mathrm{Hom}\big( H^0(X, \Omega^1_X), \Pe^{\mm}_C \big) \quad  / \quad  H_1(X;\Z)   \\ 
\sum_i n_i (x_i -p) & \mapsto &    \Big( \omega \mapsto \sum_{i} n_i I^{\mm}_{\gamma_{x_i}}(\omega) \mod \Lambda^{\mm} \Big)\nonumber \ . 
\end{eqnarray} 
It  vanishes on principle divisors if $X$ is compact.  
There is an obvious extension to curves defined over a number fields $k$, where the Betti cohomology is relative to an embedding $k\hookrightarrow  \C$. To keep the discussion simple, we have chosen to work only over $\Q$ for this introduction.

\subsection{Non-abelian version}
One can now generalize by replacing homology with a suitable notion of algebraic fundamental  group.  Here we shall consider the unipotent fundamental groupoid. It is a groupoid object in a suitable category $C$ which has Betti and de Rham realisations,  denoted $\pi_1^B$ and $\pi_1^{dR}$ respectively. 
The Betti component is simply the Mal\v{c}ev completion of the topological fundamental groupoid and is equipped with a natural map, for any points $p,x \in X(\Q)$:
$$   \pi_1^{\mathrm{top}}(X,p,x) \To \pi_1^B(X,p,x)(\Q) \ , $$
where $\pi_1^{\mathrm{top}}$ denotes the homotopy classes of paths in $X(\C)$.
The Betti and de Rham fundamental groupoids  are related by a universal comparison:
\begin{equation} \label{pi1univcomp} \pi_1^{B}(X,p,x) \times \Pe^{\mm}_C \overset{\sim}{\To} \pi_1^{dR}(X,p,x) \times \Pe^{\mm}_C\ ,
\end{equation}
which, after applying the period homomorphism, becomes an isomorphism over $\C$. 
The literature would suggest  replacing $(\ref{Jacob})$
with  the higher Albanese manifolds \cite{HaAlb}
$$    F^0 \quad  \backslash \quad  \pi_1^{dR}(X,p)(\C)  \quad  /  \quad \pi_1^{\mathrm{top}}(X,p) \ .$$
Here $F^0$ denotes a  subgroup scheme of $\pi_1^{dR}(X,p) $ defined using the Hodge filtration. There is an analogous notion 
for the de Rham torsor of paths $\pi_1^{dR}(X,p,x)$ (\S \ref{HainAlba}).
A problem with this construction is that although the coset spaces
$$F^0 \backslash \pi_1^{dR}(X,p) \cong F^0 \backslash \pi_1^{dR}(X,p,x)$$
are isomorphic as schemes, they are not canonically so, and the resulting constructions (including the coefficients of any Coleman function constructed in this manner) will depend implicitly on the point $x$. To circumvent  this, we shall instead use 
a scheme  $ \underline{\pi}_1^{dR}$, also defined using the Hodge filtration,  which admits a dominant morphism
\begin{equation} \label{pi1topi1bar} \pi_1^{dR}(X,p,x) \To   \underline{\pi}_1^{dR}(X,p,x)
\end{equation}
but is now independent of the point $x$.  It has a \emph{canonical de Rham path} 
$${}_x 1_p \in  \underline{\pi}_1^{dR}(X,p,x)(k) \ . $$
The point is that there is now a canonical isomorphism,
\begin{equation}\label{introcaniso}   \Or (\underline{\pi}_1^{dR}(X,p,x))  \overset{\sim}{\To} \Or (\underline{\pi}_1^{dR}(X,p)) \ 
\end{equation}
which generalises the crucial fact $F^1 H^1_{dR}(X,\{x,p\}) = F^1 H^1_{dR}(X)$.  Furthermore, the elements of the affine ring of $\underline{\pi}_1^{dR}$ can be viewed as holomorphic iterated integrals, which again generalises the holomorphic line integrals in the classical situation $(\ref{introAJ})$. 
The morphism $(\ref{pi1topi1bar})$ factors through $F^0 \backslash \pi^{dR}_1(X,p,x)$, but the latter is not isomorphic to $\underline{\pi}_1^{dR}(X,p,x)$ in general, except  when $H_{dR}^1(X)$ is  mixed Tate.

We obtain a sequence of maps
$$ \pi_1^{\mathrm{top}}(X,p) \To \pi_1^B(X,p)(\Q) \overset{(\ref{pi1univcomp})}{\To} \pi_1^{dR}(X,p)(\Pe^{\mm}_C) $$
 
 \begin{defn}  Out of this, we  can define  a map 
\begin{eqnarray} \label{AJnonab} 
 X(\Q)  & \To &       \mathrm{Hom}_{\alg} \big( \Or(\underline{\pi}_1^{dR}(X,p) ) ,  \Pe^{\mm}_C\big)\quad  /  \quad \pi_1^{\mathrm{top}}(X,p)   \\
x & \mapsto &  \big(  \omega \mapsto I^{\mm}_{\gamma_x}(\omega)    \big) \nonumber 
\end{eqnarray} 
where $\omega$ is an element of   $\Or (\underline{\pi}_1^{dR}(X,p))$, and $\gamma_x$ is a  path from $p$ to $x$. 
The notation $\mathrm{Hom}_{\alg}$ denotes algebra homomorphisms, and so  $(\ref{AJnonab})$ could also be written
$$X(\Q) \To   \underline{\pi}_1^{dR}(X,p) (\Pe^{\mm}_C) \quad  / \quad   \pi_1^{\mathrm{top}}(X,p)\ .$$
 The  `motivic period' $I^{\mm}_{\gamma_x}(\omega)$ is defined by the matrix coefficient 
$$I^{\mm}_{\gamma_x}(\omega)=  [\Or(\pi_1(X,p,x)), \gamma_x, \omega]^{\mm}  \qquad \in \qquad \Pe^{\mm}_C $$  and is the motivic version of  the  iterated integral of $\omega$ along the path $\gamma_x$.  

The right hand side of $(\ref{AJnonab})$  is a direct generalisation of $(\ref{JacobBetter})$. 

\end{defn}

 The key point is that $\omega \in \Or (\underline{\pi}_1^{dR}(X,p))$ is viewed as an element in the ring $\Or(\underline{\pi}^{dR}_1(X,p,x)) \subset 
 \Or(\pi^{dR}_1(X,p,x))$
via the canonical isomorphism $(\ref{introcaniso})$, and illustrates why the Albanese manifolds cannot be used  without introducing some hidden choices.

Applying the period homomorphism, we obtain a non-abelian Abel-Jacobi  map 
\begin{eqnarray} X(\Q)   & \To&      \underline{\pi}_1^{dR}(X,p )(\C) \quad / \quad \pi_1^{\mathrm{top}}(X,p) \nonumber \\
x & \mapsto & \big(  \omega \mapsto \int^p_{x} \omega \big) \ . \nonumber
\end{eqnarray} 
   The iterated integral makes sense for any point $x\in X(\C)$ and so this map extends to  a smooth and in fact, holomorphic,  map
    \begin{equation} X(\C) \To   \underline{\pi}_1^{dR}(X,p )(\C) \quad / \quad \pi_1^{\mathrm{top}}(X,p)\ .
   \end{equation}
     It extends equally well to divisors of degree zero, in the obvious manner.  Furthermore, there is no reason to restrict to curves: the construction works for any  smooth geometrically connected algebraic variety $X$ over a number field $k \subset \C$.
   
   \begin{rem} The image of $\pi_1^{\mathrm{top}}(X,p)$ inside $\pi_1^{dR}(X,p)(\Pe^{\mm}_C)$, which generalises the lattice of periods, is interesting. Already in the case when $X = \Pro^1\backslash \{0,1,\infty\}$, and $p$ is a tangent vector of length $1$ at the origin, the coefficients of this map are generated by the ring of  motivic multiple zeta values and the motivic version of $2 \pi i$. 
      \end{rem} 
   
   A simplification of $(\ref{AJnonab})$ is possible if we replace $X$ with a  simply connected open subset.   In that case,  the   path $\gamma_x$ is unique up to homotopy and
we do not need to quotient by the non-abelian periods, giving  a map
$X(\Q) \rightarrow  \underline{\pi}_1^{dR}(X,p) (\Pe^{\mm}_C)$.
   
\subsection{Changing cohomology theories} The classical Abel-Jacobi map is  obtained from  the comparison isomorphism from  de Rham to  Betti cohomology.  The method of Chabauty and Coleman via $p$-adic integrals is naturally expressed using the comparison from de Rham to crystalline cohomology. More generally, given fiber functors $\omega$ and $\omega'$ on the category $C$, we obtain a universal comparison, analogue of $(\ref{pi1univcomp})$
 \begin{equation}  \pi_1^{\omega'}(X,p,x) \times \Pe^{\omega,\omega'}_C \overset{\sim}{\To} \pi_1^{\omega}(X,p,x) \times \Pe^{\omega,\omega'}_C\ ,
\end{equation}
where $\Pe^{\omega, \omega'}_C$ is the  ring of $(\omega, \omega')$-periods of $C$. 
We shall mainly consider the cases where $\omega$ is the de Rham fiber functor, and $\omega'$ is de Rham, Betti or the crystalline fiber functor. 
In the latter  setting, the topological path  $\gamma_x$ can be replaced the unique Frobenius-invariant path   \cite{Besser}  between the reductions of $p$ and $x$.  For the majority of this paper,  we take 
$(\omega, \omega')$ to be $(dR, dR)$  and focus on the mixed Tate case, for which $\underline{\pi}_1^{dR}= \pi_1^{dR}$.
Other variants are possible:  Kim's  approach uses  $\ell$-adic  and crystalline cohomology \cite{KimAlb}.

 \section{Background on periods}
 
 \subsection{$C$-periods}  \label{sectCperiods}
Suppose that 
$C$ is a Tannakian category over $\Q$ such that
\begin{enumerate}
\item $C$ is equipped with a  pair of fiber functors $\omega_{dR}, \omega_B : C \rightarrow \mathrm{Vec}_{\Q}$ to the category of finite dimensional vector spaces over $\Q$.
\item  there is an isomorphism, functorial in $M$, and compatible with $\otimes$  
$$\comp_{B,dR}: \omega_{dR}(M) \otimes \C \overset{\sim}{\To} \omega_B(M) \otimes \C\ .$$
\end{enumerate}
 Define the ring of periods of $C$ to be  the affine ring
 $$\Pe^{\mm}_C = \Or(\Isom_{C}^{\otimes}(\omega_{dR}, \omega_B))\ .$$
It is a $\Q$-algebra, generated by matrix coefficients $[M, \gamma, \omega]^{\mm}$, where $\gamma \in \omega_B(M)^{\vee}$ and $\omega \in \omega_{dR}(M)$.  It  is the function which to $\phi \in \mathrm{Isom}^{\otimes}_C (\omega_{dR}, \omega_B)$ assigns $\gamma(\phi(\omega))$.  The ring $\Pe^{\mm}_{C}$ can be described in terms of generators and relations \cite{NotesMot}, \S2.

The  period homomorphism 
 $$\mathrm{per} : \Pe^{\mm}_C \To \C$$
 is defined by evaluation on the point $\comp_{B,dR} \in \Isom_{C}^{\otimes}(\omega_{dR}, \omega_B)(\C)$ defined by $(2)$. 
 For every object $M$ in $C$ there is a universal comparison map
 \begin{eqnarray}  \omega_{dR}(M) \otimes_{\Q} \Pe^{\mm}_C & \overset{\sim}{\To} & \omega_B(M) \otimes_{\Q} \Pe^{\mm}_C \\
 \omega \otimes 1 & \mapsto &  \sum e_i \otimes [M, e_i^{\vee}, \omega]^{\mm} \nonumber
 \end{eqnarray}  
where $\{e_i\}$ is a basis of $\omega_B(M)$, and $\{e_i^{\vee}\}$ the dual basis. This formula does not depend on the choice of basis. 
Composing the previous map with the period homomorphism, and tensoring with $\C$,  gives back the isomorphism $(2)$. 
See \cite{NotesMot}, \S2 for further details.

 \subsection{Examples} \label{sectHcat}
 As a first example, let $C$  be the category $\mathcal{H}$ (\cite{DeP1}, \S1.10) of triples $(V_{B}, V_{dR}, c)$ where 
 $V_{dR}, V_B$ are finite-dimensional $\Q$-vector spaces equipped with an increasing filtration $W$, and $c$ is an isomorphism 
 $c: V_{dR} \otimes \C \overset{\sim}{\rightarrow} V_B\otimes \C$ which respects $W$. The vector space $V_{dR}$ is furthermore equipped with a decreasing filtration $F$, such that $V_B$, equipped with $cF$ and $W$, is a graded-polarizable $\Q$-mixed Hodge structure. It can also be equipped with a functorial real Frobenius involution 
 $F_{\infty}: V_B \cong V_B$ but which will play a limited role here.  
 
 In this setting, we can define $\Pe^{\mm,+}_{\mathcal{H}}$ to be the subring of $\Pe^{\mm}_{\mathcal{H}}$ generated by matrix coefficients $[M, \gamma, \omega]^{\mm}$
 which are effective: i.e., $W_{-1}M=0$ \cite{NotesMot} \S3.2.

 For any smooth  scheme $X$ and a normal crossings divisor $D\subset X$ over $\Q$, Deligne's mixed Hodge theory implies that the triple $$H^n(X,D):=(H^n(X(\C),D(\C)), H^n_{dR}(X,D), \mathrm{comp}_{B, dR})$$
is  an object of $\mathcal{H}$, and depends  functorially in $X$.  In particular, 
we can define the `motivic' integral to be the  (effective) matrix coefficient $$ I^{\mm}_{\gamma}(\omega) = [ H^n(X,D), \gamma, \omega]^{\mm} \quad \in \quad \Pe^{\mm,+}_{\mathcal{H}}$$ 
whenever $\omega \in H^n_{dR}(X,D)$ and $\gamma$ is a smooth chain on $X(\C)$ whose boundary is contained in $D(\C)$.

Alternatively, one can take $C$ to be a variant of Nori's elementary category of motives. Such a formalism enables us to make sense of \S\ref{sectintrogen}, but  since we presently lack  control on the size of the space of motivic periods in this category,
 there are no unconditional applications to integral points. 
 
 Therefore, for the majority of this paper, we shall take  $C$ to be a category of mixed Tate motives over a number field (\S\ref{sectMT}).  In this situation, the space of motivic periods is under control, but this  condition poses  severe restrictions on  the scheme $X$.  
 
 \subsection{Variants}
The category $\mathcal{H}$  can be enriched to include  $\ell$-adic and crystalline components \cite{DeP1} \S1.10. 
More generally, given any two fiber functors $\omega_i $ from $C$ to the category of projective $R_i$-modules for $i=1,2$, we can form a ring of $\omega_1 , \omega_2$ periods 
$$\Pe^{\omega_1,\omega_2}_C = \Or ( \mathrm{Isom}^{\otimes}_C(\omega_2,\omega_1))\ .$$
Since $\mathrm{Isom}^{\otimes}_C(\omega_2,\omega_1)$, the scheme of tensor isomorphisms from $\omega_2$ to $\omega_1$, is a $G^{\omega_1}_C \times G^{\omega_2}_C$ bitorsor, where $G^{\omega_i}_C = \mathrm{Aut}_{C}^{\otimes}(\omega_i)$, it follows that   $G^{\omega_2}_C$ acts on $\Pe^{\omega_1,\omega_2}_C$ on the left, $G^{\omega_1}_C$ acts on the right.  Later in the paper, we wish to have Tannaka groups acting on the left of fundamental groups, and hence on the right of their affine rings. Since this is the opposite of the usual convention, where groups act on their representations on the left, this   will have the effect of reversing `left' and `right' throughout this section. 
The ring $\Pe^{\omega_1,\omega_2}_C$  is generated by matrix coefficients $[M, v_1, v_2]^{\omega_1,\omega_2}$ where $v_1 \in \omega_1(M)^{\vee}$, $v_2 \in \omega_2(M)$.  Of particular relevance for the applications discussed here  are the cases $(\omega_1, \omega_2) = (dR, dR)$, and $(\omega_1, \omega_2) = (crys,dR)$, and $(\omega_1,\omega_2) = (B,dR) = : \mm$.

 \subsection{Mixed Tate motives over number fields \cite{Levine, DeGo}} \label{sectMT}
 Let   $k$ be a  number field, and $S$ a finite set of primes in $\Or_k$.   Let $\MT_k$ denote the category of mixed Tate motives  over $k$, and $\MT_S\subset \MT_k$ the full subcategory of motives unramified at $S$.  It is a neutral Tannakian category over $\Q$ with 
 a canonical fiber functor 
 $$\omega : \MT_k \To \mathrm{Vec}_{\Q}$$
 which factors through the category of graded vector spaces. The de Rham fiber functor $\omega_{dR}: \MT_k \rightarrow \mathrm{Vec}_k$ is obtained from $\omega$ by tensoring with $k$: $\omega_{dR} = \omega\otimes k$.    Let 
 $$G_S^{\omega} = \mathrm{Aut}^{\otimes}_{\MT_S}(\omega) $$
 be the  Tannaka group relative to $\omega$. It is an affine group scheme over $\Q$, and if we denote by $U^{\omega}_S$ its pro-unipotent radical, we have a canonical splitting 
 \begin{equation} \label{Gssplit} G_S^{\omega} = U_S^{\omega} \rtimes \mathbb{G}_m\ .
 \end{equation}
 For every embedding $\sigma: k \hookrightarrow \C$ there is a Betti fiber functor $\omega_{\sigma} : \MT_k \rightarrow \mathrm{Vec}_{\Q}$, and hence we have a $\Q$-algebra of motivic periods $\Pe^{\mm_\sigma}_{\MT_S}= \Or( \mathrm{Isom}^{\otimes}_{\MT_S}(\omega, \omega_{\sigma}))$. It has a left action of $G^{\omega}_S$ and a period homomorphism $\per:  \Pe^{\mm_\sigma}_{\MT_S}\rightarrow \C$.  It is generated by matrix coefficients $[M, \gamma, v]$ where $v\in \omega(M)$, and $\gamma \in \omega_{\sigma}(M)^{\vee}$.

\subsubsection{Unipotent periods} The ring of $\omega$-periods  $\Pe^{\omega}_S$ is defined to 
be the affine ring $\Or(G^{\omega}_S)$, and is generated by matrix coefficients $[M, v, f]^{\omega}$ where $v\in \omega(M)$ and $f \in \omega(M)^{\vee}$. Finally, the ring of unipotent periods is  
 \begin{equation} \label{PeUdefn} \Pe^{\uu}_S = \Or(U^{\omega}_S)
 \end{equation}
 and is generated by the restrictions of the matrix coefficients $[M, v,f]^{\omega}$  to $U^{\omega}_S$.  It is equipped with the conjugation 
 action of $G^{\omega}_S$, and in particular, is graded via the action of $\G_m$. The object
 $[M,v,f]^{\omega}$ is the function which   to $u \in U^{\omega}_S$
   assigns $f(uv) \in \A^1$.  
 
 \subsubsection{Structure} \label{sectStructure}  Let $\Or_S$ denote the ring of $S$-integers in $k$.
 It is known  \cite{DeGo, Levine} that 
 $$\mathrm{Ext}^1_{\MT_S} (\Q(0), \Q(-n))  \cong K_{2n-1}(\Or_S) \otimes_{\Z} \Q$$
 and all higher Ext groups  in $\MT_S$ vanish. 
 Furthermore, Borel proved that
 \begin{equation}\label{Borelrank}
 \dim_{\Q}  K_{2n-1}(\Or_S) \otimes_{\Z} \Q = \begin{cases} |S| \quad\qquad \hbox{ if } n= 1 \ , \\
 r_2 \quad\qquad \hbox{ if } n >1 \hbox{ even} \\  r_1+r_2 \quad \hbox{ if } n >1 \hbox{ odd} \\ \end{cases}
 \end{equation}
 where $r_1$ (resp. $r_2$) is the number of real  (resp. complex) places  of $k$.  It follows that  the graded Lie algebra of $\UdrS$ is free with $|S|$ generators in degree $1$, $r_2$ generators in even degrees, and $r_1+r_2$ generators in odd degrees. 
 The affine ring of $\UdrS$ is the graded dual to the universal envelopping algebra of $\Lie^{\!\mathrm{gr}} U^{\omega}_{S}$. This gives precise control on the  dimensions of the weight-graded pieces of  $\Pe^{\uu}_S$, and is  the key
 input for theorem \ref{introthm}.  
 
   \subsubsection{$p$-adic period} We shall assume that the category of mixed Tate motives over $\Or_S$  is equipped with a  crystalline fiber functor 
 for all primes $p$ of  $\Or_S$, and a canonical comparison isomorphism $\mathrm{comp}_{dR, crys}: M_{dR}\otimes k_p \overset{\sim}{\rightarrow} M_{crys} $.
 This was announced in \cite{Yamashita}. 
  Pulling back the Frobenius via this map
  defines a functorial isomorphism 
 $$F_p : M_{dR} \otimes k_p \overset{\sim}{\rightarrow} M_{dR} \otimes k_p $$
 and hence a canonical element $F_p \in \GdrS(k_p)$. It acts on the de Rham realisation of $\Q(-1)$ by multiplication by $p$, so  its  image in $\G_m(k_p)$  is the element $p$.  Denote its image in $\GdrS(k_p)$ by $p$ also. Hence one obtains an element 
 $$\overline{F}_p = p^{-1}   F_p   \quad  \in \quad \UdrS(k_p)\ .$$
 One can also consider $F_p p^{-1}  = p \overline{F}_p p^{-1} \in \UdrS(k_p)$. 
  The (single-valued) $p$-adic period is  the homomorphism  
 $$ \per_p : \Pe^{\uu}_S \To k_p $$
defined by  evaluation at the point  $\overline{F}_p$. 
 There is a Betti analogue:
 $\per_{\sigma}: \Pe^{\uu}_S \rightarrow  k_{\sigma}$
 where $\sigma : k \rightarrow \R$ is a real place of $k$. It
 is the `single-valued period', corresponding to  the  isomorphisms
$M_{dR} \otimes \C  \overset{\sim}{\rightarrow} M_{B, \sigma} \otimes \C{\rightarrow} M_{B, \sigma} \otimes \C \overset{\sim}{\leftarrow} M_{dR} \otimes \C $
 where the map in the middle  is the real Frobenius $(-1) F_{\infty}\otimes \id $ at $\sigma$, where $-1 \in \G_m(\Q)$.

 \section{Background on unipotent fundamental groups}  \label{sectunip} Let $X$ be a smooth, geometrically connected scheme over a number field $k\subset \C$. 
 
 \subsubsection{Betti fundamental groupoid} Let $x \in X(\C)$, and let $\pi_1^{un}(X(\C),y,x)$ denote the unipotent completion of the topological fundamental groupoid $\pi_1(X(\C),y,x)$. It is a scheme over $\Q$, equipped with a natural map from homotopy classes
 of paths from $x$ to $y$ into its set of rational points:
 $$\pi_1(X(\C), y,x) \To \pi_1^{un}(X(\C),y,x)(\Q)\ .$$
 It forms a groupoid which is compatible with the previous map:
 $$\pi_1^{un}(X(\C), z, y) \times \pi_1^{un}(X(\C),  y,x)   \rightarrow \pi_1^{un}(X(\C), z, x)\ .$$
 We shall compose paths in the functional order: $\alpha. \beta$ means $\beta$ followed by $\alpha$. 

 The unipotent fundamental group admits the following Tannakian formulation. Consider the category $\mathcal{L}(X)$
 of local systems of finite-dimensional $\Q$-vector spaces on $X(\C)$, which are equipped with an exhaustive increasing filtration by local sub systems of $\Q$-vector spaces, such that the associated graded pieces are trivial (constant). This is a Tannakian category, and for every point $x\in X(\C)$ there is a fiber functor $\omega_{x} : \mathcal{L}(X) \rightarrow \mathrm{Vec}_{\Q}$.  The unipotent completions can be retrieved as  follows:
 $$\pi_1^{un}(X(\C), x) = \mathrm{Aut}^{\otimes}_{\mathcal{L}(X)}(\omega_x) \qquad \hbox{ and } \qquad 
 \pi_1^{un}(X(\C), y,x) = \mathrm{Isom}^{\otimes}_{\mathcal{L}(X)}(\omega_y,\omega_x) \ .$$
 One can  replace $x$ or $y$ with tangential base points, and the  analogous statements remain true. The case where $X$ is a curve is treated in \cite{DeP1} \S15.

 \subsubsection{de Rham fundamental groupoid}
Let $U(X)$ be the Tannakian category of  unipotent algebraic vector bundles on $X$ defined over $k$. These are equipped with a flat connection with regular singularities at infinity, and an exhaustive increasing  filtration by flat $k$ sub-bundles such that the successive quotients are trivial. For any rational point $x \in X(k)$, the functor  $x^* : U(X) \rightarrow U(\mathrm{Spec}\, k)$ is a fiber functor $\omega_X: U(X) \rightarrow \mathrm{Vec}_{k}$. 

The fibers of the de Rham fundamental groupoid are the schemes over $k$
$$\pi_1^{dR}(X, x) = \mathrm{Aut}^{\otimes}_{U(X)}(\omega_x) \qquad \hbox{ and } \qquad 
 \pi_1^{dR}(X, y,x) = \mathrm{Isom}^{\otimes}_{U(X)}(\omega_y,\omega_x) \ .$$
They form a groupoid,  and in particular a torsor
\begin{equation}  
\pi_1^{dR}(X,y,x) \times \pi_1^{dR}(X,x) \To \pi_1^{dR}(X,y,x)  \label{dRTorsor}
\end{equation} 
As before, one can  replace $x$ or $y$ with  rational tangential base points.

 The   complexification of the affine ring of the de Rham fundamental group can be written down using the Chen's reduced bar construction \cite{Chen}.  If $A(X)$ denotes the de Rham complex of smooth differential forms on $X(\C)$, then 
 $$\Or(\pi_1^{dR}(X,y,x))\otimes_k \C \cong H^0 (\overline{B}(A(X))$$
 where $\overline{B}$ is the reduced bar construction. A version of this construction is available over $k$ in the case when $X$ is the complement of a normal crossing divisor in a smooth proper variety, all defined over $k$ (\cite{HainAlg}, \S14).  
 For example, if $X$ is an affine curve, then we can choose global regular $1$ forms $\omega_1,\ldots, \omega_m$ which represent the cohomology classes $H^1_{dR}(X;k)$. 
 In this case the affine ring is isomorphic to the tensor coalgebra
 \begin{equation}  \label{pi1asTc} \Or(\pi_1^{dR}(X,y,x)) \cong T^c \big( \bigoplus_{i=1}^m \omega_i k \big)\end{equation} 
 equipped with the shuffle product. The groupoid structure is dual to the deconcatenation operation on tensors of one-forms.
   
 \subsubsection{Comparison} The Riemann-Hilbert correspondence restricts to an equivalence of categories 
 $\mathcal{L}(X) \otimes_{\Q} \C \leftrightarrow \mathcal{U}(X)\otimes_k \C$
 and hence an isomorphism of schemes
 $$\mathrm{comp}_{B,dR}: \pi_1^{un}(X, y, x)\times \C \overset{\sim}{\To} \pi_1^{dR}(X, y,x)\times \C\ ,$$
 for all $x,y \in X(k)$.   It can be computed via the map
 \begin{eqnarray} \pi_1^{\mathrm{top}}(X,y,x) & \To & \mathrm{Hom} ( H^0( \overline{B}(A(X)), \C) \nonumber \\
 \gamma & \mapsto & \big( \omega \mapsto  \int_{\gamma} \omega\big) \nonumber
 \end{eqnarray} 
 where the second map is the iterated integral of $\omega$ along the path $\gamma$. Chen's $\pi_1$-de Rham  theorem  \cite{Chen} is equivalent to the statement
 that the comparison map is an isomorphism.

 \subsubsection{Beilinson's construction}
 When $X$ is the complement of a normal crossing divisor in a smooth projective variety, the unipotent fundamental group carries a natural mixed Hodge structure due to Morgan (via minimal models) and Hain (via the bar construction).
 A different approach is due to  Beilinson, who proved that the  unipotent fundamental groupoid is the Betti homology of a certain cosimplicial scheme constructed out of $X,x,y$, \cite{Wojtkowiak}. See \cite{DeGo}, \S3.3, for details.  
 
 In the case $x\neq y$,  consider $X^n = X\times \ldots \times X$, the product of $n$ copies of $X$, and let 
 $$D_0= \{x\} \times X^{n-1}\ , D_i = X^{i-1} \times \Delta \times X^{n-i-1} \hbox{ for } 1\leq i \leq n-1 \ , \ D_n = X^{n-1} \times \{y\}$$
where $\Delta \subset X\times X$ is the diagonal. Then Beilinson's theorem states  that 
 \begin{equation}
 \Or(\pi^{un}_1(X,y,x)) \cong \lim_{n\rightarrow \infty} \big( H^n (X^n(\C),   \cup_{i=0}^n D(\C)) \big)\ ,
 \end{equation}
 where the morphisms between the  relative cohomology  groups on the right are via face maps.
 The multiplicative and Hopf algebroid (dual to the groupoid) structures are geometric: they arise from  morphisms of algebraic varieties and hence 
 induce morphisms of the corresponding relative cohomology groups  on the right hand side (although not all these facts seem to have been proved explicitly in the literature).
  By replacing singular cohomology with another appropriate cohomology theory $\omega$, one obtains the corresponding
 notion of fundamental group:
 \begin{equation}  \label{pi1asgeometric} \pi^{\omega}_1(X,y,x):= \mathrm{Spec} \,  \Big( \lim_{n\rightarrow \infty} \big( H_{\omega}^n (X^n,   \cup_{i=0}^n D) \big)\Big)\ .
\end{equation} 
In particular,   the affine ring of the de Rham  fundamental groupoid is isomorphic to its de Rham  cohomology \cite{Wojtkowiak}, which is  automatically endowed with a $k$-mixed Hodge structure. Therefore,   the triple
$$ \Or(\pi^{un}_1(X,y,x)): = \big( \Or(\pi_1^B(X,y,x)), \Or(\pi_1^{dR}(X,y,x), \comp_{B, dR} \big)$$
is an Ind-object of  the category of realisations $\mathcal{H}$  (in the case $k=\Q$, and $x,y\in X(\Q)$ since, for simplicity, we only defined $\mathcal{H}$ over $\Q$).  The construction requires some modifications when $x, y$ are tangential or coincide.

 \subsubsection{Torsors}
 We shall write $(\ref{pi1asgeometric})$ as ${}_y \Pi^{\omega}_x$. 
 Composition of paths defines morphisms
 \begin{equation} \label{Picomppaths}
 {}_z \Pi^{\omega}_y \times {}_y \Pi^{\omega}_x \To {}_z \Pi^{\omega}_x 
 \end{equation} 
 with respect to which ${}_y \Pi^{\omega}_x$ is a  ${}_x \Pi^{\omega}_x$-torsor.   Dually, this is given by  a homomorphism
\begin{equation} \label{Picompdual} \Delta:    \Or(  {}_z \Pi^{\omega}_x ) \To \Or( {}_y \Pi^{\omega}_x) \otimes_k \Or({}_z \Pi^{\omega}_y)
\end{equation} 
 which is induced by a morphism of algebraic varieties. It has the property that  $ (\id \otimes m) (\Delta \otimes \id):     \Or(  {}_z \Pi^{\omega}_x )  \otimes_k \Or(  {}_z \Pi^{\omega}_y ) \rightarrow \Or( {}_y \Pi^{\omega}_x) \otimes_k \Or({}_z \Pi^{\omega}_y)$ is an isomorphism, where $m$ denotes multiplication.  
Furthermore, a torsor over a pro-unipotent affine group scheme has a rational point. This follows from the fact the Galois  cohomology of the additive group $H^1(\mathrm{Gal}(\overline{k}/k), \G_a)$ is trivial.

 \subsubsection{Hodge and weight filtrations} 
It follows from $(\ref{pi1asgeometric})$ that $\Or({}_y \Pi^{\omega}_x)$ is equipped with an increasing weight filtration $W$
which satisfies $W_{-1}=0$ and  
$$W_{0} \Or({}_y \Pi^{\omega}_x) \cong k$$
 when $X$ is connected. The filtration $W$  is respected by $(\ref{Picompdual})$. 
  In the  de Rham realisation, 
   $\Or({}_y \Pi^{dR}_x)$ is equipped, as a consequence of   $(\ref{pi1asgeometric})$, with a natural decreasing Hodge filtration $F$, which satisfies $F^0 \Or({}_y \Pi^{dR}_x) = \Or({}_y \Pi^{dR}_x)$. It is respected by $(\ref{Picompdual})$.
   Since $\Or({}_y \Pi^{dR}_x)$ is an ind-object in the category $\mathcal{H}$,  and  morphisms in the category of mixed Hodge structures are strict with respect to $W$ and $F$, we deduce that $(\ref{Picompdual})$ is strict with respect to $W$,  and also $F$ in the case $\omega =dR$. 
These filtrations are compatible with the multiplicative structure: $W_n W_m \subset W_{n+m}$, and $F^n F^m \subset F^{n+m}$. 

\subsubsection{Other realisations}   \label{sectOtherReal} The advantage of Beilinson's construction
is that the $\ell$-adic and crystalline realisations come for free.  In fact, his  construction, being geometric,  defines 
 a groupoid  object in Nori's category of mixed motives.
In \S\ref{sectfinal} we shall  briefly mention the crystalline fundamental group  \cite{BesserTannaka}, \cite{Vologodsky}.
Under the assumptions \cite{KimAlb} that $X$ has a good model over the ring of $S$-integers in a number field, one can define  the crystalline fundamental group as a Tannaka group of  the category \cite{Fiso} of unipotent (and hence overconvergent) isocrystals. 

 \subsubsection{Motivic fundamental groupoid: mixed Tate case}
For the applications to integral points, we  require a fundamental groupoid in an actual abelian category of mixed Tate (or Artin-Tate) motives, which is highly restrictive.  For example, 
 $X$ can be  a smooth unirational variety over a number field $k$.  Our main example is  $X= \Pro^1 \backslash \{0,1,\infty\}$ over $\Q$.  In a certain class of cases, and for $x,y \in X(k)$, Deligne and Goncharov have shown that  Beilinson's construction is the realisation of the motivic fundamental groupoid 
  $$\pi^{\mathrm{mot}}_1(X, y,x)$$
 which is a pro-object of $\MT_k$ (its affine ring is an Ind-object in $\MT_k$). Call $(X,x,y)$  Tate and unramifed outside a finite set of primes $S$ if $\pi^{\mathrm{mot}}_1(X, y,x)$     lies in the subcategory $\MT_S \subset \MT_k$ for some finite set of primes $S$. The set of primes $S$ can be determined from a regular model of $X$ (\cite{DeGo}, proposition 4.17). In particular, if  $X = \Pro^1\backslash \{0,1,\infty\}$, and $x,y \in X_S$ are $S$-units, then $(X,x,y)$ is unramified outside $S$.

There are at least two peculiarities in the mixed Tate case that we can exploit. The first is the existence of the canonical fiber functor, which is more precise
than the de Rham fiber functor since it is over $\Q$. For ease of notation,  we shall write
\begin{equation} 
{}_y \Pi_x^{\omega} = \omega( \pi_1^{\mathrm{mot}}(X,y,x)) \ .
\end{equation}
Its affine ring is defined over $\Q$. The other peculiarity is that  $\Or({}_y \Pi^{\omega}_x)$ is  graded by the weight, and hence equipped with an augmentation 
$\varepsilon:  \Or({}_y \Pi^{\omega}_x) \rightarrow  \gr^W_0 \Or({}_y \Pi^{\omega}_x) \cong   \Q$. Equivalently, there is  a canonical $\omega$-path  from $x$ to $y$:
\begin{equation} \label{TateCanpath} {}_y 1_x \in {}_y \Pi_x^{\omega} \ .
\end{equation} 
We shall define a generalisation of the canonical de Rham path to the non-mixed Tate case in \S\ref{sectCanpath} using the Hodge filtration.

 \section{Some linear algebra} \label{sectLinAlg}
 The crux of our argument is a simple observation  in linear algebra.  All tensors will be over a field $k$. 
 
\subsection{Coaction of linear maps} \label{sect: co} 
Let $V,W$ be  finite-dimensional $k$-vector spaces. 
The action of   $k$-linear maps  from $V$ to $W$ defines   a canonical map:
$$ \phi \otimes v \mapsto \phi(v)\quad :  \quad  \mathrm{Hom}(V,W) \otimes V  \To    W\ .$$
Dually, this defines a   coaction:
\begin{equation}  \label{cofirst}
\co :  V \To W \otimes \mathrm{Hom}(V,W)^{\vee} \ .
\end{equation} 
Identifying $\mathrm{Hom}(V,W) = V^{\vee} \otimes W$, and $\mathrm{Hom}(V,W)^{\vee} = W^{\vee} \otimes V$, the map $\co$ can equivalently be defined as the composition 
$$ V \To k \otimes V \To  W\otimes W^{\vee}  \otimes V$$
 where the first map is $v \mapsto 1\otimes v $, and the  map $k \rightarrow W\otimes W^{\vee}$ is the map which sends $1$ to the identity map  $\id \in W \otimes W^{\vee} = \mathrm{Hom}(W^{\vee},W^{\vee})$.  Thus, if we pick a basis $\{e_i\}$ of $W$, and let $\{e_i^{\vee}\}$ denote the dual basis of $W$, then 
 $$\co (v )  =  \sum_i e_i \otimes [e_i^{\vee}, v]$$
 where $[w, v]$, for $w\in W^{\vee}, v\in V$  is the function on $\mathrm{Hom}(V,W)$ given by  $\phi \mapsto w(\phi(v))$. 

Any linear map $\phi: V \rightarrow W$  factorizes through the canonical map $(\ref{cofirst})$: 
$$ \phi : V \overset{\co}{\To}   W \otimes \mathrm{Hom}(V,W)^{\vee} \overset{\,\, \id \otimes \ev_{\phi}  \,\, }{\To} W\ ,$$
where the second map $\ev_{\phi}$ is `evaluation of $\phi$'.   This follows from 
$$ (\id \otimes \ev_{\phi}) \co(v) =  \sum_i e_i \otimes e_i^{\vee}(\phi(v)) = \phi(v)\ .$$
We shall use this construction in two different ways; firstly in the definition of the universal comparison isomorphism, and secondly by applying it to a ring of periods.

\subsection{Coaction by  a scheme $A$ of linear maps} \label{sectcoactSchemeA} A  vector space $U$  over $k$ of finite dimension can be viewed as an affine scheme, i.e., the functor from commutative unitary $k$-algebras $R$ to sets    $R \mapsto U\otimes R$ is representable. It is represented by the symmetric tensor algebra $\Or(U):=\bigoplus_{n\geq 0 } \mathrm{Sym}^n \, U^{\vee}$, since 
$$\mathrm{Hom}_{k\mathrm{-alg}} ( \bigoplus_{n\geq 0 } \mathrm{Sym}^n \, U^{\vee}, R) = \mathrm{Hom}(U^{\vee}, R) \cong U\otimes R\ .$$  
Furthermore, linear maps between vector spaces give rise to morphisms between the corresponding schemes. 
Apply this to $U=\mathrm{Hom}(V,W)$. Let us  suppose that  
\begin{equation} A \subseteq \mathrm{Hom}(V,W)\end{equation} is a closed subscheme. 
There is a natural linear map $U^{\vee} \hookrightarrow \Or(U) \rightarrow \Or(A)$. 

\begin{defn}  Denote the composition of $(\ref{cofirst})$ with the previous map by
\begin{equation} \label{secondco} \co:   V \To W\otimes \Or(A) \ .\end{equation}
\end{defn} 
Any element   $\phi \in A(R) \subset \mathrm{Hom}(V,W)\otimes R$
factorizes through  $(\ref{secondco})$:
$$\phi = (\id \otimes \ev_{\phi}) \co\ ,$$
where $\ev_{\phi}: \Or(A) \rightarrow R$ is the homomorphism which  represents $\phi$.  From now on, we shall drop the commutative ring $R$ from the notation, and simply write $\phi \in A$, where the dependence on $R$ is understood.

\subsection{Infinite-dimensional vector spaces} The previous construction only uses the finite-dimensionality of $W$, but extends  to the case when  $V$ is infinite-dimensional.  Writing $V= \varinjlim V_i$ as a limit of finite-dimensional vector spaces $V_i$, we have 
$$\mathrm{Hom}(V, W) = \varprojlim \mathrm{Hom}(V_i, W)\ , $$ 
which is a projective limit of schemes and hence a scheme.

On the other hand, we can use the construction with infinite-dimensional  graded vector spaces $V= \bigoplus_n V_n$ whose
graded quotients $V_n$ are finite-dimensional $k$-vector spaces, if  one interprets $\vee$ to be   the graded dual, i.e., $V^{\vee} := \bigoplus_n V^{\vee}_n$.

\subsection{Determinant equation} \label{sect: det} 
Now  suppose in addition that  $W$ is  a commutative ring.
By extending scalars from $k$ to $W$ the map    $(\ref{secondco})$ defines    a $W$-linear map   
$$    \co :      W\otimes V  \To  W\otimes \Or(A) $$
given by $w\otimes v  \mapsto w \co(v) $. For any $n\geq 0$, we deduce  a $W$-linear map 
$$\textstyle{   \bigwedge^n \co :  W\otimes  \bigwedge^n V   \To  W\otimes   \bigwedge^n 
  \Or(A)  }    \ . $$
Suppose that    $N=\dim_k \Or(A) +1$ is finite.  Then $\bigwedge^N \Or(A)$ vanishes and so  $\bigwedge^N \co$ is the zero map.  
  Thus given $\phi_1,\ldots, \phi_N$ in $A$ the map
$$\phi_1 \wedge \ldots \wedge \phi_N  :   \textstyle{ \bigwedge^N V}  \To W
$$
vanishes  since it factors through $\bigwedge^N \co$.

\begin{lem} \label{lemdeteqn} If $N=\dim_k \Or(A) +1 $ is finite,  then given any $\phi_1,\ldots, \phi_N \in A$, 
\begin{equation} \label{detequation}  \det  \big( \phi_i(v_j) \big)_{1\leq i,j\leq N} =0 
\end{equation}
for all   $v_1,\ldots, v_N \in V$.
\end{lem}
\begin{rem}  \label{remuniveqn} An  
$\alpha \in \mathrm{ker} \big( W \otimes V    \overset{\co}{\rightarrow}    W \otimes \Or(A) \big) \subset W\otimes V$
 defines a universal equation 
\begin{equation}  \label{univeqn}
 (\id \otimes \phi  )(\alpha) = 0  \qquad \hbox{ for all } \phi \in A \ . 
 \end{equation} 
A non-zero $\alpha$  exists if  and only if $\dim_k V > \dim_k \Or(A)$. Note that $(\ref{detequation})$ is intrinsic and defined over $k$, but $(\ref{univeqn})$
depends on a choice of $\alpha$, and  involves coefficients in $W$. 
\end{rem}

\subsection{Filtered version}\label{sect: filt} Let $V,W$ be as in \S\ref{sect: co}.
 In our applications,  $V, W$ are equipped with increasing filtrations $F_n V, F_nW$ by subspaces.  Let $\mathrm{Hom}_F(V,W)$ denote the space  of linear maps respecting $F$, i.e., $\phi: V \rightarrow W$ such that $\phi(F_nV)\subset F_nW$ for all $n$. 
We obtain  an increasing filtration, also denoted by $F$, on 
$ \mathrm{Hom}_F(V, W)^{\vee}$ via
$$F_n \big(\mathrm{Hom}(V, W)^{\vee} \big):= \mathrm{Im} \big( \mathrm{Hom}_F (F_n V, F_nW)^{\vee} \rightarrow \mathrm{Hom}_F(V, W)^{\vee}\big) $$
where the second map is the dual of the restriction $\mathrm{Hom}_F(V,W) \rightarrow \mathrm{Hom}_F(F_n V, F_n W).$ 
Note that this is not the standard definition of a filtration on $\mathrm{Hom}$. 

The space $\mathrm{Hom}_F(F_nV,F_n W)$ acts on every $F_nV$ on the left
$$\mathrm{Hom}_F(F_nV,F_n W) \otimes F_nV \To F_nW$$
so dually,  $(\ref{cofirst})$ restricts to a linear map
$$\co : F_n V \To F_n W \otimes F_n \mathrm{Hom}_F(V, W)^{\vee}\ .$$
\begin{lem} The subspace $\mathrm{Hom}_F(V,W)$ is a closed subscheme of $\mathrm{Hom}(V,W)$.
\end{lem}
\begin{proof}
Since $V$ and $W$ are finite-dimensional, there exist integers $n_0, n_1$ such that 
$F_n  = F_{n_0}$ for all $n_0 \leq n$  and $F_n = F_{n_1}$ for all $n\geq n_1$. Here, $F_i=F_j$ means that both $F_i V= F_jV$ and $F_iW = F_jW$. 
Consider the natural map
$$\mathrm{Hom}(V,W) \To \bigoplus_{n_0 \leq n \leq n_1} \mathrm{Hom}(F_n V, W/F_nW)\ .$$
Then  $\mathrm{Hom}_F(V,W)$ is its  kernel, i.e., the inverse image of the point $0$.
\end{proof} 
Now suppose that $A$ is a closed  subscheme of $\mathrm{Hom}_F (V, W)$. In particular, the points of $A$ respect the filtration $F$. 
Then  we have
$$\co : F_n V \To  F_n W \otimes F_n \Or(A) \ ,$$
where $F_n \Or(A)$ is the increasing  filtration on $\Or(A)$ defined by the image of the subspace $F_n  \Or(\mathrm{Hom}_F(V,W))  =  \Or(\mathrm{Hom}_F(F_nV,F_nW))$ via $\Or(\mathrm{Hom}_F(V,W)) \rightarrow \Or(A)$. 

\begin{cor}  \label{corHomFscheme} Now suppose that $V,W$ are filtered vector spaces, possibly infinite dimensional, such  that 
$F_nW$ is finite-dimensional for all $n$. Then 
$\mathrm{Hom}_F(V,W)$
is representable, and therefore an affine scheme.
\end{cor} 
\begin{proof}Let $V= \varinjlim V_i$ such that $F_n V_i= V_i \cap F_nV$ is finite-dimensional for all $i$. Then $\mathrm{Hom}_F(V,W) = \varprojlim \mathrm{Hom}_F(V_i,W)$ and 
$ \mathrm{Hom}_F(V_i,W)=\varprojlim  \mathrm{Hom}_F(F_nV_i,F_n W)$. Therefore $\mathrm{Hom}_F(V,W)$ is a projective limit of schemes. 
\end{proof}
\subsection{Graded version}\label{sect: grad} 
Now consider the case when $V =\bigoplus_n V_n, W=\bigoplus_n W_n$ are graded by integers. Equivalently, $V, W$ admit a left action by $\G_m$. Define 
$\mathrm{Hom}_{\G_m}(V,W)$ to be the subspace of $\mathrm{Hom}(V,W)$  of linear maps which respect the gradings,  
and let us write 
$\mathrm{Hom}_n (V,W) = \mathrm{Hom}(V_n, W_n)$. The reader is  warned that this is not the usual grading on $\mathrm{Hom}_{\G_m}(V,W)$. 
 The degree $n$ component of $\co$ is the map 
$$ \co_n : V_n \To W_n \otimes  \mathrm{Hom}_n (V, W)^{\vee }\ . $$
As in the previous paragraph, one  easily checks that $\mathrm{Hom}_{\G_m} (V, W)$ is a subscheme of $\mathrm{Hom}(V,W)$. 
 In the case when  $A$ is a closed subscheme of  $\mathrm{Hom}_{\G_m} (V, W)$, i.e., when points of $A$  respect the gradings,
then we have a coaction
$$\co : V_n \To  W_n \otimes  \Or(A)_n \ $$
where $\Or(A)_n$ is the induced grading  on $ \Or(A) $.

 \section{Non-abelian algebraic cocycles}
 
 We consider  the space of one-cocycles of an affine group scheme with coefficients in another affine group scheme. We give necessary conditions for this  to define a scheme. 
 \subsection{Reminders on cocycles} Let  $G, \Pi$ be groups such that $G$ acts on $\Pi$ on the left. This means that 
 $g(ab)= g(a) g(b)$ for all $g\in G$ and $a,b \in \Pi$. 
   Recall that the set of non-abelian cocycles $Z^1(G,\Pi)$ is the pointed set 
 consisting of maps  $c: G \rightarrow \Pi$ of sets satisfying the cocycle condition 
 $$c_{gh} = c_g \, g(c_h) \qquad \hbox{ for all } g,h \in G\ .$$
This implies that $c_1=1$. The trivial cocycle is defined by  $c_{g}=1$ for all $g\in G$. 
The set of cocycles is equivalently given by
\begin{equation} \label{Z1asHom} Z^1(G,\Pi) \cong \ker \big(  \mathrm{Hom}(G, G \rtimes \Pi) \To \mathrm{Hom}(G,G)\big)\ ,
\end{equation} 
since one verifies that given any $c: G\rightarrow \Pi$, the map  $g\mapsto (g, c_g) : G \rightarrow  G \rtimes \Pi $  is a homomorphism if and only if $c$ satisfies the cocycle condition above.

Two cocycles $c,c'$ are called equivalent if there exists $b \in \Pi$ such that  
$$c'_g = b \,  c_g \, g(b^{-1})$$
for all $g\in G$. The pointed set of equivalence classes is denoted by $H^1(G, \Pi)$. 
 
 \subsection{Cocycles for affine group schemes}  \label{sectCocycAffineGpSch}
 Now let $G, \Pi$ be two affine group schemes over a field $k$, and suppose that $G$ acts on $\Pi$ on the left.  
 There is a morphism
 \begin{equation}\label{GPileftact}  G \times \Pi \To \Pi \end{equation}
 such that  $g(ab) = g(a) g(b)$ for all ($R$-valued) points $g \in G$ and $a,b\in \Pi$. It can equivalently be expressed by a commutative diagram, which we leave to the reader.

 \begin{defn}Define an \emph{algebraic cocycle} over $k$ to be  a morphism of schemes 
 $$c: G   \To \Pi $$ 
 such that  for all commutative unitary $k$-algebras $R$,  the induced map on $R$-valued points  $G(R) \rightarrow \Pi(R)$ is a cocycle. 
 Denote the set of such morphisms by $Z^1_{\alg}(G,\Pi)(k)$.  

Define a functor $Z^1_{\alg}(G, \Pi)$  from commutative unitary $k$ algebras $k'$ to pointed sets
 $$k' \mapsto Z^1_{\alg}(G\times k',\Pi\times k')(k') \qquad  \quad (  \, =: Z^1_{\alg}(G, \Pi)(k')\,) \ .$$
 \end{defn} 

Modifying by a coboundary defines a  natural transformation of functors 
\begin{eqnarray}   \Pi \times Z_{\alg}^1(G,\Pi) & \overset{\sim}{\To} &  Z_{\alg}^1(G,\Pi)   \nonumber  \\
b , c & \mapsto & (g\mapsto b^{-1} c_g g(b) )  \ .\nonumber 
\end{eqnarray}

  \begin{ex} 
  \begin{enumerate} 
  \item   If $G$ is a constant group scheme of finite type, then 
  $$Z^1_{\alg}(G, \Pi)(k) \rightarrow Z^1(G(k), \Pi(k))$$ is  a bijection since every cocycle on points is algebraic. 
  \item In general, this is false. Let $k=\Q$, and  $G= \Pi = \G_m$ with $G$ acting trivially  on $\Pi$. Then $Z^1(G(k), \Pi(k))$ is the set of homomorphisms $c:\Q^{\times}\rightarrow \Q^{\times}$.   Setting $c(-1)=1$, and assigning any value to $c(p)$ for $p$ a positive prime defines such a homomorphism.  It is rarely  algebraic.
  \item  The functor $Z^1_{\alg}(G,\Pi)$ is not necessarily representable. Let $k=\Q,\Pi=\G_m$ and $G=\mu_N= \Spec \Q[x]/(x^N-1)$, the group scheme of $N$-th roots of unity, with $G$ acting trivially on $\Pi$. Then elements of the set  $Z^1_{\alg}(G,\Pi)(R)$
   are given by   algebra homomorphisms
  $$ \phi : R[x,x^{-1}] \rightarrow R[x ]/(x^N-1)    \hbox{ such that }  \Delta \phi = (\phi \otimes \phi) \Delta   \ ,$$
  where $\Delta (x) = x\otimes x$. Such a homomorphism is of the form $\phi(x) = x^n$ where   $0 \leq n < N$. Therefore $Z^1_{\alg}(G, \Pi)(R)$ is the pointed set $
\Z/N\Z $. It  is not representable since for any representable functor $F$, $F(R \times S) = F(R) \times F(S)$. 
 \end{enumerate}
  \end{ex}

  \subsection{Interpretations}  \label{sect3views} We shall view a  cocycle  $c\in Z_{\alg}^1(G,\Pi)(k)$  in three ways:
 
 \begin{enumerate}
 \item as a  morphism  of schemes $c: G  \rightarrow \Pi $,  
 \item as a homomorphism  of algebras
 $c: \Or(\Pi) \rightarrow \Or(G)$,
 \item as a point  $ c\in \Pi(\Or(G))$,
  \end{enumerate}
 satisfying, in each case,  certain conditions which are equivalent to the cocycle relation. 
 In the first case, this can be expressed by the commutativity of 
$$
\begin{array}{ccc}
  G \times G&  \To  & \Pi \times \Pi    \\
\downarrow  &   &   \downarrow \\
G  &  \overset{c}{\To}   &   \Pi
\end{array}
$$ 
where the vertical maps are multiplication, and the map along the top is 
$$G \times G \overset{\Delta \times \id}{\To} G \times G \times G \overset{c \times \id \times c}{\To} \Pi \times G \times \Pi \overset{\id \times (\ref{GPileftact}) }{\To} \Pi \times \Pi \ .$$
The conditions for $(2)$ and $(3)$ are  left to the reader. 

\subsection{Filtrations} Let $k, G, \Pi$ be as in \S\ref{sectCocycAffineGpSch}.   Suppose that there exist  increasing  filtrations  $F$ on
  the affine rings $\Or(G), \Or(\Pi)$ such that  
 \begin{itemize}
\item  $F_{-1} H = 0$   and $H= \bigcup_n F_n H$
\item  $\Delta F_n H  \subset  F_n (H \otimes \Or(G))$
 \end{itemize}
 for $H  \in \{\Or(G), \Or(\Pi)\}$. The filtration on a tensor product is defined by 
 $$F_n (H \otimes H' )= \sum_{a+b=n} F_a H \otimes F_b H' \ .$$
Therefore both the coproduct on $\Or(G)$ dual to the group law on $G$, and the coaction of $\Or(G)$ on $\Or(\Pi)$ dual to the action of $G$ on $\Pi$,  respect the filtrations $F$.
 
\begin{defn} Define 
 $Z^1_{\alg, F}(G, \Pi)(R)$ to be the set of algebraic cocycles  $c: G\times R \rightarrow \Pi \times R$ such that 
 $c^* : \Or(\Pi)\otimes R \rightarrow \Or(G)\otimes R$ respects the filtrations $F$.
 \end{defn}

 It defines a functor from commutative unitary $k$-algebras to pointed sets. 

 \begin{prop}  \label{propZ1Fscheme} Let $k$ be a field and suppose that 
  \begin{equation}\label{dimkFnOG} \dim_k  F_n \Or(G) < \infty\ .
  \end{equation} 
  Then   $Z_{\alg, F}^1(G,\Pi)$ is  representable, and hence defines an affine scheme over $k$.
\end{prop} 

\begin{proof} Suppose that $A$ is an affine group scheme, and that $\Or(A)$ is equipped with an increasing filtration $F$ as above. 
Let $\mathrm{Hom}_F(G, A)$ denote the functor whose $R$-valued points, for $R$ a commutative unitary $k$-algebra,  are  $R$-linear  morphisms  of Hopf algebras
$$\phi : \Or(A)\otimes_k R \To \Or(G) \otimes_k R$$
such that $\phi$ respects $F$. Corollary \ref{corHomFscheme} implies that the functor of $k$-linear maps
$ \mathrm{Hom}_F( \Or(A) , \Or(G))$
is an affine scheme over $k$,   by $(\ref{dimkFnOG})$.
The subspace 
$$\mathrm{Hom}_F(G,A)(R)\,  \, \subset \, \,  \mathrm{Hom}_F( \Or(A) , \Or(G))(R)$$
consists of linear maps satisfying algebraic equations including:
$$\phi(xy) = \phi(x) \phi(y)  \quad \hbox{ and }  \quad  \Delta(\phi(x)) = (\phi \otimes \phi )\Delta(x)$$
for all $x, y \in \Or(A)$. There are additionally antipode, unit and counit conditions which we omit.  Since these are all algebraic and defined over $k$, it follows that $\mathrm{Hom}_F(A,G)$ is a closed subscheme of $\mathrm{Hom}_F( \Or(A) , \Or(G))$. 

Now apply this remark in the cases
$$A = G \qquad \hbox{ and } \qquad A = \Pi \rtimes G $$
where the semi-direct product on the right is given by the left action $(\ref{GPileftact})$ of $G$ on $\Pi$. There is an isomorphism of algebras
$\Or(\Pi \rtimes G) \cong \Or(\Pi) \otimes \Or(G)$
which endows $\Or(\Pi \rtimes G)$ with a filtration $F$ induced from that of $\Or(\Pi)$ and $\Or(G)$. 
It follows that $\mathrm{Hom}_F(G,\Pi)$ and $\mathrm{Hom}_F(G, \Pi \rtimes G)$ are both affine schemes. Since these scheme structures are natural, the morphism 
$\Pi \rtimes G \rightarrow G$ induces a morphism of affine schemes 
$$\mathrm{Hom}_F(G, \Pi \rtimes G) \To  \mathrm{Hom}_F(G, G)$$
and its kernel (inverse image of the identity) is also a  scheme. Via  $(\ref{Z1asHom})$, the kernel can be identified,  as a functor, with $Z^1_{\alg, F}(G, \Pi)$, which is therefore representable. 
\end{proof}

The proof shows that $Z^1_{\alg,F}(G,\Pi)$ is a closed subscheme of $\mathrm{Hom}_F(\Or(\Pi), \Or(G))$.  By Yoneda's lemma, the natural transformation of functors
$$Z^1_{\alg, F}(G,\Pi) \times G   \To \Pi $$
given on $R$-points by $(c,g) \mapsto c_g$ 
is  a  morphism of schemes. It is functorial in $G, \Pi$:   given affine group 
schemes $G' \rightarrow G$ and  $\Pi \rightarrow \Pi'$ whose affine rings are   equipped with filtrations in the manner described above, the natural maps   $Z_{\alg, F}^1(G,\Pi) \rightarrow Z_{\alg,F}^1(G, \Pi')$  and $Z_{\alg, F}^1(G, \Pi) \rightarrow Z_{\alg, F}^1(G', \Pi)$ are morphisms  of schemes.
 
 \begin{example} Let $U$ be a  pro-unipotent affine group scheme. Its affine ring $\Or(U)$ is equipped with a canonical filtration $C_n \Or(U)$ called the  coradical filtration. It is   defined by 
 $C_{-1} \Or(U) =0$,  $C_0 \Or(U) = k$, and for $n\geq 1$ 
 $$x\in C_n \Or(U)_+   \qquad \Longleftrightarrow \qquad  (\Delta')^n x = 0 $$
 where $\Delta'$ is the reduced coproduct  on $\Or(U)$ defined by $\Delta' = \Delta - \id \otimes 1 - 1 \otimes \id$,
 and $\Or(U)_+$ is the kernel of the augmentation $\varepsilon: \Or(U) \rightarrow k$.  Since $U$ is connected, we have  $\Or(U) = \Or(U)_+ \oplus k$.  
 The space $C_1 \Or(U)_+$ is the  space of primitive elements of $\Or(U)$. 
 Any morphism $U \rightarrow U'$ of prounipotent affine group schemes necessarily respects the coradical filtration since it is defined purely in terms of the coproducts.

 Since the kernel of $\Delta'$ is $C_1\Or(U)_+$, it follows from the equation 
 $$\Delta' C_n   \subset  \sum_{1 \leq i \leq n-1}  C_i \otimes C_{n-i}$$
 and induction on $n$ that 
 $$\dim C_1 \Or(U) < \infty   \qquad \Longleftrightarrow \qquad \dim C_n \Or(U) < \infty \quad \hbox{ for all } n \ .$$
 Note that  the space of primitive elements $C_1 \Or(U)_+ $ is isomorphic to $H^1(U; k)$ \cite{NotesMot}.
  \end{example} 
 
 \begin{cor} Let $U, \Pi$ be  pro-unipotent affine group schemes  with $\dim H^1(U;k) <\infty $. Then  the functor of morphisms 
 $\mathrm{Hom}(U,\Pi)$ 
 is an affine  scheme.
 
 If $G$ and $\Pi$ are both pro-unipotent and $\dim H^1(G; k)<\infty $ then 
 $Z_{\alg}^1(G, \Pi)$
 is  an affine  scheme. 
  \end{cor}

  \subsection{Torsors}  \label{sectTorsors}
Now let $\Pi_x$ be a right $\Pi$-torsor with $G$-action. This means that $\Pi_x$ is an affine scheme over $k$, equipped with a morphism 
$$(q, p)  \mapsto  q.p    :  \Pi_x \times \Pi  \To   \Pi_x  $$ 
 which defines a right action of $\Pi$ on $\Pi_x$, i.e., $q. (p p') = (q. p).p'$. Furthermore, $\Pi_x, \Pi$ both admit  left actions by $G$, i.e., morphisms
 $$G \times \Pi_x \To   \Pi_x \quad \hbox{ and } \quad G \times \Pi \To \Pi $$
  which define group actions,   which are compatible with the group law on $\Pi$ and right action of $\Pi$ on $\Pi_x$.
 All this can be translated into commutative diagrams, which we omit. The torsor condition is the 
  property that 
 \begin{equation} \label{torsordef} (q, p) \mapsto (q, q.p): \Pi_x \times \Pi  \To \Pi_x \times \Pi_x 
 \end{equation} 
  is an isomorphism of affine schemes over $k$. Finally, one requires the existence of a point on $\Pi_x(R)$ for some $R$ faithfully flat extension of $k$.

Since $k$ is a field, let us suppose   that we are given a point $1_x \in \Pi_x(k')$ for some extension $k'$ of $k$.  This data defines an algebraic cocycle over $k'$ as follows. 
 For all    $g\in G(R)$, for $R$ a commutative unitary $k'$ algebra, it follows from $(\ref{torsordef})$ that there exists a unique $\alpha_g \in \Pi(R)$ such that
$$ g(1_x) = 1_x . \alpha_g \ .$$
This defines a map $\alpha: G(R) \rightarrow \Pi(R)$. 
The cocycle condition follows from the uniqueness of $\alpha$, and the equations
$gh (1_x) = 1_x. \alpha_{gh}$ and 
$$g (h(1_x) )= g( 1_x . \alpha_h)  = 1_x. \alpha_g  g(\alpha_h)\ ,$$
which imply that $\alpha_{gh} = \alpha_g g(\alpha_h)$. That $\alpha$ is algebraic can be seen as follows. Denote by $G', \Pi', \Pi'_x$ the extension of scalars of $G, \Pi, \Pi_x$  to $k'$. 
Restricting the base change of $(\ref{torsordef})$ to $\Spec k' \times \Pi' \subset \Pi'_x \times \Pi'$ and projecting onto $\Pi'_x$   defines an isomorphism 
$$1_x : \Pi'   \overset{\sim}{\To} \Pi'_x  \ .$$
It is given on points by $p\mapsto 1_x. p$. Then $\alpha$  comes from the morphism of schemes
$$  G' \To \Pi'_x    \overset{(1_x)^{-1}}{\To} \Pi' $$
where the first map is given by the action on $1_x$, i.e., $ g\mapsto g(1_x)$. In conclusion:
$$\alpha \in Z^1_{\alg}(G, \Pi)(k')\ .$$

  This cocycle can be viewed in three different ways (\S \ref{sect3views}):
  \begin{enumerate}
\item  as the morphism $\alpha: G' \rightarrow \Pi'$ as constructed above.
\item as an algebra homomorphism $\Or(\Pi')\rightarrow \Or(G')$ given by 
$$\Or(\Pi') \overset{(1^*_x)^{-1}}{\To} \Or(\Pi'_x) \overset{\Delta}{\To} \Or(\Pi'_x) \otimes_{k'} \Or(G') \overset{1_x \otimes \id}{\To} \Or(G')$$
where $\Delta$ is the coaction dual to the action of $G'$ on $\Pi'_x$. 
\item  Setting $R= \Or(G')$ in $(1)$ gives  a map   $\alpha : G'(\Or(G')) \rightarrow  \Pi'(\Or(G'))$. The image of the identity element is   a point  $\alpha(\id) \in \Pi'(\Or(G'))$.
\end{enumerate}
Changing the point $1_x$  has the effect of modifying the cocycle by a coboundary. Let $1'_x \in \Pi_x(k')$ be another point  and $\alpha'$ the associated cocycle. 
Then there exists a unique  $b \in \Pi(k')$ such that $1'_x = 1_x. b$, and 
$$\alpha'_g = b^{-1}  \alpha_g  g(b)\ .$$

\subsection{Weight filtrations}  \label{sectWeightFilt} In our examples, 
the affine rings
$\Or(\Pi_x),   \Or(\Pi) ,  \Or(G)$
are equipped with increasing exhaustive filtrations $W$, such that $W_{-1}=0$. These filtrations are compatible in that the coactions satisfy:
$$ \Delta:  W_n \Or(\Pi) \To W_n  \big( \Or(\Pi_x) \otimes \Or(\Pi)\big) $$
$$\Delta:  W_n \Or(\Pi_x) \To W_n  \big( \Or(\Pi_x) \otimes \Or(G) \big)$$
and similarly in the second equation with $\Pi_x$ replaced by $\Pi$. Furthermore, for each of these three affine rings, the filtration $W$ is compatible with
the algebra structure: 
$$W_m . W_n \subset W_{m+n}\ .$$
 Finally, the morphisms considered above are strict with respect to the filtration $W$. 
Let $1_x \in \Pi_x(k')$. Then repeating the argument of  \S\ref{sectTorsors} we find that 
$$   \Or(\Pi_x)  \overset{\Delta}{\To} \Or(\Pi_x) \otimes \Or(\Pi)  \overset{1_x \otimes \id}{\To} k' \otimes \Or(\Pi)$$ 
 preserves the filtration $W$ and induces an isomorphism  $k' \otimes \Or(\Pi_x) \cong k' \otimes \Or(\Pi)$.  By strictness, it defines an isomorphism  for every $n$:
$$ 1_x :  k' \otimes W_n \Or(\Pi_x) \overset{\sim}{\To} k'\otimes W_n \Or(\Pi)   \ .$$ 
Writing $\Pi' =\Pi\times_k k'$, and $G'=G\times_k k'$,  consider the chain of homomorphisms
$$W_n \Or(\Pi') \overset{1_x^{-1}}{\To} W_n \Or(\Pi'_x) \overset{\Delta}{\To} W_n \big( \Or(\Pi'_x) \otimes_{k'} \Or(G)\big) \overset{1_x \otimes \id}{\To} W_n \Or(G) $$
and we deduce that the cocycle associated to the point $1_x$ preserves $W$, i.e., 
$\alpha \in  Z^1_{\alg, W}(G, \Pi)(k')$. 
The point of this construction  is that if  $k$ is a field and $W_n \Or(G)$ is  of finite dimension for all $n$,  $Z^1_{\alg,W}(G, \Pi)$ will define a scheme by proposition \ref{propZ1Fscheme}.  By replacing $G$ with the quotient through which it acts on $\Pi$,  we can always assume that $G$ acts faithfully on $\Pi$. With this in mind, our criterion  for representablity will be satisfied  
in all the examples we shall need to consider.

 \begin{rem}
  The approach to proving representability of a functor of cocycles in \cite{Kim1} is different, and involves an inductive argument via long exact sequences.
 See also \cite{HainRemarks}. 
  \end{rem} 
 \subsection{Determinant equation}  Now let us suppose, as above, that $Z= Z^1_{\alg,F}(G,\Pi)$ is a scheme, for some filtration $F$. 
 Then we can apply the formalism of \S\ref{sect: det}. Let us spell this out explicitly in this situation. 
    We have a morphism  of schemes
    \begin{equation}  (c,g) \mapsto c_g: Z \times G  \rightarrow   \Pi 
      \nonumber  \end{equation} 
  whose dual is the morphism of algebras  (which we again denote by $\co$)
  \begin{eqnarray} \label{coOZ}
  \co: \Or(\Pi) \To \Or(G) \otimes_k \Or(Z) 
    \end{eqnarray} 
  This uses, of course, the fact that $Z$ is a scheme.   Given an $k'$-valued cocycle $\phi \in Z(k')$, for some commutative unitary $k$-algebra $k'$, the morphism of algebras
  $$\phi: \Or(\Pi)  \To  \Or(G)\otimes_k k'$$
  is obtained by composing  $(\ref{coOZ})$ with $\phi: \Or(Z) \rightarrow k'$, namely 
  \begin{equation}\label{phiwequation} \phi(v) = (\id \otimes \phi) \co(v)\  . 
  \end{equation} 
  Proceeding in the same manner as \S\ref{sect: det},  $(\ref{coOZ})$ extends to an $\Or(G)$-linear map 
  $$\co :\Or(G) \otimes_k  \Or(\Pi)\To     \Or(G) \otimes_k \Or(Z)$$ 
  and  by taking exterior powers, gives an $\Or(G)$-linear map
  $$\textstyle{\bigwedge^N}\co :\Or(G) \otimes_k  \textstyle{\bigwedge^N}\Or(\Pi)\rightarrow     \Or(G) \otimes_k  \textstyle{\bigwedge^N} \Or(Z)\ .$$

  In the same manner as lemma $\ref{lemdeteqn}$, we deduce the following proposition.
  
  \begin{prop}  \label{propdetZ} Suppose that $N= \dim_k \Or(Z)+1$ is finite, and let $k'$ be a commutative unitary $k$-algebra.  Given  any  set of $N$ cocycles 
  $$\phi_1, \ldots, \phi_N \in Z^1(G,\Pi)(k') \ ,$$ 
 the following determinant equation holds:
 \begin{equation} \label{detequationZ}  \det  \big( \phi_i(v_j) \big)_{1\leq i,j\leq N} =0 
\end{equation}
for all   $v_1,\ldots, v_N \in  \Or(\Pi)$, where   $\phi_i(v_j)$  is defined via $(\ref{phiwequation})$. 
  \end{prop}
 
  \begin{rem} Our scheme-theoretic formulation allows us to take $k'$ to be a ring other than $k$. It will turn out that integral points on curves will only give rise to $k$-valued cocycles. Considering a larger ring $k'$ gives considerable flexibility to the method and enables us to consider `virtual cocycles' constructed out of divisors \S\ref{sectVirtual}. 
  \end{rem}

 \section{Main argument : mixed Tate case}   \label{sectMainargument}
 
 Throughout we shall argue with $R$-points on affine group schemes without specifying the ring $R$. 
 All tensor products are over $k$ unless specified otherwise. 
 We compose paths in the functional sense, and let Tannaka groups act on fundamental groups on the left. 
 Reversing both conventions simultaneously  still leads to agreeable formulae, but changing one or other convention on its own does not. 

 At the risk of repeating some of the arguments in \S\ref{sectfinal}, we consider only the mixed Tate case in this section, since
this is one of the few  situations in which the results are known to hold unconditionally.

 \subsection{Points and cocycles}  \label{sect: cx} 
 Let $\Or_S$ be the ring of $S$-integers in a totally real number field $k$, and let  $(X,x,y)$  be such that its motivic fundamental group is  mixed Tate  unramified over $S$. 
 Let us fix a possibly tangential  base point  $0 \in X(\Or_S)$. For any other  (tangential) point  $x \in X_S$, let  ${}_x\Pi^{\dr}_0$ denote the  fundamental torsor of paths from $0$ to $x$ with respect to the canonical fiber functor. 
 Via the right-torsor structure  $(\ref{dRTorsor})$ 
\begin{equation} \label{torsor} \xpio \times \opio \To \xpio \end{equation}
the canonical path ${}_x 1_0 \in \xpio(\Q)$  defines an algebraic cocycle \S\ref{sectTorsors}
 $$c^x \in Z^1_{\alg}(\GdrS, {}_0\Pi_0^{\omega})(\Q) \ .$$
Now consider any quotient ${}_0 \Pi_0  \rightarrow \Pi$ in the category $\MT(\Or_S)$, and 
denote its de Rham  realisation  by $\Pi^{\dr}$.  Equivalently, 
the affine ring $\Or(\Pi^{\dr})$  is  a Hopf subalgebra of $\Or(\opio)$ and is stable under the right $\Or(\GdrS)$-coaction.  
Via the natural transformation 
$$Z^1_{\alg}(\GdrS,  {}_0\Pi_0^{\omega}) \To Z^1_{\alg}(\GdrS, \Pi^{\dr})$$
we obtain an algebraic cocycle:
\begin{equation} \label{cxinpi} c^x  \quad \in \quad  Z_{\alg}^1(\GdrS, \Pi^{\dr})(\Q) \ . 
\end{equation}

\subsection{Weight filtration} 
 The  fundamental groupoids  $ {}_0\Pi_0^{\omega}, {}_x\Pi_0^{\omega}  , \Pi  $  are equipped with weight filtrations satisfying the conditions of \S\ref{sectWeightFilt}. 
  The same is true of the affine ring $\Or(G^{\omega}_S)$, which satisfies 
  $$W_{-1}\Or(G^{\omega}_S)= 0 \qquad \hbox{ and } \qquad  \dim W_n(\Or(G^{\omega}_S))<\infty$$ 
  for all $n$. The first property was proved in \cite{NotesMot}, \S.3.4 (note that this is only true if the  weight filtration is defined, via the Tannakian formalism, with respect to the conjugation action of $G^{\omega}_S$ on itself) and the second follows from the fact that $S$ is finite. 
  It follows from \S\ref{sectWeightFilt} that the cocycle $c^x$ satisfies
 $$c^x \in Z^1_{\alg, W}(G^{\omega}_S, \Pi^{\dr})(\Q) \ , $$
 and furthermore  $Z^1_{\alg, W}(G^{\omega}_S,  \Pi^{\dr})$ is an affine scheme over $\Q$.

 \subsection{Weight grading}   We can be more precise still using the fact that the weight filtration actually splits in the canonical realisation.

  The subgroup $\G_m \leq \GdrS$ acts trivially on the path ${}_x 1_0$ $(\ref{TateCanpath})$.
Since the cocycle $c^x$ is defined (on points) by $ g( {}_x 1_0)  ={}_x 1_0 \times  c^x_g $, it follows that $c^x_g=1$ for all $ g\in \G_m$.

Denote the space  of left $\GdrS$-cocycles  trivial on $\G_m$ by 
$$Z_{\G_m}^1(  \GdrS, \Pi) = \mathrm{ker}  \Big( Z_{\alg, W}^1( \GdrS, \Pi) \To Z_{\alg, W}^1(\G_m, \Pi) \Big) \ .  $$
By proposition $\ref{propZ1Fscheme}$, this is an  affine scheme over $\Q$, since the map on the right is a morphism of affine schemes over $\Q$ (as $W_n \Or(\G_m)$ is finite-dimensional). 

Define $Z^1_{\alg, W}(\UdrS, \Pi)^{\G_m}$ to be the closed subscheme of $  Z^1_{\alg, W}(\UdrS, \Pi)$ whose points consists of cocycles satisfying the equation
 \begin{equation} \label{cocycleonU} c_{gug^{-1} }= g(c_u) \quad \hbox{  for  } g\in \G_m, u \in \UdrS \ .
 \end{equation} 
Then restriction to $\UdrS$ defines   an isomorphism of schemes 
$$  Z_{\G_m}^1(  \GdrS, \Pi)  \overset{\sim}{\rightarrow} Z^1_{\alg, W}(\UdrS, \Pi)^{\G_m}\ .$$
Indeed, if $c_g =1 $ for all $ g \in \G_m$ then the cocycle equation $c_{gh} = c_g g(c_h)$ implies that 
$(\ref{cocycleonU})$ holds.  Conversely, any    $\UdrS$-cocycle $c$  satisfying $(\ref{cocycleonU})$ defines an element $c'\in Z_{\G_m}^1(\GdrS, \Pi)$ by setting $c'_{ug} = c_u$ for  $u \in \UdrS$, $g \in \G_m$.

\subsection{Interpretation as a  homomorphism} We can think of   $c^x$ as an algebra homomorphism 
$c^x: \Or(\Pi) \rightarrow \Or(\GdrS)$. Its restriction to $\UdrS$ defines 
 a homomorphism 
\begin{equation} \label{cxashom}  c^x : \Or(\Pi) \To \Or(U^{\omega}_S) =   \Pe_S^{\uu}\ .
\end{equation} 
Equation $(\ref{cocycleonU})$ is equivalent to the fact that $c^x$ is $\G_m$-equivariant, where $\G_m$ acts via conjugation on $\Or(\U^{\omega}_S)$. In other words, $(\ref{cxashom})$ respects the weight-gradings on both sides.
 Explicitly,  the canonical de Rham paths define isomorphisms:
 \begin{equation}\label{Sect7.4canisom}  {}_0 \Pi^{\dr}_0 \cong {}_x \Pi^{\dr}_0   \qquad \hbox{ and } \qquad  \Or( {}_0 \Pi^{\dr}_0) \cong \Or(  {}_x \Pi^{\dr}_0) \ . \end{equation} 
The homomorphism $(\ref{cxashom})$ can then  be defined  by the composition of  maps
$$ c^x:\Or(\Pi) \subset \Or(\opio) \cong \Or(\xpio) \overset{\Delta}{\To} \Or(\xpio) \otimes \Or(\UdrS) \overset{{}_x 1_0 \otimes \id}{\To} \Or(\UdrS)\ , $$
where $\Delta$ is the coaction dual to the left action of $\UdrS$ on $\xpio$.

 \subsection{Interpretation via motivic iterated integrals} 
 We can also think of  $c^x$ as   a  $\Pe_S^{\uu}$-valued point on $\Pi$, i.e., $c \in \Pi(P_S^{\uu})$.   
 
Indeed, the map $\Delta$ in the previous subsection is nothing other than the restriction  to $\Or(\xpio)$ of the (unipotent) universal comparison isomorphism 
$$\Or(\xpio) \otimes \Pe_S^{\uu} \overset{\sim}{\To} \Or(\xpio) \otimes \Pe_S^{\uu}\ .$$
It follows that the composition 
$$   \Or(\xpio) \overset{\Delta}{\To} \Or(\xpio) \otimes \Pe^{\uu}_S \overset{{}_x 1_0 \otimes \id}{\To} \Pe^{\uu}_S$$
is given  for any $w \in  \Or(\xpio)$ as a matrix coefficient 
$ w \mapsto     [ \Or(\xpio), w, {}_x 1_0]^{\uu}. $
\begin{defn}  \label{defnIuu} For any $w \in \Or(\opio)$ let us denote by 
 \begin{equation} I^{\uu}(x;w;0)= [ \Or(\xpio), w, {}_x 1_0]^{\uu} 
 \end{equation} 
where $w$ is viewed in $\Or(\xpio)$ via the isomorphism $(\ref{Sect7.4canisom})$. It could be  called a   `unipotent de Rham motivic iterated integral'  of $w$ from $0$ to $x$.
\end{defn} 
It follows that $c^x \in \Pi(\Pe^{\uu}_S)$ is the homomorphism 
\begin{eqnarray} 
\Or(\Pi)  & \To &  \Pe^{\uu}_S \\
w  &\mapsto&  I^{\uu}(x;w;0)\ . \nonumber 
\end{eqnarray} 
 Viewed yet another way, $c^x$ is the image of the point ${}_x1_0$ under the composition 
 $$\xpio(\Pe^{\uu}_S) \overset{\sim}{\To} \xpio(\Pe^{\uu}_S) \overset{(\ref{Sect7.4canisom})  }{\cong} \opio(\Pe^{\uu}) \To \Pi(\Pe^{\uu}_S)$$
where the first isomorphism is the (unipotent) universal comparison isomorphism.
 The reason we  work with  unipotent instead of de Rham periods is because
  the unipotent version of $2\pi i$ is trivial, which is not true for its de Rham version.

    \subsection{Determinant} Let us denote   $ Z^1_{\G_m}(\GdrS, \Pi^{\dr})$ simply by $Z$. It is a scheme. Its affine ring is graded by the action of $\G_m$ induced by the left action of $\G_m$ 
    on $Z$ given by left multiplication  $h\mapsto gh : \GdrS \rightarrow \GdrS$ for $ g\in \G_m$.

  \begin{thm} \label{mainthm} Let $n\geq 0$. Let  us write
 $N= \dim \Or(Z)_{2n}$   and suppose that 
  $$
 \dim \gr^W_{2n} \Or(\Pi^{\dr}) >  N \ .
$$
Let  $w_1, \ldots, w_N$ be (linearly independent) elements in $\gr^W_{2n} \Or(\Pi^{\dr})$. Then any $N$ points $x_1, \ldots, x_N \in X_S$ satisfy the equation 
 \begin{equation}
 \label{detinthm} 
  \det \big(  I^{\uu}(x_j;w_i;0)\big)_{1\leq i,j \leq N} =0 \ .
  \end{equation}
   \end{thm} 
   \begin{proof} This is the weight-graded version of proposition \ref{propdetZ}.   
  The map 
\begin{eqnarray}
\co : \Or(\Pi^{\dr})   & \To &    \Pe_S^{\uu} \otimes \Or(Z)   \nonumber \\
w & \mapsto & (c\mapsto w(c_u)) \nonumber 
\end{eqnarray} 
where $u  \mapsto w(c_u) \in \Pe_S^{\uu}$ is viewed as a morphism from $\UdrS$ to $\A^1$. 
 Since  $ \Or(\Pi^{\dr}) $, $\Pe_S^{\uu}$ are  graded by weight and have even degrees, and cocycles in $Z$ respect the gradings, the definition in 
\S\ref{sect: grad} provides   a grading $\Or(Z) = \bigoplus_n \Or(Z)_{2n}$ on $\Or(Z)$ such that  
$$
\co : \gr^W_{2n}  \Or(\Pi^{\dr})  \To  \gr^W_{2n} (\Pe_S^{\uu} )  \otimes \Or(Z)_{2n}     \  . 
$$
The grading on $\Or(Z)$ is induced then, either by the action of $\G_m$ on $\Or(\Pi^{\dr})$, or the conjugation action of $\G_m$ on $\UdrS$. These 
coincide by  $(\ref{cocycleonU})$.  The conjugation action of $\G_m$ on $\UdrS$ is equivalent to the left action of $\G_m$ on $\UdrS \rtimes \G_m = \GdrS$. 
 Every cocycle $(\ref{cxinpi})$ factorizes through this map.
By the argument of proposition   \ref{propdetZ}  we deduce that  
$$   \det  \big( w_i(c^{x_j}) \big)_{1\leq i, j \leq N} = 0\ . $$
     By definition    \ref{defnIuu},  the coefficient  of $w_i(c^{x_j})$ is precisely $I^{\uu}(x_j;w_i;0)$.
  \end{proof} 

   \subsection{Discussion}  \label{sectComments}

 \begin{enumerate} 
\item  The hypothesis  $
 \dim \gr^W_{2n} \Or(\Pi^{\dr}) >  \dim \Or(Z)_{2n}$ of the theorem provides, via remark  \ref{remuniveqn},  the existence of a non-explicit function
 $$\sum p_i  I^{\uu}(x; w_i;0) =0 $$
 for all $x \in X_S$, where $p_i \in \Pe^{\uu}_S$ and $w_i \in \Or(\Pi)$ are of weight $2n$.  Since it is not known  how to construct the 
 elements in $\Pe^{\uu}_S$ except in very special cases, this approach seems to be  impractical. An added complication is that the injectivity of the $p$-adic period homomorphism on $\Pe^{\uu}_S$ is completely open, so it is not known that the elements $p_i$  have non-vanishing $p$-adic periods. There are multiple reasons, therefore, for  preferring the explicit approach via  $(\ref{detinthm})$.
  
\item  Equation $(\ref{detinthm})$ is a  `motivic' Coleman function which vanishes on $\textstyle{\bigwedge^N} X_S$. We can retrieve a  function vanishing on $X_S$ as follows.   Suppose  there exist $N-1$ points $x_1,\ldots, x_{N-1}$ on $X_S$. Write $x_N=x$.  A row expansion of $(\ref{detinthm})$ yields 
 \begin{equation} \label{Colemaneqn} \sum_{i=1}^{N-1}   p_{w_i}  I^{\uu} (x; w_i;0)=0\ ,\  
 \end{equation} 
 which holds for all integral points $x\in X_S$. The unipotent period  $p_{w_i} \in \Pe_S^{\uu}$ is the determinant of  the minor 
 of $(  I^{\uu}(x_j;w_i;0)  )_{1\leq i, j \leq N}   $ obtained by deleting 
 row $N$ and column $i$.  The equation $(\ref{Colemaneqn})$ could be  trivial if every $p_{w_i}$ were to  vanish. 
 But in that case, the equation $p_{w_N}=0$ is equivalent to an equation
 $$  \det  \big(  I^{\uu}(x_j;w_i;0)\big)_{1\leq i, j\leq N-1} =0 $$
 which vanishes for every set of $N-1$ points $x_1,\ldots, x_{N-1} \in X_S$.  Continuing in this way, we reach the following conclusion.
 
 \begin{cor} Either $\big| X_S \big|<N$
  or every point $x\in X_S$ satisfies   an  equation of the form $(\ref{Colemaneqn})$  where not all coefficients $p_{w_i} \in \Pe^{\uu}_S$  are identically zero. 
 \end{cor}
 
 \item Let $p< \infty$ be prime not in $S$. Denote the $p$-adic period of  a motivic  unipotent de Rham motivic  iterated integral (definition \ref{defnIuu}) by
 $$I^p (x;w;0) = \per_p \big( I^{\uu}(x;w;0)\big) \ .$$
 Since $I^{\uu}(x;w;0)$ is the unipotent de Rham period $ [ \Or(\xpio), w, {}_x 1_0]^{\uu} $, we have
 $$  I^p (x;w;0)   =   {}_x 1_0 ( \overline{F}_p w)$$
 where $\overline{F}_p$ is the normalised Frobenius acting on $\Or(\xpio)\otimes_{\Q} \Q_p$.  
 It is  the single-valued $p$-adic iterated integral of $w$ from $0$ to $x$. 
  
 \begin{cor}  \label{corfinite}  Assume the above set-up, and  the conditions of theorem $\ref{mainthm}$. 
 For any set of  $N$ points $x_1, \ldots, x_N \in X_S$, we have 
   \begin{equation}    \label{detIp} \det \big(   I^{p}(x_j;w_i;0)\big)_{1\leq i,j \leq N} =0 \  .
   \end{equation}
  In particular, if $|X_S|\geq N$ then there exists a non-trivial $p$-adic analytic function on $X(\Q_p)$
 which vanishes on the image of $X_S$. 
 \end{cor} 
 \begin{proof} By  the same minimality argument as $(2)$,  it is enough  to show that the single-valued $p$-adic iterated integral 
 $x\mapsto I^p(x;w; 0)$ is a non-zero $p$-adic analytic function. In fact,  they are linearly independent for linearly independent $w$,  
 which  follows from the differential equation ($p$-adic KZ equation) they satisfy (\cite{Fu}, \S3.2).
 \end{proof}
\item Under the conditions of the  previous corollary,  the set  $|X_S|$ is finite when $X$ is a curve. Even in this case, the conditions for the motivic fundamental groupoid of $X$ to be mixed Tate are highly restrictive: for example $X = \Pro^1 \backslash \Sigma$ where $\Sigma$ is a finite non-empty set of points unramified outside $S$.  
 Note also that $(\ref{detIp})$  does not depend on $S$, but only its cardinality $|S|$.
 \end{enumerate}

\section{Dimension formulae}
We now  compute some formulae for the dimensions  of the weight-graded  pieces of $\Or(\Pi)$ and $\Or(Z)$ in theorem \ref{mainthm}   via two different methods.

\subsection{Upper bound on the space of cocycles} \label{sectUBSpace} The first method is via the Hopf algebra interpretation \S\ref{sect3views} (2) of cocycles. 
Let $\GdrS, \Pi^{\dr}$ be as in theorem \ref{mainthm}, and let   $c \in Z=Z_{\G_m}^1(\GdrS, \Pi^{\dr})$.  It defines an algebra homomorphism
$$c: \Or(\Pi) \To \Or(\UdrS) = \Pe^{\uu}_S$$
which respects the weight gradings. Since  $c$ is multiplicative, it is determined by its value on a choice of representatives $Q \Or(\Pi^{\dr})$ of indecomposable elements in $\Or(\Pi^{\dr})$. 
 Let $\Delta' : \Pe^{\uu}_S \rightarrow \Pe^{\uu}_S \otimes_{\Q} \Pe^{\uu}_S$ denote the reduced coproduct $\Delta' = \Delta - \id \otimes 1 - 1 \otimes \id$, where $\Delta$ is dual to the multiplication in $\UdrS$. The space of primitive elements in $\Pe^{\uu}_S$ is
 $$\mathrm{Prim} \, \Pe^{\uu}_S = \ker \,  \Delta'\ . $$
 Let $w\in \gr^W_n \Or(\Pi^{\dr})$.  The value of $\Delta' c(w)$ is completely determined, via  the cocycle condition,  by  $c(w')$ for $w' $ of weight $< n$.  The indeterminacy of $c(w)$ is therefore given by an element of  $\gr^W_n\mathrm{Prim} \, \Pe^{\uu}_S$. Equivalently, for two such homomorphisms  $c_1, c_2$ which agree 
 on $W_{n-1} \Or(\Pi^{\dr})$ and satisfy $\Delta' c_1(w) = \Delta' c_2(w)$,  the difference $c_1-c_2$, restricted to $\gr^W_n \Or(\Pi^{\dr})$, defines an element of  
 $$ \mathrm{Hom} (\gr^W_n Q \Or(\Pi^{\dr}), \gr^W_n \mathrm{Prim} \, \Pe^{\uu}_S)$$
 where $\mathrm{Hom}$ denotes linear maps  of vector spaces.   Since the affine ring of the scheme of homomorphisms is the symmetric algebra on its dual \S \ref{sectcoactSchemeA}, it follows that $\Or(Z)$ has at most 
 $$\Big(\dim  \gr^W_n Q \Or(\Pi^{\dr}) \Big) \times \Big( \dim  \gr^W_n \mathrm{Prim} \, \Pe^{\uu}_S\Big) $$
 algebra generators in degree $n$.

 \subsection{Poincar\'e series}   Since all objects are graded with weights only in even degrees, we shall hereafter divide the weights by two.
 \begin{lem} Let $\ell_m = \dim_{\Q}  \, \gr^W_{2m} Q \Or(\Pi^{\dr}) $ be the dimension of the space  of algebra generators of $\Or(\Pi^{\dr})$ in weight $2m$. Then  
\begin{equation} \label{genseries1}
 \dim_{\Q}\,  \gr^W_{2n} \Or(\Pi^{\dr}) =  \hbox{coeff. of } t^n \hbox{ in } 
 \quad   \prod_{m \geq 1} {1 \over (1-t^m)^{\ell_m}}\ .
 \end{equation}
 \end{lem} 
 
 \begin{proof} The ring $\Or(\Pi^{\dr})$ is  a commutative graded Hopf algebra. By the Milnor-Moore theorem it is isomorphic to the graded  polynomial algebra on its generators. 
 \end{proof} 
 
 For all $n\geq 1$ we have
 $$\gr^W_{2n} \mathrm{Prim} \, \Pe^{\uu}_S  \cong \gr^W_{2n} H^1(\UdrS, \Q) \cong   \mathrm{Ext}^1_{\MT(\Or_S)} (\Q(0), \Q(-n)) \ . $$
 by, for example,  \cite{NotesMot}  \S6. 
 The dimensions of the right-hand vector space are given by Borel's theorem $(\ref{Borelrank})$.  Therefore $\gr^W_{2}   \mathrm{Prim} \, \Pe^{\uu}_S = |S|$. 
  For $m\geq 2$,  write
 $$r(m)  =   \gr^W_{2m} \mathrm{Prim} \, \Pe^{\uu}_S  = \begin{cases} r_1+r_2 \,\,\, \hbox{ if } m \hbox{ odd} \\ r_2  \qquad \,\,\, \hbox{ if }  m \hbox{ even} \ ,  \end{cases} $$
 where $r_1$ (resp. $r_2$) denotes the number of real (resp. complex) embeddings of $k$.

 \subsection{}  By the discussion of \S \ref{sectUBSpace}, we deduce that $\Or(Z)$ has at most 
 $r(m) \ell_m$
 algebra generators in weight $2m$, and therefore 
\begin{equation} \label{genseries2}
 \dim  \Or(Z)_{2n}  \leq \hbox{coeff. of } t^n \hbox{ in } 
 \quad  { 1\over (1-t)^{|S|\ell_1}}  \prod_{m\geq 2}  {1 \over (1-t^m)^{ r(m) \ell_{m} }} \ ,
\end{equation}
It turns out that this inequality  $(\ref{genseries2})$ is in fact an equality, essentially because $\UdrS$ has no $H^2$. 
In any case, a sufficient condition for theorem $\ref{mainthm}$ to apply is if, for some $n$,
$$ \Big( \hbox{coeff. of } t^n \hbox{ in } 
 \quad   \prod_{m \geq 1} {1 \over (1-t^m)^{\ell_m}}  \Big) \quad \geq \quad \Big( \hbox{coeff. of } t^n \hbox{ in } 
 \quad  { 1\over (1-t)^{|S|\ell_1}}  \prod_{m\geq 2}  {1 \over (1-t^m)^{ r(m) \ell_{m} }} \Big)$$
This clearly cannot happen if all $r(m)\geq 1$. Therefore  we must assume that 
 $r_2=0$, i.e., that $k$ is totally real.  When the previous inequality is satisfied, theorem $\ref{mainthm}$ can be applied. The quantity $N$ is the integer in the right-hand side of the inequality.

\subsection{} Nowhere have we assumed that the scheme $Z^1_{\G_m}(\GdrS, \Pi^{\dr})$, nor for that matter, $\Pi^{\dr}$  is finite dimensional (it is not in general). The argument works without this assumption. 
In the finite-dimensional case, we can apply the following lemma.

\begin{lem} \label{lemPS}  Consider a power series  with $M$ factors:
$$F(t) = \prod_{i=1}^{M} {1 \over (1-t^{\alpha_i})}  \in \Q[[t]]$$
where the $\alpha_i$ are strictly positive integers. Then the coefficients of $t^n$ in $F$ grow  polynomially 
in $n$ of degree $M-1$. \end{lem}  
\begin{proof} The coefficients of $F(t)$ are bounded above by those of 
$${1 \over (1-t)^{M}} = \sum_{n\geq 1} \binom{n+M-1}{M-1} t^n $$
since the  geometric series  expansion of $(1-t^{\alpha_i})^{-1}$ has non-negative coefficients which are termwise bounded above by those of 
$(1-t)^{-1} = 1+ t+ t^2+ \ldots .$  A lower bound for the coefficients of $F(t)$
is likewise  given  by an expansion of $ (1-t^{\alpha})^{-M}$, where $\alpha$ is the least common multiple of $\alpha_1,\ldots, \alpha_M$.  By the previous formula, the coefficient of $t^{\alpha n}$ in this series is $ \binom{n+M-1}{M-1} $, which is again polynomial in $n$ of degree $M-1$. 
\end{proof}

Suppose that there exists $w>0$ such that
\begin{equation} \label{assumelm} \ell_m  = 0 \quad \hbox{ for all } \quad  m> w\ .
\end{equation} 
This happens if we replace $\Pi$ with the quotient $\Pi/ W_{-2w-1} \Pi$, whose affine ring  is the Hopf subalgebra of $\Or(\Pi)$ generated by 
$W_w \Or(\Pi)$. In fact, $(\ref{assumelm})$ is equivalent to $W_{-2w-1} \Pi =1$.   With this assumption,   the generating series $(\ref{genseries1})$ and $(\ref{genseries2})$ have finitely many factors and we can apply the lemma to deduce the following

\begin{cor}Suppose $k$ is totally real with $r_1$ real embeddings. A 
sufficient condition for some coefficient $t^n$ in $(\ref{genseries1})$ to exceed that of $(\ref{genseries2})$ is 
\begin{equation} \label{ellineq} \ell_1 + \ldots + \ell_w > |S| \ell_1 + r_1 (\ell_3 + \ell_5 +  \ldots + \ell_r) 
\end{equation} 
where $r $ is the largest odd integer with  $3\leq r \leq w$. In this situation,  \ref{mainthm} applies.
 \end{cor} 
This is the  theorem of \cite{Hadian}.
 The assumption $(\ref{assumelm})$ implies that  $\Or(\Pi^{\dr})$ and  $\Or(Z)$ have finite transcendence degree and so $\Pi$ and $Z$ are both  finite-dimensional with 
\begin{eqnarray}  \label{Z1dimineq}
 \dim \Pi  &= & \ell_1 + \ldots + \ell_w \ ,  \nonumber \\
   \dim  Z^1_{\G_m}(\GdrS,  \Pi)  &\leq & |S| \ell_1 +r_1(\ell_3  + \ell_5 + \ldots+ \ell_r)   \ .
 \end{eqnarray} 
The latter inequality is in fact an equality. We will show below that  $(\ref{ellineq})$ is implied by the inequality  $\dim \Pi > \dim Z^1_{\G_m}(\GdrS,  \Pi)$. We emphasize, however, that our approach does not require finite-dimensionality of either scheme.

 \subsection{Dimensions via homological algebra}\label{sectskipped}  We give a second approach to computing dimensions  using  standard methods of homological algebra. The   argument is formally similar to \cite{KimAlb}.

 \begin{lem}  \label{lemZ1toH1} The natural transformation of functors
$$Z^1_{\G_m}(\GdrS, \Pi) \rightarrow H^1_{\G_m}(\GdrS, \Pi)  $$
is bijective, where $H^1_{\G_m}(\GdrS,\Pi) = \ker \big( H_{\alg}^1(\GdrS,\Pi) \rightarrow H_{\alg}^1(\G_m,\Pi) \big)$ and $H^1_{\alg}$  is the functor from commutative unitary $k$-algebras to pointed sets whose $R$-points are equivalence classes of elements in $Z^1_{\G_m}(\GdrS, \Pi)$ modulo boundaries.  
\end{lem}
\begin{proof} A class  in $H_{\G_m}^1(\GdrS,\Pi)(R)$ can be  represented by 
$a: \GdrS \times R \rightarrow \Pi\times R$  which defines an algebraic cocycle. We claim that it has a unique representative whose restriction to $\G_m\times R$ is trivial.  
For the existence,  the triviality of  the class of  $a$ restricted to  $\G_m \times R $ means that  there exists
$c \in \Pi(R)$ such that 
$a_g = c g(c^{-1})$ for all $g\in \G_m\times R$. By replacing $a_g$ with  $g\mapsto c^{-1} a_g g(c)$, we can assume that $a$ is trivial on $\G_m\times R$. 
For the uniqueness,  consider another   representative $a'$ where $a'_g = b^{-1} a_g g(b)$ for some $b\in \Pi(R)$, and such that $a'$ is also  trivial on $\G_m \times R$. It follows that
$$ g(b) = b \qquad \hbox{for all } g \in \G_m (R)$$
and hence $\log(b) \in \Lie \Pi(R)$ is   $\G_m(R)$-invariant. The logarithm  exists since $\Pi$ is pro-unipotent.  But $(\Lie \Pi(R))^{\G_m(R)}$ is a  quotient of $(\Lie \opo(R))^{\G_m(R)}$. Since $\G_m(R)=R^{\times}$ contains  $\Q^{\times}$ and  the latter acts non-trivially on all Tate objects $\Q(n)_{dR}$ for all $n<0$  it follows that  $(\Lie \opo(R))^{\G_m(R)}$ is a quotient of  $\gr^W_0 \Lie \opo(R) = 0$.  Therefore  $\log(b)$ vanishes, and hence $b$ is the unit element.   
\end{proof}

In the proof we showed  that   $\Pi^{\dr}(R)$ has trivial $\G_m(R)$-invariants for all $R$.

Let $L^1 \Pi = \Pi$ and $L^{n+1} \Pi = [ \Pi , L^n \Pi]$ for  $n\geq 1$ be the lower central series of $\Pi$. Set $\Pi_n = \Pi/L^{n+1} \Pi$ and consider the short exact sequence
of affine group schemes 
$$ 1 \To \gr^n_{LCS} \Pi \To \Pi_{n} \To \Pi_{n-1} \To 1\ .$$
 Taking non-abelian cohomology with respect to  $G=$ $G^{\omega}_S$ or $\G_m$ leads to a sequence
 $$  \Pi_{n-1}^{G} \To H^1(G, \gr^n_{LCS} \Pi) \To H^1(G, \Pi_{n}) \To H^1 (G, \Pi_{n-1}) \To H^2(G, \gr^n_{LCS} \Pi)\ ,$$
which, on taking points, is a long exact sequence of non-abelian cohomology sets.  
We  showed that $\Pi_{n-1}^{G}$ is trivial since $G$ contains $\G_m$  and $\Pi^{\G_m}=1$.
Since $\gr^n_{LCS} \Pi$ is abelian,  and since ordinary group cohomology can be computed using cocycles,
$$H^i(G,  \gr^n_{LCS} \Pi) = \mathrm{Ext}^i_{G-\hbox{mod}} (\Q,  \gr^n_{LCS} \Pi)\ .$$ 
The latter is trivial for $i \geq 2$ if  $G= G^{\omega}_S$  by \S \ref{sectStructure}, and vanishes for $i\geq 1$ if $G= \G_m$. We deduce by induction on $n$
and the fact that $\Pi_0=1$ that $H^1(\G_m, \Pi_n)=0$ for all $n$ and hence $H^1(\G_m, \Pi)$ is trivial. 
It follows that $H^1_{\G_m}(G^{\omega}_S, \Pi) = H^1(G^{\omega}_S, \Pi)$.  By lemma  \ref{lemZ1toH1}, the former defines a representable functor and hence a scheme. 

 Therefore we obtain the sequence
$$0 \To \mathrm{Ext}_{\MT_S}^1(\Q,  \gr^n_{LCS} \Pi) \To H^1(G^{\omega}_S, \Pi_{n}) \To H^1(G^{\omega}_S, \Pi_{n-1}) \To  H^2(G, \gr^n_{LCS} \Pi) \ .
$$
This means that the fibers of the morphism $H^1(G^{\omega}_S, \Pi_{n}) \rightarrow H^1(G^{\omega}_S, \Pi_{n-1})$  of schemes are principle $\mathrm{Ext}_{\MT_S}^1(\Q,  \gr^n_{LCS} \Pi)$-spaces. It follows from properties of the dimension  of schemes that 
$$\dim H^1(G^{\omega}_S, \Pi_{n}) \leq \dim  H^1(G^{\omega}_S, \Pi_{n-1})+ \dim \mathrm{Ext}_{\MT_S}^1(\Q,  \gr^n_{LCS} \Pi)\ .$$
This  inequality does not use the  surjectivity in the previous sequence, and therefore applies in the more general situation of \S \ref{sectfinal}. 
It is in  fact an equality in the present example since the final term is trivial.
Assume that $\Pi = \Pi/W_{m+1}$ for some $m$. In particular, the lower central series terminates.  It follows by induction that 
$$\mathrm{dim}\,  H^1(G^{\omega}_S, \Pi) \leq \sum_{n \geq 0} \dim \mathrm{Ext}^1_{\MT_S}(\Q, \gr^n_{LCS} \Pi)\ .$$
If we write $\gr_{LCS} \Pi \cong  \bigoplus_w \Q(-w)^{\ell_w}$  (the lower central series filtration coincides with the weight-filtration)  then by lemma \ref{lemZ1toH1} we deduce that
$$\mathrm{dim}\,  Z_{\G_m}^1(G^{\omega}_S, \Pi) \leq\sum_n \ell_n \dim \big(\mathrm{Ext}^1_{\MT_S}(\Q(0), \Q(-n))\big)$$
which gives $(\ref{Z1dimineq})$ by   $(\ref{Borelrank})$.

 \section{Application: the unit equation} \label{sectunitexample}
  Let $k=\Q$, $X= \Pro^1 \backslash \{0,1,\infty\}$, $\Or_S = \Z[{1\over S}]$ where  $S$ is a  finite set of rational primes,  and let $0$ be the tangential base-point $1$ at $0$.   The set of integral points $X_S$ are the solutions to the equation 
  $ u + v= 1$
  where $u, v $ are $S$-integral. We shall also consider certain tangential basepoints on $X_S$ as solutions to this equation.
  For background on the motivic fundamental group, see \cite{DeP1}, \cite{DeGo}, \cite{BrICM}.
 \subsection{Unit equation}  
 First let $\Pi= {}_0 \Pi_0$  be the entire fundamental group.  Since $k=\Q$, the canonical fiber functor  $\omega$ coincides with the de Rham fiber functor  $\omega_{dR}$. Then 
 $\Or( \Pi^{\dr}) = T^c( \Q e_0 \oplus  \Q e_1 ) $
 is the graded vector space spanned by words in two letters $e_0, e_1$.  It is equipped with the shuffle product $\sha$ and deconcatenation coproduct whose action on generators is given by the formula
 $$\Delta^{\mathrm{dec}} (e_{i_1} \ldots e_{i_n}) = \sum_{k=0}^n e_{i_1} \ldots e_{i_k} \otimes e_{i_{k+1}} \ldots e_{i_n}  \ . $$
 The augmentation  $\varepsilon$ is projection onto the empty word.
  An element of $\Pi^{\dr}(R)$ is an invertible   formal power series in non-commuting variables $R\langle \langle e_0, e_1 \rangle \rangle$ which is group-like with respect to the completed coproduct for which $e_0,e_1$ are primitive.

  Any point  $x\in X_S$ defines an algebraic cocycle $c^x \in \Pi(\Pe^{\uu}_S)$. It is a homomorphism $\Or(\Pi^{\dr}) \rightarrow \Pe^{\uu}_S$ of graded algebras and is  given by the generating series 
 $$ L^{\uu}(x) = \sum_{w}  w \, \Li^{\uu}_w(x)$$
 where we use the notation 
 $$\Li^{\uu}_w(x) := I^{\uu} (x; w; 0)$$
  for  the unipotent iterated integral from $0$ to $x$ of the word  $w$, where 
  $e_0$ stands for   ${dz \over z}$ and $e_1$ for ${dz \over 1-z}$.   In particular, if $-1_1$ is the tangent vector $-1$ at $1$, then 
 $L^{\uu}(-1_1) = \sum_w w \zeta^{\uu}(w)$ is the unipotent version of the motivic Drinfeld associator \cite{BrICM}, (2.10).

One has $\dim \gr^W_{2n} \Or(\Pi) = 2^n$. An explicit  set of algebra generators for $\Or(\Pi)$ are given by the set
of Lyndon words in $e_0,e_1$, with respect to the ordering $e_0<e_1$. By Witt's inversion formula, any set of algebra generators in weight $w$ number
$$\ell_w = {1 \over w}  \sum_{d|w} \mu\big({w\over d}\big) 2^d$$
where $\mu$ is the M\"obius function.  In particular, $\ell_1=2$, since $e_0$ and $e_1$ are Lyndon words.  Since $r_1=1$, equation $(\ref{ellineq})$ is equivalent to the inequality
$$\ell_2 + \ldots+\ell_{2k} > 2 |S|-2 $$ which is easily satisfied for all $k$ sufficiently large, since $\ell_w \sim 2^w$. 
The conditions of theorem \ref{mainthm} are satisfied for some $\Pi/W_{-k-1} \Pi$ and it  follows that $X_S$ is finite.   
 
 \subsection{Depth one quotient} \label{sectDepthonequot}
  Rather than considering all iterated integrals,
 it is enough only to consider the unipotent versions of the classical polylogarithms:
 \begin{eqnarray} \log^{\uu}(x)  &=  &  \Li^{\uu}_{e_0} (x) = I^{\uu}(x;e_0;0)   \\
  \Li_n(x) & = &  \Li^{\uu}_{e_1e_0^{n-1} } (x) = I^{\uu}(x;e_1\underbrace{e_0\ldots e_0}_{n-1};0)   \nonumber 
  \end{eqnarray}
  Denote their images under $\per_p$ with a superscript $p$ instead of $\uu$. 
 Therefore let $\Pi$ be the length $2k$,  depth 1  quotient of ${}_0 \Pi_0$. Its affine ring 
 $\Or(\Pi^{\dr}) \subset \Or( {}_0 \Pi^{\dr}_{0})$
 is the Hopf subalgebra generated by words of degree at most $1$ in the letter $e_1$ and of length $\leq 2k$.  The latter are  stable under deconcatenation. Its indecomposable elements with respect to the shuffle product are given by the $2k+1$ words 
 $$\{ e_0 \  , \   e_1  \ , \ e_1e_0 \ , \  \ldots \ , \  e_1 e_0^{2k-1}\} \  . $$
  These are Lyndon words for the ordering $e_0<e_1$.  Its Poincar\'e series is thus $(\ref{genseries1})$
 $$ \sum_{ n \geq 0 } d_n t^n   =   {1 \over (1-t)^2 (1-t^2) (1-t^3) \ldots (1-t^{2k}) }$$
 where 
 $d_n = \dim_{\Q} \gr^W_{2n}\Or(\Pi^{\dr})$, and  $\Or(Z)$ Poincar\'e
 series $(\ref{genseries2})$
 $$ \sum_{ n \geq 0 } c_n  t^n  =  {1 \over (1-t)^{2 |S|} (1-t^3) (1-t^5) \ldots (1-t^{2k-1}) }$$
 where $c_n =   \dim \Or(Z)_{2n}$. 
 Therefore by applying lemma \ref{lemPS} a necessary and sufficient condition for some $d_n$ to exceed $c_n$ is  $2k+1 > 2|S|+k-1$, i.e., 
\begin{equation} \label{kSinequality}   k > 2 (|S|-1)\ . \end{equation}
 Therefore, for   any $|S|$, and  integer $k$ satisfying $(\ref{kSinequality})$,  or possibly infinite, let $w$ be minimal such that $d_w>c_w$. Define  
 \begin{equation}\label{Nkwkdefn}  N(k,|S|) = c_w+1 \qquad \hbox{ and } \qquad w(k,|S|) = w \ .
 \end{equation}
 Theorem \ref{mainthm} implies the following corollary.
  \begin{cor} Let $k > 2 (|S|-1)$ as above, and write $N= N(k,|S|)$. Then for all $N$-tuples $x_1,\ldots, x_N \in X_S$,  the equation 
  $$\det (P_{j}(x_i)) =0 $$
  holds,  where $P_{1},\ldots, P_{N}$ denotes any set of distinct monomials of the form
 \begin{equation} \label{logLi}  (\log^{\uu}(x))^{r_0}   \prod_{i=1}^k \big( \Li_i^{\uu}(x) \big)^{r_i} 
 \end{equation} 
where $r_i \geq 0$ are integers, and the total weight is $w(k,|S|)$. This last condition is equivalent to  $r_0 +r_1 +2r_2 +\ldots + kr_k = w(k, |S|)$.

 Let $p\notin S$ be a finite prime. Then  in particular
 $$\det (\per_p P_{j}(x_i)) =0\ . $$
 \end{cor} 
 The corollary states that in order to detect integral points $X_S$, one is only required to compute the $p$-adic polylogarithms  $\log^p(x), \Li_1^p(x), \ldots, \Li_{2k}^p(x)$.  Unfortunately, it seems  that  taking the single-valued period at the infinite prime gives the zero equation, and hence no information about integral points. This is  because
 the single-valued versions of $\Li_{n}(x)$, for $n$ even, vanish  on $X(\R)$. 
  
  \subsection{Coproduct}
  There is an explicit formula for the coaction on unipotent iterated integrals in this case which is equivalent to a formula due to Goncharov \cite{GG}. It can be deduced directly by dualising \cite{BrICM} an older formula of Ihara's for the action of the action of the absolute Galois group on the $l$-adic completion of the fundamental group of the projective line minus three points. 
  
   Let  $x\in X_S$, and let  $\Delta: \Or(\xpio)  \rightarrow  \Or(\xpio) \otimes_{\Q} \Pe^{\uu}_S$ denote the right coaction dual to the left  action of $\UdrS$ 
  on $\xpio$. Let us denote by  $D = (\id \otimes \pi ) \Delta$, where $\pi$ denotes the  projection  modulo products:
  $$ \pi: \Pe^{\uu}_S \To   { (\Pe^{\uu}_S)_{>0}  \over  ( \Pe^{\uu}_S)^2_{>0}}\ .$$
  The `infinitesimal' coaction  is obtained via  the following formula.
  \begin{thm} 
    For any $a_1,\ldots, a_{N+1} \in \{0,1\}$, where $a_{N+1}=0$,  we have 
  \begin{multline}  \label{Deltafull} D \, I^{\uu}( x; a_{1}\ldots a_N ;a_{N+1} )  =  \\
  \sum_{1\leq i<j \leq N+1}  I^{\uu}(x;  a_1 a_2\ldots a_i  a_{j} \ldots a_{N-1} a_N ; a_{N+1} ) \otimes \pi ( I^{\uu}(a_i; a_{i+1} \ldots a_{j-1};  a_j))  \\
    + \sum_{1\leq j \leq N+1}  I^{\uu}(x;  a_j a_{j+1} \ldots a_{N-1}a_N;   0) \otimes \pi ( I^{\uu}(x;  a_1 a_2\ldots a_{j-1};   a_{j}))
  \end{multline}
where the right-hand term in the second (resp. third)  line is suitably interpreted \cite{BrICM}, \S2  as a unipotent period of $\Or({}_a \Pi^{\dr}_b)$  (resp. $\Or({}_x \Pi^{\dr}_b)$) where $a, b \in \{\tone_0, -\tone_1\}$ 

    \end{thm} 
    
    The terms in the right-hand side of the tensor in the  second line of $(\ref{Deltafull})$ are unipotent versions of motivic multiple zeta values (denoted by $\zeta^{\mathfrak{a}}$ in \cite{BrDec}).   
     
  \begin{cor}
  For all  $x\in X_S$,   
  $$D \,\Li^{\uu}_N(x) =  \Li^{\uu}_{N-1}(x) \otimes   \pi ( \log^{\uu}(x) ) + 1 \otimes   \pi ( \Li^{\uu}_{N}(x)  )  \ .$$ 
  \end{cor} 
 \begin{proof}
  Set $a_1= a_2=\ldots =a_{N-1}=0$ and $a_N=1$ in $(\ref{Deltafull})$.  Use  the fact that  
   $  I^{\uu}(a_i; a_{i+1} \ldots a_{j-1};  a_j)$ vanishes whenever $a_i=a_j$,  or all $a_{i+1} = \ldots = a_{j-1}=0$ to deduce that all terms in the second line of $(\ref{Deltafull})$ vanish. 
 By the shuffle product formula, $I^{\uu}(x; 0;a_j)$ is a power of $I^{\uu}(x;0;a_j)$, so  only the terms $j=1, 2$ in the third line of  $(\ref{Deltafull})$ survive.
  \end{proof} 
 
 By exponentiating $D-1 \otimes \pi$, we retrieve the formula for the full coaction:
 \begin{equation} \label{DeltaLiuu} \Delta \,\Li^{\uu}_n(x) = 1 \otimes \Li^{\uu}_n(x) + \sum_{i=0}^{n-1} \Li^{\uu}_{n-i}(x) \otimes {1 \over i!} \big(\!\log^{\uu}(x)\big)^i
 \end{equation}  
 where $\Li^{\uu}_0(x):=1$.   It is valid for all $x \in X(\Q)$.  This can also be proved by direct methods, for example using  $(2.3)$ of \cite{NotesMot},  without passing via the previous theorem. 

  \subsection{Depth one grading} \label{SectDepthonegrading}
 A peculiarity of the depth one quotient of $\Or({}_x \Pi_0^{\omega})$ is that the  associated depth-graded  is preserved by the action of
 $\UdrS$. More precisely, the coaction $(\ref{DeltaLiuu})$  restricted to the affine ring of the depth one quotient  of $\xpio$   happens to factor, by  $(\ref{DeltaLiuu})$, through the deconcatenation coproduct: 
  $$\Delta^{\mathrm{dec}} e_1 e_0^{n-1}  =  1\otimes e_1 e_0^{n-1} + \sum_{i=0}^{n-1} e_1e_0^{n-1-i} \otimes e_0^i \ . $$
  This can also be proved directly on noting that  the depth one quotient of $\xpio$ is a $\Pi^{\dr}$ torsor,  where 
  $\Pi^{\dr}$  was defined in \S \ref{sectDepthonequot}, 
 and the fact that $\UdrS$ acts trivially on $\Pi^{\dr}$. 
   Since the left-hand factors are  all of $D$-degree one (the $D$-degree is the degree in the letter $e_1$), it follows that the depth grading is preserved by $\UdrS$ in this situation.    A glance at  $(\ref{Deltafull})$ shows that this is absolutely not true in general.

Since $e_0$ and the $e_1e_0^n$ generate $\Or(\Pi)$ as an algebra, we can therefore grade the affine ring of the depth one quotient $\Or(\Pi)$ by $D$-degree and, by \S\ref{sect: grad}, $\Or(Z)$ too.

The (weight, $D$-degree)-bigraded version of the generating series $(\ref{genseries1})$ is
 \begin{equation} \label{Dseries1} 
  \sum_{ k,n \geq 0 } \dim_{\Q} \gr^W_{2n} \gr^D_k \Or(\Pi^{\dr}) s^k t^n   =   {1 \over (1-t)(1-st) (1-st^2) (1-st^3) \ldots (1-st^{2k}) } . 
  \end{equation} 
The bigraded version of  $(\ref{genseries2})$ is 
 \begin{equation} \label{Dseries2}  \sum_{ n \geq 0 }  \dim \Or(Z)_{2n,k } s^k  t^n  =  {1 \over (1-t)^{|S|} (1-st)^{|S|} (1-st^3) (1-st^5) \ldots (1-st^{2k-1}) } 
 \end{equation} 
 where $\Or(Z)_{2n,k } $ is the component of $\Or(Z)$ of weight $2n$ and $D$-degree $k$.

\subsection{Examples} 

\subsubsection{}   Take $|S|=1$, and $k=2$.  The coefficient of $st^2$ in $(\ref{Dseries2})$  is 2. Its coefficient in  $(\ref{Dseries1})$ is $1$. 
Let $N=2$. The two elements $e_1e_0$ and  $e_1 \sha e_0 = e_0e_1 +e_1e_0$ are linearly independent in $ \gr^W_4 \gr^D_1 \Or(\Pi^{\dr})$. The corresponding iterated integrals are $\Li^{\uu}_2(x)$ and $\Li^{\uu}_1(x) \log^{\uu}(x)$. 
Theorem \ref{mainthm} implies that for every pair of points $x_1,x_2 \in X_S$:
$$\det
\left(
\begin{array}{cc}
 \Li^{\uu}_2(x_1)   &  \log^{\uu}(x_1)  \Li^{\uu}_1(x_1)  \\
 \Li^{\uu}_2(x_2)   &  \log^{\uu} (x_2) \Li^{\uu}_1(x_2)    
\end{array}
\right) =0\ . $$
 If $S = \{p\}$, let $x_2$ be the tangential base point $-p$ at $\infty$. We have 
 \begin{equation} \label{Liuatpvec} \Li^{\uu}_n(x_2) = {1 \over n!} (\log^{\uu}(p))^n  \quad \hbox{ and } \quad \log^{\uu}(x_2) =\log^{\uu}(p) \ .
 \end{equation}
 Since $\log^{\uu}(p) \neq 0$, we deduce that for every $x\in X_S$ the equation 
 $$\det
\left(
\begin{array}{cc}
 \Li^{\uu}_2(x)   &  \log^{\uu}(x)  \Li^{\uu}_1(x)  \\
 {1\over 2}   & 1    
\end{array}
\right) =0 $$
 holds. Therefore for every $x\in X_S$, we deduce that
 $$\Li^{\uu}_2(x) - {1\over 2} \log^{\uu}(x) \Li_1^{\uu}(x) =0 $$
 which can be viewed as a `motivic' equation vanishing on integral points. 
  Taking the $p$-adic period, for some different $p \notin S$, gives the equation for all $x\in X_S$:
 $$\Li^{p}_2(x) - {1\over 2} \log^p(x) \Li_1^p(x) =0 $$
 This equation is due to Coleman. Note that it does not depend on $S$, so it  actually proves that $|\cup_{|S|=1} X_S|$ is finite, which is easy to verify by hand.
  \vspace{0.1in}

\subsubsection{}   Now, with $|S|=1$ still, set $k=2$. The coefficient of $st^4$ in $(\ref{Dseries2})$ (resp. $(\ref{Dseries1})$) is 2
(resp. 4). The three elements $e_1e_0e_0e_0$ and  $e_1e_0e_0 \sha e_0$ and $e_0^{\sha 3} \sha e_1$ are linearly independent in $ \gr^W_{8} \gr^D_1 \Or(\Pi^{\dr})$. Theorem \ref{mainthm} implies  that for every set of three points $x_1,x_2,x_3 \in X_S$ the following equation is satisfied: 
  $$\det
\left(
\begin{array}{ccc}
 \Li^{\uu}_4(x_1)   &  \log^{\uu}(x_1)  \Li^{\uu}_3(x_1)  &    (\log^{\uu}(x_1))^3  \Li^{\uu}_1(x_1)  \\
 \Li^{\uu}_4(x_2)   &  \log^{\uu} (x_2) \Li^{\uu}_3(x_2)  &    (\log^{\uu}(x_2))^3  \Li^{\uu}_1(x_2)  \\
 \Li^{\uu}_4(x_3)   &  \log^{\uu} (x_3) \Li^{\uu}_3(x_3)  &    (\log^{\uu}(x_3))^3  \Li^{\uu}_1(x_3)  \\
\end{array}
\right)=0 $$ Suppose that $S = \{2\}$, and let $x_3$ be the tangent vector $-2$ at infinity.  Then the point ${1\over 2} \in X_S$ and we can set $x_2 = {1\over 2}$. Again by  $(\ref{Liuatpvec})$, this gives the equation 
  $$\det
\left(
\begin{array}{ccc}
 \Li^{\uu}_4(x_1)   &  \log^{\uu}(x_1)  \Li^{\uu}_3(x_1)  &    (\log^{\uu}(x_1))^3  \Li^{\uu}_1(x_1)  \\
 \Li^{\uu}_4({1\over 2})   &  \log^{\uu} ({1\over 2} ) \Li^{\uu}_3({1\over 2})  &    (\log^{\uu}({1\over 2 } ))^3  \Li^{\uu}_1({1\over 2})  \\
 {1 \over 24}   &  {1\over 6} &    1  \\
\end{array}
\right)=0 $$
using the fact that $\log^{\uu}(2)$ is non-zero. 
Taking the $p$-adic period (replace all superscripts $\uu$ with $p$), and 
  multiplying out gives a Coleman function in weight 4 for $X_S$, which is  equivalent to the main theorem of \cite{DanCWewers}. 
  The method of that paper relied on a conjectural `exhaustion' property for motivic iterated integrals, and the conjectural non-vanishing of a $p$-adic zeta value.
  The approach taken above  is unconditional. 
  
\subsection{Remark on asymptotics} \label{sectasymptotics} Returning to the situation of \S\ref{sectDepthonequot}, let us fix $s= |S|$, and give a crude  estimate of  $N=N(s, \infty)$ as defined in $(\ref{Nkwkdefn})$ by approximating the generating series $(\ref{genseries1})$ and $(\ref{genseries2})$.  First of all,  one knows the rough asymptotics
 $$\hbox{coeff. of } t^n \hbox{ in } \quad \prod_{k\geq 0} {1 \over 1-t^{2k+1}} \quad \sim \quad {3^{3\over 4} \over 12}  {e^{{ \pi \over 3} \sqrt{3n}} \over n^{3 \over 4}}$$ 
 $$\hbox{coeff. of } t^n \hbox{ in }\quad \prod_{k\geq 1} {1 \over 1-t^{k}} \quad \sim \quad {3^{1\over 2} \over 12}  {e^{{ \pi \over 3} \sqrt{6n}} \over n}$$ 
using the more precise asymptotics of the partition function due to Hardy and Ramanujan \cite{HardyRam}. 
 The symbol $\sim$ means that the ratio of both sides tends to $1$ as $n$ goes to infinity.  The asymptotics of the corresponding series multiplied by  $(1-t)^{-m}$, 
  are  obtained from the above by  integrating:
$$\hbox{coeff. of } t^n \hbox{ in } \quad {1\over (1-t)^r} \prod_{k\geq 1} {1 \over 1-t^{2k+1}} \quad \sim \quad \Big({2 \sqrt{3n}\over  \pi}\Big)^r\, {3^{3\over 4} \over 12}  {e^{{ \pi \over 3} \sqrt{3n}} \over n^{3 \over 4}}$$

$$\hbox{coeff. of } t^n \hbox{ in } \quad {1\over 1-t} \prod_{k\geq 1} {1 \over 1-t^{k}} \quad \sim \quad {2^{1 \over 2} \over 4\pi }  {e^{{ \pi \over 3} \sqrt{6n}} \over \sqrt{n}}$$ 

\noindent 
Set $r=2s$. A short calculation  shows that  $w(s,\infty)$ is roughly of order $s^{2+2 \epsilon}$, and hence $N(s,\infty) \sim \exp(s^{1+\epsilon})$.
The condition $|X_S|<N$ in the dichotomy of corollary $\ref{corfinite}$ thus  compares unfavorably with a  theorem due to Evertse \cite{Evertse}, which  states that $|X_S| < 3 \times  7^{1+2|S|}$. It is curious that the asymptotics are fairly similar.   I do not know if considering the full fundamental group instead of  the depth one quotient would signficantly improve this estimate for $N$ or not.

\section{Virtual $S$-units}  \label{sectVirtual} Although  theorem \ref{mainthm} suffices to prove finiteness results, its defect is that one requires $N-1$ points on $X_S$ in order to construct an explicit Coleman function for the remaining points on $X_S$.  As shown in the previous paragraph, $N$ grows at least exponentially in $|S|$ and there may not exist the required number of  integral points on $X_S$ in the first place. This would seem to make an effective application of the method impractical. In this section we propose a remedy,  by introducing what could be called virtual $S$-units or virtual cocycles.

\subsection{Scheme of divisors} Let $R$ be a commutative $k$-algebra. Let us 
 denote the free $R$-module  of divisors on $X_T$  with coefficients in $R$ to be 
$$\mathrm{Div}(X_{T})(R) = \{\sum_{x\in X_T} n_x x \qquad \hbox{ where } n_x\in R  \}\ ,$$
where $T$ is a finite set of rational primes. 
\begin{prop} The functor $R\mapsto  \mathrm{Div}(X_{T})(R)$ is representable, and defines an affine vector space scheme of finite dimension over $k$, which we denote by 
$\mathrm{Div}(X_T)$.
\end{prop} 
\begin{proof} By Siegel's theorem  the set $X_T$ is finite. Therefore $\mathrm{Div}(X_{T})(R)$ is simply
$\mathrm{Hom}(X_T, k) \otimes_k R$. Since $\mathrm{Hom}(X_T, k)$ is a finite-dimensional vector space, it defines an affine scheme as in \S\ref{sectLinAlg}.
\end{proof} 

Let $\Pi$ be a quotient of the fundamental groupoid ${}_0 \Pi_0$ in the category $\MT(\Or_S)$.  Thus $\Or(\Pi^{\dr}) \subset \Or({}_0 \Pi^{\dr}_0)$ is a Hopf subalgebra stable under the action of $\GdrS$.  Choose
$$M \subset \Or(\Pi^{\dr}) $$
a  graded vector space spanned by a set of algebra generators which are homogeneous in the weight. In other words,  the natural map
$$M \overset{\sim}{\To}  Q \Or(\Pi^{\dr}) $$
is an isomorphism of vector spaces, where $Q$ denotes the  space of indecomposables $ I  / I^2$ where $I \leq \Or(\Pi^{\dr})$ is the  augmentation ideal.   For example, if $\Pi={}_0 \Pi_0$ is the full fundamental group, we may take $M$ to be the graded vector space spanned by the set of Lyndon words in $e_0,e_1$.  The group scheme $\Pi$ can be retrieved from $M$, since  $M$ determines $\Or(\Pi^{\dr})$, which in turn uniquely determines  $\Or(\Pi)$ 
by the Tannaka theorem.

For simplicity,  write
 \begin{equation} H_S(M)=  \mathrm{Hom}_{ k-\hbox{vec}} ( M , \Pe^{\uu}_{S})^{\G_m} \end{equation} 
 for  the space of $k$-linear maps from $M$ to $ \Pe^{\uu}_{S}$ which respect the  weight-gradings. 
We have established   that $W_{2n} \Pe^{\uu}_S$ is finite-dimensional for all $n$, and therefore by corollary \ref{corHomFscheme}, $H_S(M)$ is representable and hence an affine scheme. It is in fact a projective limit of vector-space schemes and therefore has a $k$-linear structure.

By the Milnor-Moore theorem, $\Or(\Pi^{\dr})$ is a polynomial algebra on $M$, so any  linear map from $M$ to an algebra extends uniquely to an algebra homomorphism on $\Or(\Pi^{\dr})$. In particular,  there is a
   natural equivalence of functors:
$$ \mathrm{Hom}_{ k-\hbox{alg}} ( \Or(\Pi^{\dr}), \Pe^{\uu}_{S})^{\G_m} \overset{\sim}{\To} H_S(M)\ ,$$
where the left-hand $\mathrm{Hom}$ denotes homomorphisms of algebras which respect the weight gradings
 and the map is restriction to $M$.
It follows that given two choices  $M, M'$ of generating spaces, there is a canonical  isomorphism of schemes 
$$ H_S(M) \overset{\sim}{\To} H_S(M')$$
by extending a linear map  to $\Or(\Pi^{\dr})$ by multiplicativity, and then restricting to $M'$.  
However, the linear structures on $H_S(M)$ and $H_S(M')$ are not equivalent. 
We can  therefore view the choice of $M$ as  endowing  the space of graded algebra homomorphisms from $\Or(\Pi^{\dr})$ to $\Pe^{\uu}_{S}$
 with a choice of $k$-linear structure.

\begin{defn} The cocycle map extends by linearity to a   natural $R$-linear transformation  of functors from commutative rings to free $R$-modules:
\begin{eqnarray} \label{DivXsToHom} 
c:  \mathrm{Div}(X_S)    &\To&    H_S(M) \\
\sum n_x x & \mapsto & \sum n_x c^x\big|_M \ . \nonumber 
 \end{eqnarray} 
By the Yoneda lemma, $(\ref{DivXsToHom})$ is a morphism of schemes.
\end{defn}
The map $(\ref{DivXsToHom})$ is given explicitly by 
$$c( \sum n_x x   ) \  = \   (w \mapsto  \sum n_x \Li^{\uu}_w(x))  \qquad \hbox{ for all } w \in M \ . $$
 Since a cocycle is determined by its action on $M$,  the space of cocycles 
$$Z_{\G_m}^1(\GdrS, \Pi^{\dr})  \To    H_S(M) $$
is a closed subscheme. The point is that, although there is no linear structure on cocycles,  there is a linear structure on $H_S(M)$.

\begin{defn} Define the \emph{scheme of virtual $M$-cocycles} for $X_S$ to be the fiber product
\begin{equation}  V^1_S(X_S, M) =       \mathrm{Div}(X_S)       \times_{H_S(M)} Z_{\G_m}^1(\GdrS, \Pi^{\dr})\ .
\end{equation} 
Its points consist of those divisors which satisfy the equations  defining a cocycle.
\end{defn}

Since the points of $X_S$ define divisors in $\mathrm{Div}(X_S)(\Q)$ and also cocycles, we have 
$$X_S   \quad \subset \quad  V^1_S(X_S,M) (\Q)\ .$$
The definition is functorial with respect to $M$ in the sense that if $M  \subset  M'$, then  the natural map
$V^1_S(X_S,M') \rightarrow V^1_S(X_S,M)$
is a closed embedding.

\subsection{Enlarging the set of primes $S$} Now let $T$ be a finite set of primes containing $S$. 
There is a natural injective map of graded Hopf algebras
$$\Pe^{\uu}_S  \To \Pe^{\uu}_T $$
since $\MT(\Or_S)$ is a full subcategory of $\MT(\Or_T)$ and hence $U^{\dr}_T\rightarrow \UdrS$ is faithfully flat.   In particular, we deduce a morphism of affine schemes
$$H_S(M)  \To  H_T(M)$$
which is a closed embedding.  Via this embedding we can view 
$ Z_{\G_m}^1(\GdrS, \Pi^{\dr}) $ 
as a closed subscheme of  $H_T(M)$.

\begin{defn}  Define the \emph{scheme of $T$-virtual $M$-cocycles} for $X_S$ to be 
\begin{equation}  V^1_T(X_S, M) =       \mathrm{Div}(X_T)       \times_{H_T(M)}  Z_{\G_m}^1(\GdrS, \Pi^{\dr})\ .
\end{equation} 
Its points are  divisors which  define a cocycle and are unramified at primes in $T \backslash S$.
\end{defn}

The set of integral points $X_S$ satisfy $X_S \subset   \mathrm{Div}(X_T) (\Q)$ and therefore
\begin{equation} \label{XSinV1} X_S \quad \subset \quad  V^1_T(X_S,M)(\Q)\ .
\end{equation}
This inclusion can be strict for small $M$, as the examples below will show, but 
 we claim  that for any $T$, one has 
$X_S =  \varprojlim_M V^1_T(X_S,M) (\Q)$.

The scheme $V^1_T(X_S, M )$ is functorial  in  $T$ (as well as $M$).  Given $T' \supset T$ there is a natural morphism such that the following diagram commutes
$$
\begin{array}{ccc}
V^1_T(X_S,M) & \To &  V^1_{T'}(X_S,M) \\ 
 \downarrow  &   & \downarrow   \\
 Z_{\G_m}^1(\GdrS, \Pi^{\dr}) &    = &   Z_{\G_m}^1(\GdrS, \Pi^{\dr})
\end{array}\ .$$
The map along the top is a closed embedding. 
\subsection{Variants}
Taking the union over all finite sets $T$ containing $S$, we can set
$$V^1_{\infty}  (X_S,M) = \varinjlim V^1_T (X_S,M)$$
It is a functor from commutative $k$-algebras to sets. 
Elements in $V^1_{\infty}  (X_S,M)$ are simply divisors on $X$ which give rise to cocycles for $\GdrS$.

Similarly, it is convenient to consider divisors with support along a given finite set 
$D \subset X(\Q)$  of rational points. Then let $\mathrm{Div}(D)$ be the scheme whose points are 
$$\mathrm{Div}(D) (R) = \{ \sum_{x \in D}  n_x x :  n_x \in R \}\ .$$
There exists $T$ such that all points in $D$ are $T$-integral, i.e. $D \subset X_T$. So we can define
\begin{equation} V^1_D(X_S, M) =  \mathrm{Div}(D) \times_{H_T(M)} Z_{\G_m}^1(\GdrS, \Pi^{\dr})
\end{equation}
It does not depend on the choice of $T$, and 
defines a closed subscheme of $V^1_T(X_S, M)$.  Its points consists of the divisors supported on $D$ which form a cocycle for $\GdrS$.

\subsection{Coleman functions via virtual cocycles}
The point of the previous  construction is that  divisors provide a way to construct cocycles via the natural map
$$V^1_{\infty}(X_S, M)(\overline{\Q})  \To Z^1(\GdrS, \Pi^{\dr}) (\overline{\Q}) $$
This  provides the following generalisation of theorem \ref{mainthm}.
    
  \begin{thm} \label{mainThmDiv} Let $n\geq 0$ and  $k$ be a finite extension of $\Q$.  Let $\Pi$,  $M$  be as above.  Let  us write
 $N= \dim \Or(Z)_{2n}$   and suppose that 
  $$
 \dim \gr^W_{2n} \Or(\Pi^{\dr}) >  N \ .
$$
Let  $w_1, \ldots w_N$ be linearly independent elements in $\gr^W_{2n} \Or(\Pi^{\dr})$. Then any $N$ elements 
$$\xi_1, \ldots, \xi_N \quad \in \quad  V^1_{\infty}(X_S, M)(k)$$ satisfy the  following equation  with coefficients in $k$:
 \begin{equation}
 \label{detinthmDiv} 
  \det \big(  \Li^{\uu}(\xi_j)\big)_{1\leq i,j \leq N} =0 
  \end{equation}
where $\Li^{\uu}$ is extended linearly to divisors:  $\Li^{\uu}(\sum n_x x)$ is defined to be $\sum n_x \Li^{\uu}( x)$. 
   \end{thm}

 This discussion motivates the following question:

\begin{question} Is  $V^1_D(X_S,M)$ Zariski-dense in $Z_{\G_m}^1(\GdrS, \Pi^{\dr})$ for sufficiently large $D$?
\end{question}

\subsection{Ramification}
The natural map $(\ref{DivXsToHom})$ sends 
$$ c:  \mathrm{Div}(X_T)    \To    H_T(M)  \ .$$
Short of demanding that $c$ define a cocycle, we can first find conditions for $c$ to be unramified  outside $S$. 
\begin{defn} Define  $\mathrm{Div}_{M, S}(X_T)$, the space of \emph{$M$-unramified virtual divisors}, by 
$$\mathrm{Div}_{M, S}(X_T) =  \mathrm{Div}(X_T)  \times_{H_T(M)} H_S(M) \ .$$
\end{defn} 

The space $\mathrm{Div}_{M, S}(X_T)$ is an affine scheme. It is  a projective limit of vector space schemes.   In particular, it has a $k$-linear structure.  It is finite dimensional if $M$ is.

 There is a variant $\mathrm{Div}_{M,S}(D)$ in which we replace $X_T$ with a subset $D \subset X_T$. Its points are  divisors supported on $D$
 which are $M$-ramified outside of $S$.

 The ramification conditions can be computed recursively. The graded Lie algebra $\mathfrak{u}_T^{\dr}$ of $\UdrS$ is a free Lie algebra and satisfies
 $$(\mathfrak{u}_T^{\dr} )^{ab} \cong \bigoplus_{n\geq 1} \mathrm{Ext}_{\MT(\Or_T)}^1(\Q(0),\Q(n))\otimes_{\Q} \Q(-n)\ .$$
 There are $|T|$ generators in degree (one half of the Hodge-theoretic weight) $1$ corresponding to every prime $p\in T$, and a generator in every odd degree $\geq 3$. 
 The former are dual to $\log^{\uu}(p)$, the latter to $\zeta^{\uu}(2n+1)$, for $n\geq 1$. Denote a choice of  lift of these generators to $\mathfrak{u}^{\dr}_T$ by 
 $\nu_p$, for $p \in T$ and $\sigma_{2n+1}$, for $n\geq 1$.  They act by derivations on $\Pe^{\uu}_T$. In fact, the elements  $\nu_p$ are uniquely determined, but the $\sigma_{2n+1}$  involve choices.  Their action factors through the operator $D$ given by the formula $(\ref{Deltafull})$, so in particular the action of $\nu_p$ can be written down in closed form.
 
 \begin{lem} \label{lemram} An element $\xi \in \Pe^{\uu}_T$  lies in the subspace $\Pe^{\uu}_S$ if and only if 
 \begin{enumerate}
 \item  $\nu_p \xi  = 0  $  \quad\qquad  for all $p \in T \backslash S.$
 \item $\nu_p \xi \in  \Pe^{\uu}_S$  \quad\qquad  for all $p \in S$.
 \item $\sigma_{2n+1} \xi \in \Pe^{\uu}_S$ \quad  for all $n \geq 1$.
 \end{enumerate}
  \end{lem} 
  \begin{proof} The conditions imply that the action of $U^{\dr}_T$ on $\xi$ factors through   $\UdrS$.
  \end{proof} 
Since the operators $\sigma_{2n+1}$ decrease the degree  by $2n+1$,  and $\Pe^{\uu}_T$ has no elements of negative degrees, there are only finitely many conditions to be satisfied in  $(2)$, and the lemma yields a finite number of conditions for $\xi$ to be unramifed at primes outside $S$. 

Since the ramification conditions are linear, the space $\mathrm{Div}_{M,S}(X_T)$ is isomorphic to an affine space $\A^n$ if $M$ is finite-dimensional.

 \begin{prop}  \label{propdivisinf} Let $M$ be finite-dimensional. Then 
the dimension 
of $\mathrm{Div}_{M,S}(X_T)$ grows exponentially in $|T|$. 
 \end{prop}
 \begin{proof} Let $m = |T \backslash S|$. The  number of conditions for an element $\xi \in \mathrm{Div}(X_{T})$ to be an $M$-virtual $S$-unit is a polynomial in $m$, by lemma \ref{lemram},   since $M$ is finite.   But by \cite{EST},  the number of $T$ units grows faster than polynomially in $|S'|$.  
 \end{proof}

We  view the space of virtual cocycles
$$V^1_T(X_S,M) \quad \subset \quad \mathrm{Div}_{M,S}(X_T)$$
as a closed subscheme of this affine space.

\begin{thm} In the situation of theorem $\ref{mainThmDiv}$,
there exists a finite set of primes $T$ containing $S$,  a  finite extension $k$ over $\Q$, and a    motivic Coleman function such that 
$$\sum_{i=1}^m p_i  \Li^{\uu}_{w_i}(x) =0  \qquad \hbox{ for all } x \in X_S \ , $$
where the $p_i\in \Pe^{\uu}_S \otimes_{\Q} k $ are not all zero and are given by explicit polynomials in the $\Li^{\uu}_{w_k}(\xi_j)$ for some divisors $\xi_j \in \mathrm{Div}_{X_T}(k)$ with support in $X_T$. 
\end{thm}

\begin{proof}

The number of equations defining $V^1_T(X_S,M)$ only depends on $M$. It follows that the dimension 
of $V^1_T(X_S,M)$ tends to infinity as $|T|$ tends to infinity (this lends some plausibility to question 1.)
In particular, for sufficiently large $T$, we have 
$$\dim V^1_T(X_S,M) >0\ $$
and hence an infinite supply of virtual cocycles. 
Theorem \ref{mainThmDiv} therefore overcomes the difficulties explained in \S\ref{sectasymptotics}: although there may not be $N$ integral points on $X_S$ to which one can apply the theorem $ \ref{mainthm}$, we can always find $N$ virtual cocycles if $T$ is suffciently large, after passing to a finite field extension.   

In particular,  we can proceed as in \S\ref{sectComments} (2). Let us suppose, by repeatedly taking minors of the matrix
$(\ref{detinthm})$ and reducing the value of the integer $N$, that every $N-1$ by  $N-1$ minor is not identically zero on all sets of $N-1$ elements
of $V^1_{\infty}(X_S, M)(k)$. 
Therefore we can take  $\xi_1,\ldots, \xi_{N-1} \in V^1_{\infty}(X_S,M)(k)$   and $\xi_N=x \in X_S$ an actual integral point via $(\ref{XSinV1})$ to obtain  a motivic Coleman function 
$$\sum p_i  \Li^{\uu}_{w_i}(x) =0  \qquad \hbox{ for all } x \in X_S \ , $$
where the $p_i \in \Pe^{\uu}_S \otimes_{\Q} k $ are minors of the matrix  $(\ref{detinthmDiv})$.  By minimality of $N$, they are not all zero.
\end{proof}

\begin{question} Is $\Pe^{\uu}_{\MT(\Q)}$ generated by the $\Li^{\uu}_w(x)$, for $x \in X(\Q)$?
\end{question} 
Likewise, one can ask if $\Pe^{\mm,+}_{\MT(\Q)}$ is generated by the $\Li^{\mm}_w(x)$?

\subsection{Examples} 
 
 \subsubsection{Example 1} Let $\Pi= {}_0 \Pi_0^{ab}$. It is the  weight one quotient of the motivic fundamental group. Its affine ring is the symmetric 
 tensor algebra on $\{e_0,e_1\}$.  Let 
 $S= \{2\}$ and $S' = \{2,3\}$.  The set $X_{S'} \cap (0,1)$ of $S'$ integral points is 
 $$ D= \textstyle{ \{  {1\over 9},   { 1\over 4 } ,  {1\over 3},{1 \over 2},  {2 \over 3}, {3 \over 4}, {8 \over 9 }  \} }\ .$$
The point $\{{1\over 2} \} $ isa  genuine integral point on $X_S$. Since $\Or(\Pi^{\dr})$ is generated by $M_1= \Q e_0 \oplus \Q e_1$,  and $\Li^{\uu}_{e_0} = \log^{\uu}$, $\Li^{\uu}_{e_1} (x)= \Li_1^{\uu}(x) =- \log^{\uu}(1-x) $, the condition for 
a divisor  $\xi$  supported on these points to be unramified is that 
$$\log^{\uu}(\xi)   \    \hbox{ and } \  \log^{\uu}(1-\xi)   \in  \Pe^{\uu}_S$$
i.e., both are unramified at prime $3$. Thus we are looking for  linear combinations of the  points  in $D$ such that all $\log^{\uu}(3)$ terms cancel.  
By the functional equation of the unipotent logarithm  (\cite{NotesMot} \S5.3) we have 
$$\log^{\uu} ( \prod p_i^{n_i} ) = \sum n_i \log^{\uu}(p_i)$$
where $p_i$ are positive primes, and $n_i \in \Z$. We find, then, that in addition to $\xi_0 = {1\over 2}$, the following elements
are $M_1$-virtual $S$-units:
$$
\begin{array}{lcl}
\xi_1= \textstyle{[ {1 \over 9}] -2 [ {1\over 3}]}    & \qquad    & \xi_2= \textstyle{[ {2 \over 3}] -  [ {1\over 3}]}   \\
  \xi_3 = \textstyle{[ {3\over 4}] + [ {1 \over 3}]+ [ {1\over 4}]}&  \qquad & \xi_4 = [ {8\over 9}] -2\textstyle{[ {1 \over 3}]} 
\end{array}
$$
For example,  $\Li_{e_0}^{\uu}(\xi_1) = \log^{\uu}(\xi_1) = \log^{\uu}({1\over 9}) - 2 \log^{\uu}({1\over 3})=0$ and similarly $\Li_{e_1}^{\uu}(\xi_1) = -\log^{\uu}(1-\xi_1) = \log^{\uu}(2)$, which are both  unramified  at the prime $3$.   This can also be checked using $(1)$ and $(2)$ of lemma \ref{lemram}. Therefore
$$\mathrm{Div}_{M_1,S}(D)(\Q) = \Q  \xi_0 \oplus \Q  \xi_1  \oplus \Q  \xi_2  \oplus \Q  \xi_3  \oplus \Q \xi_4   $$
and in this case we have $V^1_D(X_S,M_1) =  \mathrm{Div}_{M_1,S}(D)$ since there are no further conditions for a divisor to define a cocycle. 
\subsubsection{} Now let $S$ and $S'$ be as above, and let $\Pi$ be the weight two quotient of ${}_0 \Pi_0$. Its  affine ring in $\dr$ is the Hopf algebra generated by 
$W_2 \Or({}_0\Pi^{\dr}_0)$, and is generated  by  $\{e_0, e_1,e_1e_0\}$.  Let  $M_2= \Q e_0 \oplus \Q e_1 \oplus \Q e_1 e_0$.  Since $\Li^{\uu}_{e_1e_0} = \Li_2^{\uu}$, the extra condition for a divisor $\xi$ to be $M_2$-unramified  is that 
$$\Li^{\uu}_2 (\xi) \in \Pe^{\uu}_S\ .$$
We have already computed 
\begin{equation} \Delta \Li^{\uu}_2(x) =1\otimes \Li^{\uu}_2(x)  +  \Li^{\uu}_1(x) \otimes \log^{\uu}(x) + \Li^{\uu}_2(x)\otimes 1 . 
\end{equation}
Therefore the condition that $\Li^{\uu}_2(\xi)$ be unramified at $3$,  where $\xi= \sum n_x [x]$,  is the condition  that  the coefficients in
$$\Delta' \Li^{\uu}_2(\xi) = \sum n_x \log^{\uu}(1-x) \otimes  \log^{\uu}(x) $$
of the terms $\log^{\uu}(2) \otimes \log^{\uu}(3)$, $\log^{\uu}(3) \otimes \log^{\uu}(3)$ and $\log^{\uu}(3) \otimes \log^{\uu}(2)$ should  all vanish. 
There are three such equations, and we find that 
$$\mathrm{Div}_{M_2,S}(D)(\Q) = \Q  \xi_0 \oplus  \Q \xi_5 \ , $$
is two dimensional, where
\begin{equation}\label{xi0} \xi_0 = [\textstyle{1\over 2}] \quad   \hbox{ and } \quad  \xi_5 =  \textstyle{   6 [ {1 \over 3}]- [{1\over 9}] -6 [ {2\over 3}] + [ {8\over 9}]    }\ .
\end{equation}
We easily check that   $\Li_{e_0}^{\uu}(\xi_5)= -3 \log^{\mm}(2)$ and $\Li_1^{\uu}(\xi_5) = -3 \log^{\mm}(2)$. 
$$\Delta' \Li^{\uu}_2(\xi_0)=- \log^{\uu}(2) \otimes \log^{\uu}(2) \qquad \hbox{ and } \qquad  \Delta' \Li^{\uu}_2(\xi_5)= 0 $$
The second equation implies that $\Li^{\uu}_2(\xi_5)=0$, and that 
$\Li^{\mm}_{2} (\xi_5) = \alpha \zetam(2)$, for some $\alpha \in \Q$. By taking the period, one verifies numerically that  $\alpha$ is very close to $-1$.

 \subsubsection{} Continuing the previous example,  let $\xi$ be an $M_2$-virtual $S$-unit. Then the condition that $\xi$ define a virtual cocycle in $Z^1_{\G_m}(\GdrS, \Pi)$ is the  equation 
 $$\Delta' \Li_2^{\uu}(\xi) = \Li_1^{\uu}(\xi) \otimes \log^{\uu}(\xi)  \ . $$
 If we write  $\xi = \sum n_i [\alpha_i]$, this is equivalent to the equations
 \begin{equation} \label{alphaeqns}  \sum_i  n_i  \big( \log^{\uu}(1-\alpha_i) \otimes  \log^{\uu}(\alpha_i)\big) = \big(\sum_i n_i \log^{\uu}(1-\alpha_i) \big) \otimes \big(\sum_j n_j \log^{\uu}(\alpha_j)\big) \ .
 \end{equation} 
Let us identify $\mathrm{Div}_{M_2,S}(D)\cong \A^2$ by writing   $\xi = t_1 \xi_0 + t_2  \xi_5$.
 By earlier calculations  $\Delta' \Li_2^{\uu}( \xi) =- t_1 \log^{\uu}(2) \otimes \log^{\uu}(2)$, 
   $\log^{\uu}(\xi) = (-t_1 -3t_2 ) \log^{\uu}(2)$ and  $\Li_1^{\uu}(\xi) =  (t_1 -3t_2 )\log^{\uu}(2)$. Therefore the condition that $\xi$ lie in $ 
V^1_D(X_S, M_2)$ is the equation 
$$  - t_1 \log^{\uu}(2) \otimes \log^{\uu}(2)    = (-t_1 -3t_2 )  (t_1 -3t_2 )  \log^{\uu}(2) \otimes \log^{\uu}(2)\ .$$
Since $\log^{\uu}(2)$ is non-zero,   the  subscheme
$$  V^1_D(X_S, M_2) \ \subset \  \mathrm{Div}_{M_2,S}(D)$$
is given by the conic 
   $$t_1^2 - 9 t_2^2 - t_1 =0 \ .$$
    By parametrizing this conic we find that the following family of divisors 
\begin{equation}     {1 \over 1-a^2}   \xi_0 +   {a \over 3(a^2-1)} \xi_5    \end{equation} 
lie in $V^1_D(X_S, M_2)$,  
for any $a\neq \{0,1\}$, where $\xi_0$ was defined in $(\ref{xi0})$.  When $a=0$ we retrieve the  genuine point ${1\over 2} \in X_S$, but for other values of $a$ we obtain a  new cocycle. 

\subsubsection{}  Let $M_2$, $\Pi$, $D$ be as above. Let $S'= \{2,3\}$ but this time take 
 $S= \{3\}$.  Proceeding in a similar manner, we find a vector space of two unramified divisors:
 $$\xi_6= -6[\textstyle{1 \over 2}]+[\textstyle{1\over 9}]-6[\textstyle{1\over 3}]+{3\over  2}[\textstyle{1\over 4}]-3[\textstyle{2\over 3}] 
 \quad \hbox{ and } \quad \xi_7= -3[\textstyle{1\over 3}]-6[\textstyle{1\over 2}]-6[\textstyle{2\over 3}]+{3\over 2}[\textstyle{3\over 4}]+[\textstyle{8\over 9}]$$
 So $\mathrm{Div}_{M_2,S}(D)(\Q) =  \Q \xi_6 \oplus \Q \xi_7$. 
The scheme of virtual $\Pi$-cocycles is a conic in this two dimensional affine space, and contains the following rational point
$$  - \textstyle{6\over 7}  [\textstyle{1 \over 2} ]+\textstyle{15 \over 7} [\textstyle{1 \over 3} ]-\textstyle{9\over 7} [\textstyle {1 \over 4}]-\textstyle{6 \over 7}[ \textstyle{1\over 9}]-\textstyle{24 \over 7}[\textstyle{2 \over 3}]+\textstyle{3\over 2} [\textstyle{3 \over 4}]+ [\textstyle{8 \over 9}]\qquad \in  \qquad V_D^1(X_S, M_2)(\Q) \ . $$
It satisfies $(\ref{alphaeqns})$.
 As far as the dilogarithm is concerned, this divisor plays the role of  a point on $X_S$, even though there are no such points since $X_S = \emptyset$. 

\subsection{Remarks on Zagier's conjecture and Bloch group} Beilinson and Deligne's interpretation of Zagier's conjecture involves finding divisors $\xi = \sum n_x [x]$ such that 
\begin{equation} \label{Zagcond}  \Delta' \Li_n^{\uu}(\xi) =0\ . \end{equation} 
For $n>1$, such a $\xi$ defines an element of $\mathrm{Ext}^1_{\MT(\Q)}(\Q(0),\Q(n))$, and the single-valued period of the associated de Rham period  is a  rational multiple of $\zeta(2k+1)$ if $n=2k+1$ is odd. The conjecture follows easily from the known theorems about mixed Tate motives over $\Q$ in this case. 
The   condition $(\ref{Zagcond})$ is  linear, and gives an explicit description of the rational algebraic $K$-theory of $\Q$.  This story can be extended to number fields.

Compare, on the other hand, the equations defining a virtual cocycle, which in the case of the classical polylogarithms are the non-linear equations
\begin{equation}   \Delta' \Li_n^{\uu}(\xi) =  \sum_{i=0}^{n-1} \Li^{\uu}_{n-i}(\xi) \otimes {1 \over i!} \big(\!\log^{\uu}(\xi)\big)^i
\end{equation} 
  Compare with $(\ref{DeltaLiuu})$. More precisely, these are the equations defining $V^1_D(X_S,M)$ inside $\mathrm{Div}_{M,S}(D)$
where $\xi$ is any divisor supported on $D$, where $M = \Q e_0 \oplus \bigoplus_{0 \leq k < n } \Q e_1 e_0^k$.

 \section{General curves} \label{sectfinal}
 We  indicate how the above arguments can be generalised to arbitrary curves over a number field $k$.
  The missing ingredient
is the conjectural upper bound on the space of motivic periods, which is expected to follow from a version of Beilinson's conjecture. For this reason
we only provide the key steps in the argument and leave the computations of the conjectural dimensions for future applications.  
The main novelty in this section is the definition and use of the canonical de Rham path on the schemes $\PF$. It is not clear how to make the argument work using the higher Albanese manifolds.

We work in a Tannakian category $\mathcal{H}$ of realisations \cite{DeP1}, \S1.10 which contains as many cohomology theories as required  and sufficiently many to ensure 
that, at least conjecturally, the category of mixed motives over $k$ should be a full subcategory of $\mathcal{H}$. 
Let $X$ be  a curve over $k$, and consider the  situation  \S\ref{sectOtherReal}. 
In particular, Beilinson's cosimiplicial construction of the unipotent fundamental groupoid of $X$ defines   a groupoid of pro-objects
$\pi^\mathcal{H}_1(X,y,x)$ in $\mathcal{H}$ for every $x,y\in X(k)$, whose realisations coincide with those defined in \S\ref{sectunip}. 
For any pair of fiber functors
$\omega, \omega'$ on $\mathcal{H}$, let $\Pe^{\omega',\omega}_{\mathcal{H}}$ denote the  ring of $\omega', \omega$-periods of $\mathcal{H}$, and let 
$\Pe_X^{\omega',\omega} \subset \Pe^{\omega',\omega}_{\mathcal{H}}$   denote the ring generated by the matrix coefficients of $\Or(\pi^\mathcal{H}_1(X,y,x))$ for all $x,y\in X(k)$, or possibly tangential base points over $k$.
  The main examples 
 that we have in mind are when $\omega$ is the de Rham or crystalline  fiber functor, and $\omega'$ is the Betti functor relative to an embedding of $k$ into $\C$.  A version of   Beilinson's conjectures predicts some control on the dimensions of the weight-graded pieces of $\Pe_X^{\omega',\omega} $, which we shall not discuss here.

For $x,y $ as above,  write
${}_y\Pi^{\omega}_x$ for the $\omega$-realisation of $\pi^\mathcal{H}_1(X,y,x)$.  These form torsors:
\begin{equation}\label{Generaltorsor} {}_y\Pi^{\omega}_x  \times {}_x\Pi^{\omega}_x  \To {}_y\Pi^{\omega}_x
\end{equation}
and possess a rational point, since $ {}_x\Pi^{\omega}_x$ is a pro-unipotent affine group scheme.

The discussion proceeds along the lines of \S\ref{sectMainargument}. The main difference is that we now consider two fiber functors $\omega$ and $\omega'$, which leads to a notion of bitorsors.  In the mixed Tate case we used `canonical de Rham paths' using the weight-grading on the de Rham or canonical fiber functor.   Since these do not exist in general,  they will be replaced
 with a new notion of canonical de Rham paths which we shall discuss at some length, and  in the crystalline case by Frobenius-invariant paths, which are well-known.

 \subsection{Cocycles of  a bi-torsor}   \label{sectBitorsCocyc} Let $0$ be a fixed rational base point (tangential or otherwise) of $X$.
 Let $G^{\omega} = \mathrm{Aut}_{\mathcal{H}}^{\otimes}(\omega)$, and likewise for $\omega'$,  and set
  $$P = \mathrm{Isom}_{\mathcal{H}}^{\otimes} (\omega, \omega') = \mathrm{Spec}  \,  \Pe^{\omega',\omega}_{\mathcal{H}} \ .$$
  It is a  $G^{\omega} \times G^{\omega'}$ bitorsor (since our convention is that Tannaka groups act on the left of fundamental groups, and hence on  the right on their affine rings). 
 Suppose that for every $x\in X_S$ we are given two paths 
 \begin{equation}   \label{xcgo} \xco \in \xpio(k)   \qquad \hbox{ and } \qquad  \xgo \in {}_x \Pi^{\omega'}_0(k')  \    \end{equation}
where $ k'$ is an extension of $k$ (for example, a $p$-adic field).
 The scheme $P$ acts via:
 $$P\times {}_x \Pi^{\omega'}_0 \To {}_x \Pi^{\omega}_0\ .$$
  Since ${}_x \Pi^{\omega}_0$ is a ${}_0 \Pi_0^{\omega}$-torsor we obtain a morphism of schemes over $k'$
\begin{equation}\label{abicocycledef} a: P \times_k k'  \To {}_0 \Pi_0^{\omega} \times_k k' 
\end{equation}
given on points by the equation  $\phi(\xgo ) = \xco \,   a_{\phi} $ for $\phi \in P$.  It satisfies
\begin{equation} a_{g \phi h }  = \alpha_g \, g(a_{\phi}) \, g \phi(\beta_h)  \qquad \hbox{for } \quad  (g,h) \in G^{\omega} \times G^{\omega'}\ ,
\end{equation}
where 
$$\alpha \in Z_{\alg}^1(G^{\omega}, \opio)(k) \qquad \hbox{  and  }  \qquad \beta \in Z_{\alg}^1(G^{\omega'}, {}_0 \Pi^{\omega'}_0)(k')$$
 are the algebraic cocycles defined
by $\xco$ and $\xgo$ respectively, i.e., $g (\xco) = \xco \,  \alpha_g$ and 
$g (\xco)  = \xco  \, \beta_g$.  We call the data of $a,\alpha, \beta$ an algebraic  bi-torsor cocycle.

Changing $\xco$ to $\xco b^{-1}$ and  $\xgo$ to $\xgo c^{-1}$ modifies 
it by  a boundary:
\begin{equation} \label{atwistbyboundary} 
( \alpha_g,a_{\phi}, \beta_h) \mapsto  (  b \, \alpha_g\, g(b)^{-1}, b \, a_{\phi} \, \phi(c)^{-1},  c \,\beta_h \, h(c)^{-1})
\end{equation}
where  $(b,c) \in {}_0 \Pi^{\omega}_0 (k) \times {}_0 \Pi^{\omega'}_0 (k')$.
 From now on,  we shall  mostly emphasise the functor $\omega'$  and consider only the right action of $G^{\omega'}$ on $P$. In other words, we shall  drop $\alpha$ from the data and consider only $(a, \beta)$. 
  
 \subsection{}  Let $\Pi$ be an affine  group  scheme in $\mathcal{H}$.  This means that its affine ring is an Ind-object of $\mathcal{H}$.   Let 
  $Z_{r, \alg}^1(P, \Pi^{\omega}) $ denote  the functor   whose $R$ points, where $R$ is a commutative ring over $k$, are pairs of  morphisms  of schemes
  \begin{eqnarray}
  a : P \times_k R & \To &  \Pi^{\omega} \times_k R \nonumber \\ 
  \beta  : G^{\omega'}  \times_k R & \To &  \Pi^{\omega'} \times_k R \nonumber 
    \end{eqnarray} 
  such that the following cocycle condition is satisfied
 \begin{equation}  \label{atorseqn} a_{ \phi h } = a_{\phi} \,  \phi(\beta_h) \end{equation} 
 for all  $h \in G^{\omega'}, \phi \in P$. This implies that  $\beta$ is necessarily an algebraic  $G^{\omega'}$-cocycle, and   therefore there  is a natural transformation of functors 
 \begin{equation}  \label{ZPtoZG}
 Z_{r,\alg}^1 (P, \Pi^{\omega} ) \To Z_{\alg} ^1(G^{\omega'}, \Pi^{\omega'})\ .
 \end{equation} 
 Its  fibers are left $\Pi^{\omega}$-torsors. To see this, let   $a,a' \in Z_r^1 (P, \Pi^{\omega} )(R)$  be two elements with the same image.
Define $b: P\times_k R \rightarrow \Pi^{\omega}\times_k R$ by $b_{\phi} = a'_{\phi} a_{\phi}^{-1}$. 
Then by $(\ref{atorseqn})$, 
$$b_{\phi h} = a'_{\phi h } a^{-1}_{\phi h} = a'_{\phi} \phi(\beta_h) \phi(\beta_h)^{-1} a_{\phi}^{-1} = b_{\phi}$$ for all $h \in G^{\omega'} \times_k R$, and since $P$ is a $G^{\omega'}$-torsor, it is trivialised over some finite flat 
extension $R'$ of $R$.  It follows that $b_{\phi}$ is constant for all $\phi \in P(R')$. Therefore the images of $a,a' \in Z_{r,\alg}^1 (P, \Pi^{\omega} ) (R')$ satisfy $a' = b a$ for some $b\in \Pi^{\omega}(R')$. The converse is clear.  
On the other hand, since $\Pi^{\omega}$ is pro-unipotent, it follows that the fibers of $(\ref{ZPtoZG})$ are either empty or trivial $\Pi^{\omega}$-torsors.

\begin{rem} Define the  space of  left $P$-cocycles  $Z_{l,\alg}^1(P, \Pi^{\omega})$ to be the functor  whose $R$-points are morphisms
 $a : P\times_k R \rightarrow \Pi^{\omega}\times_k R$ together with a $a \in Z_{\alg}^1(G^{\omega}, \Pi^{\omega})(R)$ such that 
 $ a_{g \phi} =  \alpha_g \, g(a_{\phi}) $
 for all  $g \in G^{\omega}, \phi \in P$. There is a natural transformation
 \begin{equation} \label{ZleftPtoZG}
 Z_{l ,\alg}^1 (P, \Pi^{\omega} ) \To Z_{\alg}^1(G^{\omega}, \Pi^{\omega}) 
 \end{equation}
whose fibers are right $Z^0_l$-torsors, where $Z^0_l$ is the inverse image of the trivial cocycle. Its $R$-points are morphisms  $r: P \times_k R \rightarrow \Pi^{\omega}\times_k R$
such that $r_{g\phi} = g(r_{\phi})$ for all $g \in G^{\omega}$. The right action of $Z^0_l$ on $ Z_l^1 (P, \Pi^{\omega} )$ is given by $a_{\phi} \mapsto a_{\phi}r_{\phi}$. 
\end{rem}

Let $a, a' \in Z_{r, \alg}^1(P, \Pi^{\omega})(k') $, and suppose that   $a'=ba$  for some $b \in \Pi^{\omega}(k')$ as above, It follows from $(\ref{atwistbyboundary})$ that $a$ and $a'$ differ by a twist by the boundary $b$. Denote by $\alpha, \alpha'$ the corresponding cocycles in 
$Z^1(G^{\omega}, \Pi^{\omega})(k')$, images of $a,a'$ under $(\ref{ZleftPtoZG})$. Then  by $(\ref{atwistbyboundary})$ it follows that
$\alpha'_g = b \alpha_g g(b)^{-1}$.

\begin{rem}  
In our applications, the schemes $P$, $G$, $\Pi$ are all equipped with increasing weight filtrations $W$. The morphisms
$a, \alpha, \beta$  arising from paths  $(\ref{xcgo})$ will respect these filtrations. From now on, our spaces of cocycles will implicitly denote those cocycles which respect the weight filtrations (denoting this by a subscript $W$ would clutter the notation unnecessarily).

In this case, our sufficient condition for $Z_{r, \alg}^1(G^{\omega'}, \Pi^{\omega'})$ to define a scheme, namely, the finite-dimensionality of $W_n (\Or(G^{\omega'}))$ for all $n$, 
will also suffice to ensure that $Z_{r, \alg}^1(P, \Pi^{\omega})$ is a scheme, since $P$ is a $G^{\omega'}$-torsor and so $W_n(\Or(P))$ will also be finite-dimensional. Therefore $Z_{r, \alg}^1(P, \Pi^{\omega})$ will be a closed subscheme of 
$$\mathrm{Hom}_W( \Or(\Pi^{\omega'}), \Or(G^{\omega'})) \times \mathrm{Hom}_W( \Or(\Pi^{\omega}), \Or(P))\ . $$
\end{rem} 
 
\subsection{} Now  let $\Pi$ be an affine group scheme in $\mathcal{H}$, given by a quotient of ${}_0\Pi_0$. The image of the bicocycle $(\ref{abicocycledef})$ corresponding to a point $x$ and the paths $\xgo \in {}_0\Pi_0(k')$, $\xco \in {}_0\Pi_0(k)$,  define   algebra homomorphisms $a^x,\alpha^x,\beta^x$:
\begin{eqnarray} 
\alpha^x:  \Or(\Pi^{\omega}) & \To&   \Pe_{X}^{\omega, \omega}  \nonumber   \\
a^x : \Or(\Pi^{\omega}) & \To&   \Pe_{X}^{\omega', \omega}   \otimes_k k'\nonumber  \\
\beta^x:  \Or(\Pi^{\omega'}) & \To&   \Pe_{X}^{\omega', \omega'} \otimes_k k'  \nonumber  
\end{eqnarray} 
since $a$ is dual to $\Or(\Pi^{\omega})  \otimes_k k' \rightarrow \Or(P) \otimes_k k'= \Pe_{\mathcal{H}}^{\omega', \omega}\otimes_k k'$, and its image lands in the subring spanned by the periods of
the fundamental groupoid of $X$.  The cocycle conditions above can be translated into commutative diagrams which we shall omit. 
The morphisms $a, \alpha,  \beta$   respect the weight filtrations.

 \subsection{Frobenius-invariant paths}   We can reduce the size of the space of bitorsor cocycles as  follows. Suppose that 
 $\xgo \in {}_x \Pi^{\omega'}_0(k')$ is invariant under the action of an element (`Frobenius') $F \in G^{\omega'}(k')$, for
 some $k$-algebra $k'$. In this case, the map
 $$a: \Or(\Pi^{\omega}) \To \Pe^{\omega',\omega}_X\otimes_k k'$$
 lands in the $F$-invariant subspace of $ \Pe^{\omega',\omega}_X \otimes_k k'$.
 Suppose that the trivial path  $1\in {}_0 \Pi^{\omega'}_0(k')$ is the unique $F$-invariant element. 
 Then the  cocycle $\beta$  corresponding to ${}_x \gamma_0$ lies in the space $\beta \in Z_{F,\alg}^1(G^{\omega'},  \Pi^{\omega'})(k')$
of cocyles which are  trivial on  $F \in G^{\omega'}(k')$ and, by the same argument as  the first paragraph of \S\ref{sectskipped}  we have
\begin{equation} \label{Z1isH1Finv}
 Z_{F,\alg}^1(G^{\omega'},  \Pi^{\omega'}) \overset{\sim}{\To}  H_{F,\alg}^1(G^{\omega'},  \Pi^{\omega'})\ .
 \end{equation}
 where the subscript $F$ denotes elements which are trivial on $F\in G^{\omega'}(k')$.
 This certainly applies in the case when $\omega'$ is the crystalline  fiber functor (or de Rham by transporting $F$),  and $\xgo$ is the unique Frobenius invariant path \cite{Besser} \S1.5.2.
 
In the case when $\omega'$ is the Betti realisation corresponding to a real embedding of $k$, we can choose $\xgo$ to be invariant under complex conjugation (invariant under the real Frobenuis $F_{\infty}$).  In this situation we merely
deduce that the target space of the maps $a^x,\beta^x$ lie in the subspace of real-Frobenius invariant periods.

  \subsection{Canonical de Rham paths}  \label{sectHodgepath} We now explain how to choose the path $\xco$ in the case when $\omega$ is the de Rham fiber functor, by exploiting the Hodge filtration.
  
  \subsubsection{Hodge filtration} 
   The affine ring of  the de Rham fundamental groupoid carries a natural Hodge filtration, denoted $F^n$.
   The left  coaction dual to the  right-torsor structure $ \xpio \times \opio \rightarrow \xpio$ respects the Hodge filtration, i.e., 
   $$\Delta F^n \Or(\xpio) \subset F^n (\Or(\opio) \otimes_k \Or(\xpio) )\ .$$
   Since  $ F^0\Or({}_{\bullet}\Pi^{\dr}_0)= \Or({}_{\bullet}\Pi^{\dr}_0)$ for $\bullet \in \{x,0\}$, this implies that
   \begin{equation} \label{DeltaFn} \Delta F^n \Or(\xpio) \subset  \sum_{p+q=n, p,q\geq0 }  F^p \Or(\opio) \otimes_k F^q \Or(\xpio) \ .
   \end{equation}
   In particular, $\Delta F^1 \subset F^0 \otimes F^1 + F^1 \otimes F^0$. We can define 
   \begin{equation}\label{FnPidef} F^{n-1} {}_\bullet \Pi^{\dr}_0 = \Spec \,\big( \Or({}_\bullet \Pi^{\dr}_0)/F^{n}\big) \qquad \hbox{ for } \bullet = x, 0
   \end{equation}
  because $F^i. F^j \subset F^{i+j}$ and hence $F^n$ defines  an ideal in $\Or({}_\bullet \Pi^{\dr}_0)$. By $(\ref{DeltaFn})$
  $$\Delta  \Or(\xpio)/F^1   \quad   \subset  \quad \Or(\opio)/F^1 \otimes_k   \Or(\xpio)/F^1$$
  and by a similar equation with $x$ replaced by $0$, we deduce that $F^0 {}_0 \Pi^{\dr}_0  $ is a closed subgroup of $  {}_0 \Pi^{\dr}_0 $.
  By contrast with  \cite{KimAlb}, our $F^n$ are closed subschemes but are not subgroups for $n \geq 1$.  Furthermore, the subgroup $F^0$ is not  in general   normal.

 \begin{lem}  \label{lemF0tors} Multiplication gives a structure of a trivial right-torsor 
\begin{equation} \label{F0torseqn} F^0 {}_x \Pi^{dR}_0  \times F^0 {}_0 \Pi^{dR}_0   \To F^0 {}_x \Pi^{dR}_0 \ \ . 
\end{equation}  
In particular, there exists a rational point in $F^0 {}_x \Pi^{dR}_0 (k)$. 
\end{lem}
\begin{proof}  By the torsor property $(\ref{Generaltorsor})$, the isomorphism
$$ \Or({}_x \Pi^{dR}_0) \otimes_k \Or({}_x \Pi^{dR}_0) \overset{\Delta \otimes \id}{\To} \Or({}_0 \Pi^{dR}_0) \otimes_k \Or({}_x \Pi^{dR}_0)^{\otimes 2} \overset{\id \otimes m}{\To} \Or({}_0 \Pi^{dR}_0) \otimes_k \Or({}_x \Pi^{dR}_0)$$
induces  an isomorphism  $\Or \otimes F^1 \overset{\sim}{\rightarrow} \Or \otimes F^1$ and hence 
$$  \Or({}_x \Pi^{dR}_0) \otimes_k \gr^0_F \Or({}_x \Pi^{dR}_0) \overset{\sim}{\To} \Or({}_0 \Pi^{dR}_0) \otimes_k \gr^0_F \Or({}_x \Pi^{dR}_0)\ .$$
It respects the Hodge filtration on both sides. By strictness, it induces an isomorphism on the associated $\gr^0_F$ of both sides of the previous equation. This gives an isomorphism  
$$ \gr^0_F  \Or({}_x \Pi^{dR}_0) \otimes_k \gr^0_F \Or({}_x \Pi^{dR}_0)  \overset{\sim}{\To}  \gr^0_F \Or({}_0 \Pi^{dR}_0) \otimes_k \gr^0_F \Or({}_x \Pi^{dR}_0)\ .$$
This proves $(\ref{F0torseqn})$ and implies that it is a torsor. To show that it is trivial, use the fact that  the affine ring of $F^0 {}_0 \Pi^{dR}_0$ is a connected  filtered 
Hopf algebra (e.g., by the weight filtration), and so  $F^0 {}_0 \Pi^{dR}_0$ is pro-unipotent. Any non-empty torsor over a pro-unipotent affine group scheme admits a rational point.
 \end{proof}

\subsubsection{The affine scheme $\PF$}
  
  \begin{defn} \label{Hdefn}
  Let ${}_xH_0$ denote  the largest subalgebra of 
   $\Or({}_x \Pi^{dR}_0)$ such that:

   $(i)$. $W_0  \, {}_xH_0  \cong W_0 \Or({}_x \Pi^{dR}_0) \cong  k  ,$
   
   $(ii)$.  ${}_xH_0$ is stable under the coaction $\Delta: {}_xH_0 \rightarrow \Or({}_0 \Pi^{dR}_0) \otimes_{k} {}_xH_0  $,

   $(iii)$. ${}_xH_0   \subset F^1 \Or({}_x \Pi^{dR}_0)+ W_0\, {}_xH_0 $. 
  
 \noindent
 Since $W_0 \cap F^1 \Or({}_x \Pi^{dR}_0)= 0$, the sum in the right-hand side of $(iii)$ can be replaced with a direct sum.
  Likewise, define ${}_0H_0$ by relacing $x$ with $0$ in the above definition.  It is not necessarily a Hopf algebra, only a  left $ \Or({}_0 \Pi^{dR}_0)$-comodule algebra. 
    \end{defn}
  
 \begin{defn}  For $\bullet \in \{0,x\}$ let us define an affine scheme 
  $$     {}_{\bullet} \PF^{dR}_0  = \Spec \, \big( {}_{\bullet} H_0 )\ .$$
  \end{defn}
    There is a natural morphism of schemes 
 ${}_\bullet \Pi^{dR}_0 \rightarrow    {}_{\bullet}\PF^{dR}_0 $
   and condition $(ii)$ implies that  the following diagram commutes
$$
\begin{array}{ccccc}
  {}_\bullet \Pi^{dR}_0  &\times  &{}_0\Pi^{dR}_0 &  \To  &{}_\bullet \Pi^{dR}_0   \\
\downarrow&& \downarrow   &   & \downarrow   \\
    {}_{\bullet}\PF^{dR}_0 &  \times  & {}_0 \Pi^{dR}_0   & \To  &     {}_{\bullet}\PF^{dR}_0
\end{array}
$$
In other words,  $   {}_{\bullet}\PF^{dR}_0$ is stable under the right-action of ${}_0 \Pi^{dR}_0$.

There is an identical  construction for any affine groupoid scheme $\Pi$ in the category $\mathcal{H}$  which is a  quotient of ${}_0 \Pi_0$ (e.g.,  $\Pi={}_0 \Pi_0/W_{m}$), which we again denote by $\PF^{dR}$.

\begin{ex}  \label{extensorsF1} Suppose that $\Or({}_{\bullet} \Pi^{dR}_0)$  is isomorphic to the tensor coalgebra on a vector space $V$ and that the coaction is given by deconcatenation of tensors, e.g. $(\ref{pi1asTc})$. Then ${}_{\bullet}H_0$ is  the subspace generated by  $V\otimes \ldots \otimes V \otimes F^1 V$.  It is strictly contained in $F^1 \Or({}_{\bullet} \Pi^{dR}_0)$, which is the subspace generated by $V^{\otimes m} \otimes F^1 V \otimes V^{\otimes n}$. 
\end{ex}

The reader is warned that since the Hodge filtration is not motivic, the scheme $ {}_{\bullet} \PF^{dR}_0$ does not admit an action by the de Rham Galois group  $G^{dR}$.  Equivalently ${}_\bullet H_0$ is not stable under  the  coaction by its affine ring $\Or(G^{dR})$.

   \subsubsection{Canonical de Rham path}  \label{sectCanpath}
   
   By definition  $\ref{Hdefn}$,  there is a natural map
$$\varepsilon: {}_xH_0 \To {}_xH_0 /F^1 {}_xH_0 \overset{(iii)}{ \cong} W_0 \,  {}_xH_0 \overset{(i)}{\cong} k \  $$
which is a homomorphism.      It defines a point $\Spec k \rightarrow  {}_x \PF^{dR}_0$.
\begin{defn} Define the  \emph{canonical path}  ${}_x 1_0  \in {}_x \PF^{dR}_0(k)$ to be this point. 
\end{defn} 
 Consider the composition of the coaction (definition \ref{Hdefn} $(iii)$) 
 $$\Delta : {}_x H_0 \To \Or({}_0 \Pi^{dR}_0 ) \otimes_k {}_x H_0 $$
 with $\id \otimes {}_x1_0$. It  yields a homomorphism we denote by ${}_x1_0 : {}_x H_0 \rightarrow \Or({}_0 \Pi^{dR}_0 ) $.
 \begin{lem}  \label{lemx10canisom} It defines a canonical  isomorphism of algebras
 $${}_x1_0 : {}_x H_0 \overset{\sim}{\To} {}_0 H_0 \ .$$
  \end{lem}  
 \begin{proof} By lemma \ref{lemF0tors}, we can choose a point $c \in F^0 {}_x \Pi^{dR}_0(k)$. The coaction 
 $$\Delta : \Or({}_x \Pi^{dR}_0) \To \Or({}_0 \Pi^{dR}_0) \otimes_k \Or({}_x \Pi^{dR}_0)$$
 satisfies $\Delta F^1 \subset F^1 \otimes F^0 + F^0 \otimes F^1$. But $c$ annihilates $F^1\Or({}_x \Pi^{dR}_0)$, therefore 
 \begin{equation} \label{inproofidc} (\id \otimes c) \Delta:  F^1 \Or({}_x \Pi^{dR}_0) \To F^1\Or({}_0 \Pi^{dR}_0)\ .
 \end{equation} 
 By the torsor property, and strictness, $(\ref{inproofidc})$ is an isomorphism since it induces an isomorphism on the associated graded for the Hodge filtration. Furthermore, by coassociativity of $\Delta$, $(\ref{inproofidc})$ is compatible with the left-coaction by 
 $\Or({}_0 \Pi^{dR}_0)$ (on  points, $(\ref{inproofidc})$ corresponds to $g\mapsto cg : {}_0 \Pi^{dR}_0 \overset{\sim} {\rightarrow} {}_x \Pi^{dR}_0$, which respects right-multiplication by $ {}_0 \Pi^{dR}_0$ since $ga \mapsto cga$).  The algebras ${}_x H_0$ and ${}_0 H_0$ are  characterised by definition $\ref{Hdefn}$ and are  therefore  mapped isomorphically onto each other by $(\ref{inproofidc})$.  \end{proof}
 
  \begin{rem} In the case when $\Or({}_\bullet \Pi^{dR}_0)$  has Hodge numbers of type $(p,p)$ only, then  definition $\ref{Hdefn}$  gives  ${}_\bullet H_0 = \Or({}_\bullet \Pi^{dR}_0)$.
The natural map ${}_\bullet\Pi^{dR}_0 \rightarrow {}_\bullet\PF^{dR}_0$ is then an  isomorphism, and the  path ${}_x 1_0$ is an element of ${}_\bullet \Pi^{dR}_0(k)$.  
  Thus ${}_x1_0$ is the natural generalisation of the canonical de Rham path to the non mixed-Tate setting. 
   \end{rem}

 \subsection{Comparison with  Albanese manifolds}  \label{HainAlba}
 Since $F^1 \Or({}_{\bullet} \Pi_0^{dR})$ is the ideal of functions vanishing on $F^0 {}_{\bullet} \Pi_0^{dR}$   by $(\ref{FnPidef})$, we make the following definition.
 
 \begin{defn} For $\bullet \in \{x,0\}$, define an affine scheme over $k$
 $$F^0 \backslash {}_{\bullet} \Pi_0^{dR} = \Spec (W_0 +F^1 \Or ( {}_{\bullet} \Pi_0^{dR}) )\ .$$
 It follows from the earlier properties of the Hodge filtration that $W_0+F^1$ is indeed an algebra. It is not the group-theoretic quotient since $F^0$ is not a normal subgroup. Its complex points are the  manifolds defined by Hain \cite{HaAlb}. 
 \end{defn}

 The scheme $F^0 \backslash {}_{\bullet} \Pi_0^{dR}$ is equipped with a canonical point 
 given by the composition: 
  $$ \Or(F^0 \backslash {}_{\bullet} \Pi_0^{dR}) \To  \Or(F^0 \backslash {}_{\bullet} \Pi_0^{dR}) /F^1  = W_0 \Or ( {}_{\bullet} \Pi_0^{dR}) \cong  k $$
  since $W_0 \cap F^1 =0$. 
 The inclusions of spaces ${}_{\bullet}H_0 \subset  (W_0+F_1) \Or ( {}_{\bullet} \Pi_0^{dR}) \subset \Or ( {}_{\bullet} \Pi_0^{dR}) $, together with the previous remark, gives rise to the 
 following commutative diagram:
   $$
\begin{array}{ccccc}
F^0 {}_{\bullet} \Pi^{dR}_0   & \To   & \Spec(k)   & = &  \Spec(k)   \\
\downarrow   &   & \downarrow &    & \downarrow \\
 {}_{\bullet} \Pi^{dR}_0    &  \To   &   F^0 \backslash {}_{\bullet} \Pi_0^{dR}   & \To & {}_{\bullet} \PF^{dR}_0
\end{array}
$$  
 In particular,  the image of  any path in  $F^0 {}_{\bullet} \Pi^{dR}_0 (k)$ is the canonical path ${}_x1_0 \in  {}_{\bullet} \PF^{dR}_0 (k)$.
 
\begin{rem} Let $N^0$ denote the normaliser of $F^0$ in ${}_0 \Pi^{dR}_0$. Then the quotient $N^0 \backslash {}_0 \Pi^{dR}_0$ is the affine group scheme  whose affine ring is the largest Hopf subalgebra $A \subset (W_0 +F^1)(\Or(  {}_0 \Pi^{dR}_0))$. Therefore $\Delta A \subset A \otimes_k A$, and it follows that $A\subset {}_0 H_0$ by definition $\ref{Hdefn}$.  Our space ${}_0\PF^{dR}_0$ therefore satisfies
$$  F^0 \backslash {}_{0} \Pi_0^{dR}  \To {}_0\PF^{dR}_0 \To N^0 \backslash {}_{0} \Pi_0^{dR} $$
but in general neither morphism is an isomorphism.  Indeed, in example \ref{extensorsF1}  the algebra $A$ is the tensor algebra on $F^1V$ and consists of `totally holomorphic' iterated integrals. 
\end{rem} 

\begin{warning} \label{Caveat} A point $c \in F^0 {}_x\Pi^{dR}_0(k)$ defines, via $(\ref{inproofidc})$, an isomorphism  
$$F^0 \backslash {}_0 \Pi^{dR}_0  \overset{\sim}{\To} F^0 \backslash {}_x \Pi^{dR}_0$$
but this isomorphism is  not canonical: it depends on the choice of point $c$.   
On the other hand, we have shown that the schemes 
$ {}_0 \PF^{dR}_0$ and $ {}_x \PF^{dR}_0$ are canonically isomorphic.
\end{warning} 
      
   \subsection{General situation} Let  $\omega$ be the de Rham fiber functor. Let $n$ be a positive integer. 
  Then   
  $W_n \Or({}_x \Pi^{\dr}_0) $ 
   is the fiber at $x\in X(k)$ of an object in the category of unipotent vector bundles on $X$. There exists an open affine subset $U \subset X$, containing $0$,
   such that the underlying algebraic  vector bundle $W_n \Or(\Pi_0)$ (forgetting its connection) is trivial on $U$. Therefore there is a canonical isomorphism  for all  $x \in U(k)$:
   $$\xco:  W_n \Or({}_x \Pi^{\dr}_0) \cong  \Gamma(U,  W_n \Or(\Pi_0) ) \cong  W_n \Or({}_0 \Pi^{\dr}_0)\ . $$
    Now  choose a path $\xgo \in {}_x \Pi^{\dr'}_0(k')$ for  $x \in U(k)$. 
   The  map $a^x$   is dual to 
  \begin{eqnarray} 
a^x :  W_n \Or({}_0 \Pi^{\dr}_0)  & \To &  \Pe^{dR, \omega'}_{X} \otimes_k k'   \\
  \eta & \mapsto&  [ \Or(\pi_1(X,x,0)), {}_x \gamma_0, {}_xc_0^{-1} \eta ]^{dR, \omega'} \nonumber 
    \end{eqnarray}
   By applying the method  of \S\ref{sectLinAlg}, this can be used to construct `motivic' Coleman functions and hence detect points in $U(k)$ whenever the requisite inequality on dimensions 
   is satisfied. If $\omega'$ is the Betti fiber functor and $U(\C)$ is simply-connected, then the class of 
   ${}_x \gamma_0$ is uniquely defined. In the case when $\omega'$ is de Rham or crystalline, we can ask that the path  ${}_x \gamma_0$ be Frobenius-invariant, and again its class is canonical.   In the crystalline  case, the  $p$-adic period is given precisely by Coleman integration  by \cite{BesserTannaka}.
   
 This method should in principle work well enough to bound points on local affine charts $U$ of $X$. In order to eliminate this dependence on $U$, we use the canonical de Rham paths on the schemes $\PF$ as defined above. This gives:
    \begin{eqnarray}  \label{acan}
a^x :  {}_0 H_0   & \To &  \Pe^{dR, \omega'}_{X} \otimes_k k'   \\
  \eta & \mapsto&  [ \Or(\pi_1(X,x,0)), {}_x \gamma_0,  {}_x1_0^{-1} \eta ]^{dR, \omega'} \nonumber 
    \end{eqnarray}
  where ${}_0H_0 $ is canonically identified with ${}_x H_0 \subset \Or({}_x\Pi_0^{dR})$ via  lemma \ref{lemx10canisom}. 
 The point is that one only computes iterated integrals of forms in the subspace ${}_0 H_0 \subset \Or({}_0\Pi_0^{dR})$.
 When $\omega'$ is de Rham or crystalline, and ${}_x \gamma_0$ is the unique Frobenius-invariant path, then the  homomorphism $(\ref{acan})$
 depends on no choices. In this case its image is contained  in the subspace of Frobenius-invariant periods. 
 
  \subsection{Dimensions}    We now wish to apply   the construction of \S\ref{sect: det} in the situation where $\omega$ and $\omega'$ are de Rham and crystalline
  fiber functors, respectively.    Let  $\xgo \in {}_x\Pi_0^{\omega'}(k)$ be the unique Frobenius-invariant path.

Let us    choose  a path  $\xco \in F^0 {}_x\Pi^{dR}_0(k)$ in the following way. 
The action of $G^{dR}$ on the de Rham image of the Tate object $\Q(-1)$ defines a character $\chi: G^{dR} \rightarrow \G_m$. 
By a version of Levi's theorem, it is possible to choose a splitting of this map over $k$, which in turn induces a splitting of the weight filtration on $\omega_{dR}(M)$ for all objects $M$ in $\mathcal{H}$, in such a way that it is compatible with the Hodge filtration (in fact one can split the Hodge and weight filtrations simultaneously).  Thus for every object $M$ in $\mathcal{H}$ we have fixed a functorial isomorphism 
$M \cong \gr^W M$.
Applying this to the Ind object $\Or({}_x \Pi_0^{dR})$, we obtain a homomorphism by projecting onto weight zero:
$$\xco : \Or({}_x \Pi_0^{dR}) \overset{\sim}{\To} \gr^W \Or({}_x \Pi_0^{dR})  \To \gr^W_0 \Or({}_x \Pi_0^{dR})\cong k$$
for any point  $x$.  Since this is compatible with $F$, we have $ F^1   \Or({}_x \Pi_0^{dR}) \subset \ker \xco$ and therefore 
$\xco \in F^0 {}_x\Pi^{dR}_0(k)$. The  image  of  $\xco$ in ${}_x\PF^{dR}_0(k)$ is the canonical point ${}_x1_0$. 
  
  This data provides homomorphisms
  \begin{eqnarray}
  \alpha^x :  \Or({}_0\Pi^{dR}_0)  &  \To   & \Pe^{dR, dR}_{X}  \nonumber \\ 
  a^x :    \Or({}_0\Pi^{dR}_0)  &  \To   & \Pe^{dR, \omega'}_{X} \otimes_k k' \nonumber \\
  \beta^x: \Or({}_0\Pi^{\omega'}_0) & \To &  \Pe^{\omega', \omega'}_{X} \nonumber
  \end{eqnarray} 
  where $\beta^x$ is cocycle associated to $\xgo$.  Note that although $  a^x$ depends on  our choice of weight-splitting via $\xco$, we shall only consider its restriction  to the subspace ${}_0H_0$, which  is canonical.  In particular, the
  data $ (\alpha_x,  (a^x, \beta^x))$  defines an element of 
  $$Z^1_{\alg}( G^{dR},  \Pi^{dR}) \times Z^1_{r, \alg} ( P, {}_0\Pi_0^{\omega})(k')$$ where $P$ is the scheme of isomorphisms from $dR$ to $\omega'$. 
  Consider the subspace $Z$ of this product consisting of 
  bi-cocycles which respect the weight filtration, and furthermore such that 
  $\beta^x$ is Frobenius-invariant, and $\alpha^x$ is $\G_m$-invariant, where $\G_m$ is viewed as a subgroup of $\G^{dR}$ via our choice of 
  weight-splitting. The space $Z$  is a functor from commutative rings to sets, and is equipped with a natural transformation 
  \begin{equation} \label{mapZtoZFlastsection} Z \To Z_{F, \alg}^1(G^{\omega'}, {}_0\Pi_0^{\omega'})\ .
  \end{equation} 
  Let us assume that  $Z, Z_{F, \alg}^1(G^{\omega'}, {}_0\Pi_0^{\omega'})$ are schemes: for instance, if $W_n\Or(G^{\omega'})$ is finite dimensional for all $n$. The fibers of $(\ref{mapZtoZFlastsection})$ are
 empty or  torsors over  $( {}_0 \Pi^{dR}_0)^{\G_m}$,  by the discussion following  $(\ref{ZPtoZG})$. This is because   the image of $Z$ in $Z^1_{\alg}( G^{dR},  \Pi^{dR})$ is   $\G_m$-invariant by definition of $Z$.
 By the same argument as in lemma \ref{lemZ1toH1}, the group $ ( {}_0 \Pi^{dR}_0)^{\G_m}$ is trivial. Therefore we have shown that the fibers  of 
 $(\ref{mapZtoZFlastsection})$ are empty or $0$-dimensional.
 
   Since ${}_x \gamma_0$ is the 
  canonical Frobenius invariant path, we deduce by  $(\ref{Z1isH1Finv})$ that 
$$\dim Z \leq  \dim  H_F^1(G^{\omega'}, {}_0\Pi_0^{\omega'})\ .$$
The identical reasoning applies for any quotient $\Pi$ of ${}_0 \Pi_0$ in the category $\mathcal{H}$ (for instance, $\Pi =  {}_0 \Pi_0/W_m$). 
  Thus the methods of $\S \ref{sect: det}$ can be applied whenever
  \begin{equation} \label{dimineqfinal} \dim   \PF ^{dR} >  \dim  H_F^1(G^{\omega'}, \Pi^{\omega'})\ .
  \end{equation} 
  
   We obtain a function vanishing on integral points by applying \S \ref{sectLinAlg}
with $V= \Or(\PF^{dR})$,  $W =   \Pe_X^{\omega, \omega'}$ and $Z$ the scheme of bi-cocycles   described above.

 \begin{rem}
 The condition $(\ref{dimineqfinal})$ differs from  
  $  \dim \big( F^0 \backslash \Pi^{\omega} \big)  >   \dim  H^1(G^{\omega'}, \Pi^{\omega'}) $
 stated in  \cite{KimAlb}, conjecture 1. 
  \end{rem}

  \subsection{Case of $\Pro^1\backslash \{0,1,\infty\}$} The previous discussion can be applied in the case $X = \Pro^1 \backslash \{0,1,\infty\}$. We  retrieve an identical version of theorem $\ref{mainthm}$
 using Coleman integration instead of the single-valued $p$-adic periods. Because the fundamental group is mixed Tate, we have
 $F^0(\Pi^{\omega})=1$
  and ${}_0H_0 = \Or( {}_0\Pi^{dR}_0)$. Therefore ${}_x 1_0$ is the canonical de Rham path $(\ref{TateCanpath})$ and ${}_0 \Pi^{dR}_0 \cong {}_0 \PF_0^{dR}$.  Inequality 
  $(\ref{dimineqfinal})$ therefore reduces to 
  $$ \dim \Pi > \dim H_{F}^1(G_S^{\omega'}, \Pi^{\omega'})\ .$$
 Under this condition   we can apply the identical reasoning to \S\ref{sectskipped} to deduce the analogue of theorem $(\ref{mainthm})$ with single-valued $p$-adic integrals replaced with 
   Coleman integrals, since the dimension computations are identical.

  \subsection{} If one assumes a version of   Beilinson's conjectures  as mentioned above, then one checks that $(\ref{dimineqfinal})$
   will be satisfied for any curve of genus $g \geq 2$ over $\Q$, for essentially the same reasons as given in \cite{KimAlb}. 
    Roughly,  the size of the graded pieces of the unipotent fundamental group grow approximately of the order $(2g)^n$ (the difference between
    $\PF$ and $\Pi$ is essentially negligeable by example \ref{extensorsF1}), and the space of motivic extensions, once the weight is sufficiently large, is expected
   to be a lattice in the space of extensions of mixed Hodge structures. This  is easily calculated, for example, by  \cite{NotesPer} (6.4). 
    Under this assumption, it  would be interesting   to carry out a precise analysis of the dimensions, as in \S\ref{sectunitexample}, and work out some explicit versions of theorem $\ref{mainthm}$, namely, write down a determinant of $p$-adic iterated 
    integrals which vanishes conjecturally on all $N$-tuples of integral points, for suitable $N$.  
    
   It would be  interesting  to ascertain whether one can use complex periods to detect integral  points on curves of higher genus, or whether this is ruled out.   For example, one might optimistically hope to take $\omega=\omega'=dR$, but this time apply the single-valued iterated integrals as defined in  \cite{NotesMot}, \S4.1, \S8.3.     Finally, the method above can also be extended to  divisors in  a very similar manner to \S \ref{sectVirtual}.

\bibliographystyle{plain}
\bibliography{main}

\end{document}